\documentclass[12pt,a4paper]{amsart}
\usepackage{graphicx,latexsym,amsfonts,amsmath,amssymb,rotating,txfonts,mathrsfs,enumerate, mathtools,multirow}
\usepackage{tcolorbox}
\usepackage[latin1]{inputenc}
\usepackage{epic}
\usepackage{curves}
\usepackage{pdfsync}
\usepackage{ytableau}
\usepackage{tikz}
\usepackage{calrsfs}
\usepackage{color, colortbl}
\usepackage{array,makecell}
\usepackage{tikz}

\usetikzlibrary{shapes.geometric, arrows}
\tikzstyle{boxes} = [rectangle, minimum width=1cm, minimum height=0.5cm, text centered, draw=black, fill=white]
\tikzstyle{nocontr} = [rectangle, minimum width=1cm, minimum height=0.5cm, text centered, draw=black, fill=yellow]
\tikzstyle{contr} = [rectangle, minimum width=0.6cm, minimum height=0.2cm, text centered, draw=black, fill=red!30]
\tikzstyle{arrow} = [thick,->,>=stealth]

\input xy
\xyoption{all}

\newcommand{\Proj}{{\rm Proj}}

\newcommand{\Spec}{{\rm Spec}}

\newcommand{\Hom}{ \,{\rm Hom} \,}
\newcommand{\Sym}{ \,{\rm Sym} \,}

\newcommand{\im}{ \,{\rm Im} \,}

{\bf}{\rm}
\newtheorem{theorem}{Theorem}[section]
\newtheorem*{theorem*}{Theorem}
\newtheorem*{maintheorem}{Main Theorem}
\newtheorem{proposition}[theorem]{Proposition}
\newtheorem{corollary}[theorem]{Corollary}
\newtheorem{lemma}[theorem]{Lemma}
\newtheorem{definition}[theorem]{Definition}
\newtheorem{remark}{Remark}

{\bf}{\it}

\newcommand{\RR}{{\mathbb R }}
\newcommand{\AAA}{{\mathbb A }}
\newcommand{\CC}{{\mathbb C }}

\newcommand{\ZZ}{{\mathbb Z }}
\newcommand{\PP}{ {\mathbb P }}
\newcommand{\QQ}{{\mathbb Q }}

\newcommand{\GG}{{\mathbb G }}
\newcommand{\ff}{{\mathbf f }}
\newcommand{\bgg}{{\mathbf g }}
\newcommand{\kk}{{\mathbf k }}
\newcommand{\tPP}{ {\tilde{\mathbb P }}}

\newcommand{\calc}{\mathcal{C}}

\newcommand{\calo}{\mathcal{O}}

\newcommand{\cale}{\mathcal{E}}

\newcommand{\calt}{\mathcal{T}}
\newcommand{\call}{\mathcal{L}}
\newcommand{\calz}{\mathcal{Z}}
\newcommand{\calm}{\mathcal{M}}

\newcommand{\lieb}{{\mathfrak b}}
\newcommand{\liek}{{\mathfrak k}}
\newcommand{\lieu}{{\mathfrak u}}

\newcommand{\liet}{{\mathfrak t}}
\newcommand{\liev}{{\mathfrak v}}

\DeclareMathOperator{\stab}{Stab}
\DeclareMathOperator{\Stab}{Stab}

\newcommand{\env}{\!
\mathbin{\text{\rotatebox[origin=c]{70}{\scalebox{1.2}{$\approx$}}}} \!}

\newcommand{\weight}{\omega}
\newcommand{\wt}{\gamma}

\newcommand{\hU}{\hat{U}}

\newcommand{\tf}{\tilde{f}}

\newcommand{\reg}{\mathrm{reg}}

\newcommand{\Euler}{\mathrm{Euler}}

\newcommand{\mapnk}{J(n,m)}

\newcommand{\gnk}{\mathrm{Diff}_n \times \mathrm{Diff}_m}

\newcommand{\Tp}{\mathrm{Tp}}

\newcommand{\mult}{\mathrm{mult}}
\newcommand{\epd}[1]{\mathrm{eP}[#1]}
\newcommand{\mdeg}[1]{\mathrm{mdeg}[#1]}
\newcommand{\epdgl}[1]{\mathrm{eP}_{\GL(n) \times \GL(m)}[#1]}
\newcommand{\emu}{\mathrm{emult}}

\newcommand{\dist}{{\mathrm{dst}}}
\newcommand{\lead}{\mathrm{lead}}
\newcommand{\Span}{\mathrm{Span}}
\newcommand{\Bipi}{{\boldsymbol{\Pi}_k}}
\newcommand{\bipi}{{\boldsymbol{\pi}}}
\newcommand{\sg}[1]{\mathcal{S}_{#1}}
\newcommand{\res}{\operatornamewithlimits{Res}}

\newcommand{\ires}{\res_{z_1=\infty}\res_{z_{2}=\infty}\dots\res_{z_k=\infty}}
\newcommand{\sires}{\res_{\mathbf{z}=\infty}}

\newcommand{\dbz}{\,d\mathbf{z}}

\newcommand{\coeff}{\mathrm{coeff}}

\newcommand{\symdot}{\mathrm{Sym}^{\le k}\CC^n}

\newcommand{\grass}{\mathrm{Grass}}

\newcommand{\flag}{\mathrm{Flag}}
\newcommand{\diff}{\mathrm{Diff}}
\newcommand{\TT}{\mathrm{T}}

\newcommand{\kt}{{K}}

\newcommand{\jetk}[2]{J_{k}({#1},{#2})}

\newcommand{\jetreg}[2]{J_{k}^{\mathrm{reg}}({#1},{#2})}

\newcommand{\jetnondeg}[2]{J_{k}^{\mathrm{nondeg}}({#1},{#2})}
\newcommand{\nondeg}{\mathrm{nondeg}}

\newcommand{\tc}{\hat T}

\newcommand{\GL}{\mathrm{GL}}
\newcommand{\sym}{\mathrm{Sym}}
\newcommand{\mon}{\mathrm{mon}}
\newcommand{\red}{\mathrm{red}}

 \newcommand{\lieg}{{\mathfrak g}}

\newcommand{\Hilb}{\mathrm{Hilb}}
\newcommand{\CHilb}{\mathrm{CHilb}}

\newcommand{\Thom}{\mathrm{Thom}}

\newcommand{\blow}{\mathrm{Bl}}
\newcommand{\C}{\mathbb{C}}
\newcommand{\tZmin}{\tilde{Z}_{\min}}
\newcommand{\vv}{v}

\newcommand{\bz}{\mathbf{z}}

\newcommand{\bv}{\mathbf{v}}
\newcommand{\bt}{\mathbf{t}}
\newcommand{\bu}{\mathbf{u}}

\newcommand{\bx}{\mathbf{x}}

\def\a{\alpha}
\def\b{\beta}
\def\g{\gamma}

\def\l{\lambda}

\def\x{\xi}

\def\s{\sigma}

\def\vp{\varphi}

\title{Non-reductive geometric invariant theory and Thom polynomials} 

\author{Gergely B\'erczi}
\address{Department of Mathematics, Aarhus University} \email{gergely.berczi@math.au.dk}
\thanks{This work was partially supported by the Aarhus University Research Foundation grant AUFF 29289.}
\dedicatory{To Eszter, Dani, Flora and Epsilon}

\date{}
\begin{document}

\begin{abstract}
We combine recently developed intersection theory for non-reductive geometric invariant theoretic quotients with equivariant localisation to prove a formula for Thom polynomials of Morin singularities. These formulas use only toric combinatorics of certain partition polyhedra, and our new approach circumvents the poorly understood Borel geometry of existing models. 
\end{abstract}

\maketitle
\tableofcontents

\section{Introduction} 
We combine the results of \cite{BDHK0} with a blow-up process to construct a new non-reductive GIT type model for the moduli space of $k$-jets, which parametrises holomorphic $k$-jets of map germs $(\CC,0) \to (\CC^n,0)$ modulo polynomial reparametrisations of $\CC$ at the origin. Intersection theory developed in \cite{bkcoh} for non-reductive GIT quotients accompanied with equivariant localisation and a residue vanishing theorem leads us to a closed iterated residue formula for Thom polynomials of Morin singularities. Our formula builds only on the toric geometry of partition polyhedra, and we avoid the use of unknown Borel geometry appeared in the earlier works \cite{kazarian3,bsz,rimanyifeher}.

Global singularity theory studies the topology of holomorphic maps between complex manifolds. A fundamental question is how to describe the topological locus in the source space where a generic map has a given type of singularity. A more precise formulation of this problem goes back to Thom \cite{thom} and reads as follows. Let $f: N \to M$ be a map between two complex manifolds, $N$ and $M$, of dimensions $n\le m$. We say that $p\in N$ is a {\em singular} point of $f$ if the rank of the differential $df_p:T_pN\to T_{f(p)}M$ is less than $n$. In order to introduce a finer classification of singular points, choose local coordinates near $p\in N$ and $f(p)\in M$, and consider the resulting
map-germ $f_p:(\CC^n,0)\to(\CC^m,0)$, which may be thought of as a
sequence of $m$ power series in $n$ variables without constant terms. Let $J(n,m)$ denote the space of such map-germs endowed with the action of the group $\gnk$ of infinitesimal local coordinate changes. The $\gnk$-orbits or, more generally, $\gnk$-invariant subsets
$O\subset \mapnk$ are called {\em singularities}. For a singularity $O$ and
holomorphic $f:N\to M$, we can define the set
\[\Sigma_O(f) = \{p\in N;\; f_p\in O\}, \] which is independent of any
coordinate choices. $N$. Assuming $N$
is compact and $f$ is sufficiently generic, $\Sigma_O(f)$ is an analytic subvariety of $N$ giving a Poincar\'e dual class $[\Sigma_O(f)]\in H^*(N,\mathbb{Z})$. This problem was first studied by Ren\'e Thom (cf. \cite{thom, haef}) in the category of smooth varieties and smooth maps; in this case cohomology with $\ZZ/2\ZZ$-coefficients is used. Thom discovered
that to every singularity $O$ one can associate a bivariant characteristic class $\tau_O$, which, when evaluated on the pair $(TN,f^*TM)$ produces the Poincar\'e dual class $[\Sigma_O(f)]$.  One of the consequences of this result is that the class $[\Sigma_O(f)]$ depends only on the homotopy class of $f$.

The structure of the $\gnk$-action on $\mapnk$ is rather complicated; even the parametrization of the orbits is difficult.  There is, however, a simple invariant on the space of orbits: to each map-germ $f=(f_1,\ldots, f_m):(\C^n,0)\to(\C^m,0)$, we can associate the finite-dimensional nilpotent algebra $A_f=(x_1,\ldots, x_n)/(f_1,\ldots, f_m)$ defined as the quotient of the algebra of power series with no constant term by the ideal generated by the pull-back subalgebra $f^*(y_1\ldots,y_m)$. This algebra $A_{f}$ is trivial if
the map-germ $f$ is nonsingular, and it does not change along a $\gnk$-orbit. Hence, for a fixed nilpotent algebra $A$ with $n$ generators one can associate the invariant subset 
$\Sigma_A=\{f \in J(n,m):A_f \simeq A\}$,
the set of jets with local algebra isomorphic to $A$. The {\em Thom principle} in this holomorphic setting (cf \cite{kazariantalk,rimanyifeher2,bsz}) for $\calo=\Sigma_A$  tells that there is a jet order $k$ depending only on $A$ such that  
\[ [\Sigma_A(f)]=\epd{\Sigma_A,J_k(n,m)}(f^*TN,TM),\]
is obtained by substituting the Chern roots of $TN,f^*TM$ into the $\GL(n) \times \GL(m)$-equivariant dual of $\Sigma_A$ in $J_k(n,m)$. This equivariant dual sits in the ring $\CC[x_1,\ldots, x_n,y_1,\ldots, y_m]^{S_n\times S_m}$ of bi-symmetric polynomials on $n+m$ variables. Damon and Ronga \cite{damon,sigma11} proved that $[\Sigma_A(f)]$ depends on $TN,f^*TM$ only through the Chern classes of the difference bundle $TN-f^*TM$, and hence $[\Sigma_A(f)]=\Tp_A(c_1,c_2,\ldots)$ is a polynomial in these classes, which we call the {\em Thom polynomial of $A$}. Calculation of these polynomials remains a major open problem ever since then, see e.g \cite{arnold, rimanyi, kazariantalk, rimanyifeher3}.

In this paper we study the Morin case, i.e  when $A_k=\CC[t]/t^{k+1}$ for some $k$. The problem of calculating $\Tp_k=\Tp_{A_k}$ goes back to Thom \cite{thom,teissier}.  The case $k=1$ is the classical formula of Porteous \cite{porteous}: $\Tp_1=c_{m-n+1}$. The Thom polynomial in the $k=2$ case was computed by Ronga in \cite{sigma11}. An explicit formula for $\Tp_3$ was proposed in \cite{bfr}, and P. Pragacz has given a sketch of its proof \cite{prag}. Finally, using his method of restriction equations, Rim\'anyi \cite{rimanyi} was able to treat the $n=k$ case, and computed $\Tp_k$ for $k\le 8$ in the equidimensional $n=m$ case (cf. \cite{gaffney} for the case $k=4$).

In \cite{bsz} we introduced a new approach and used equivariant localisation to prove iterated residue formulas for Morin singularities for any $n,k$. Let $J_k(1,n)$ denote the $k$-jets of holomorphic map germs $(\mathbb{C},0) \to (\mathbb{C}^n,0)$. These jets can be reparametrised by the polynomial reparametrisation group $\diff_k=J_k^{reg}(1,1)$ formed by jets with nonzero first derivative. The quotient $J_k(1,n)/\diff_k$ is quasi-projective, it is the moduli of invariant jets of order $k$ in $\CC^n$. 
A key observation of \cite{bsz} is that (some open part of) $\Sigma_k(n,m)$ fibers over the non-reductive quotient $J_k(1,n)/\diff_k$ with affine fibers. We show in \cite{bsz} that the curvilinear component of the Hilbert scheme of $k$ points on $\CC^n$ provides a natural compactification of this non-reductive quotient, and \textit{the main message of \cite{bsz} is that the Thom polynomials are certain equivariant intersection numbers on the curvilinear Hilbert scheme}. 

This formula, however, contains an unknown ingredient, the $Q_k$ polynomial, which is the equivariant dual of a certain Borel orbit in a $\GL(n)$-module. Our knowledge about $Q_k$ is very limited. In \cite{kazarian3}, Kazarian reinterprets the defining equations of $Q_k$ as the ideal of the Morin algebra sitting in a smooth ambient space what he calls the {\em nonassociative Hilbert scheme}. These equations fully determine $Q_k$---however the description of them is out of reach, and a major problem. The problem is that besides the {\em associativity equations} (which do not cut out even the right dimensional variety--there there are other {\em exotic equations}, which are determined by the Borel geometry of the nonassociative Hilbert scheme. 

In this paper we develop an approach which avoids working with $Q_k$ and Borel orbits, and reduces the problem to toric geometry. We replace the curvilinear Hilbert scheme with a birational model, namely, the non-reductive GIT quotient $\overline{J_k(1,n)}/\diff_k$ for some proper projective completion $\overline{J_k(1,n)}$ of $J_k(1,n)$. This compactification will be a result of a blow-up procedure fully determined by the combinatorics of certain partition polyhedra. We combine intersection theory developed in \cite{bkcoh} for non-reductive GIT quotients with equivariant localisation on the blown-up space, and in particular, we use the non-reductive Jeffrey-Kirwan type integration formula proved in \cite{bkcoh}.  We prove a residue vanishing theorem, which leads us to closed iterated residue formulas for Thom polynomials of Morin singularities, detailed in the next section. 

In particular, this paper provides a combinatorial (toric) approach to a {\em partial fraction decomposition} of the rational expression $\frac{Q_k}{\prod_{1\le i+j \le k} (z_i+z_j-z_k)}$ in the formula of B\'erczi-Szenes \cite{bsz}, which might serve as a starting point towards Rim\'anyi's conjecture on Chern positivity of Thom polynomials \cite{rimanyi,pragaczweber}.  

As a final remark, we note that although this paper focuses only on Morin singularities, our approach extends to compute Thom polynomials for other contact singularity classes (see \cite{rimanyifeher,rimanyi}).

\section{The strategy and the result}\label{sec:strategy}

Fix the parameters $k\le n\le m$. Let $J_k(n,m)$ denote the $k$-jets of holomorphic map germs $(\mathbb{C}^n,0) \to (\mathbb{C}^m,0)$, and  $\Sigma_k(n,m)$
be the set of jets with local algebra $t\mathbb{C}[t]/t^{k+1}$. Its closure, $\overline{\Sigma}_k(n,m)$ is a $\GL(n) \times \GL(m)$-invariant, singular subvariety of the affine space $J_k(n,m)$. The Thom principle (see e.g (see \cite{rimanyi,kazarian3,bsz} says that the  polynomial of $A_k$ singularities is given by the $\GL(n) \times \GL(m)$-equivariant dual of $\Sigma_k(n,m)$ in $J_k(n,m)$ after substituting the Chern roots of $TN$ and $f^*TM$ into the weights:
\[\Tp_k^{n,m}=\epdgl{\Sigma_k(n,m),J_k(n,m)}(TN,f^*TM)\] 
Gaffney \cite{gaffney} observed that a generic germ $\Psi \in \Sigma_k$ has a test curve $\gamma \in J_k^{reg}(1,n)$ ({\em reg} refers to nonvanishing linear part) such that $\Psi \circ \gamma=0$ in $J_k(n,m)$. Here $\Psi=\Psi^1+\ldots +\Psi^k$ is generic if its linear part is generic, that is, the rank of $\Psi^1$ is $1$. We denote the set of these jets by $\Sigma_k^{0}(n,m)$. The test curve is unique up to polynomial reparametrizations of $(\mathbb{C},0)$; that is, up to the action of the reparametrisation group $\diff_k=J_k^{reg}(1,1)$. Therefore $\Sigma_k^{0}(n,m)$ fibers over the quasi-projective quotient $J_k^{reg}(1,n)/\diff_k$, which we call the moduli space of $k$-jets in $\CC^n$. The fibers are linear subspaces in $J_k(n,m)$ of codimension $km$.

In \cite{bsz} we construct an embedding $\phi: J_k^{reg}(1,n)/\diff_k \hookrightarrow \mathrm{Grass}(k,J_k(n,1)^*)$
and we work with the the closure $\overline{\mathrm{im}(\phi)}$ as a natural projective completion of the quotient. The key property of this compactification is that the linear fibration $\Sigma_k^{0}(n,m)$ over $J_k^{reg}(1,n)/\diff_k$ extends to a bundle $V$ over $\overline{\mathrm{im}(\phi)}$ which sits in the trivial bundle $V \subset J_k(n,m)\times \overline{\mathrm{im}(\phi)}$. We develop equivariant localisation on $\overline{\mathrm{im}(\phi)}$ to come up with an iterated residue formula for integrals over $\overline{\mathrm{im}(\phi)}$, and we present Thom polynomials of $A_k$ singularities as an equivariant intersection number, namely,
\[\epd{\Sigma_k(n,m),J_k(n,m)}=\int_{\overline{\mathrm{im}(\phi)}}\Thom(V \subset J_k(n,m))\]
is the integral of the Thom form of $V$ in $J_k(n,m)\times \overline{\mathrm{im}(\phi)}$, whose restriction to each linear fiber $V_f$ is the Euler class of the normal bundle  of $V_f$ in $J_k(n,m)$.  

In \cite{b0} we observed that $k$-jets of smooth curves at the origin on $\CC^n$ correspond to so-called curvilinear subschemes on $\CC^n$. These are points of the Hilbert scheme $\Hilb^{k+1}(\CC^n)$ which are supported at the origin and isomorphic to the algebra $t\CC[t]/t^{k+1}$. The closure of this locus in the Hilbert scheme is a punctual component, called the curvilinear component, denoted by $\CHilb^{k+1}_0(\CC^n)$. Hence this component is a natural projective compactification of the moduli space of jets, and in \cite{b0} we showed that 
$\CHilb^{k+1}_0(\CC^n)=\overline{\mathrm{im}(\phi)} \subset \grass_k(J_k(n,1)^*)$. Let $\cale$ denote the tautological rank $k$ bundle, and $\CC^m$ the trivial $\GL(m)$-equivariant bundle over $\overline{\mathrm{im}(\phi)}$. Then $\cale \otimes \CC^m$ is $\GL(n) \times \GL(m)$-equivariant, and in \cite{bsz} we show that 
$\Thom(V \subset J_k(n,m))=\mathrm{Euler}(\cale \otimes \CC^m)$ is the equivariant Euler class, hence
\[\epd{\Sigma_k(n,m),J_k(n,m)}=\int_{\CHilb^{k+1}_0(\CC^n)}\Euler(\cale \otimes \CC^m).\]

In this paper we substitute the curvilinear Hilbert scheme with a birational model, namely, the non-reductive geometric invariant theoretic (NRGIT) quotient $\mathrm{Jet}_k/\!/\diff_k$ for some suitable projective completion $\mathrm{Jet}_k=\overline{J_k(1,n)}$ of $J_k(1,n)$. 
For our purposes 'well-behaved' means a completion $\overline{J_k(1,n)}$ such that our $\diff_k$-invariant morphism $\phi: J_k^{reg}(1,n) \to \CHilb^{k+1}(\CC^n)$ extends to the NRGIT semistable locus in $\overline{J_k(1,n)}$. This morphism guarantees that integration over the curvilinear Hilbert scheme can be pulled back to the NRGIT quotient, and we can work with the much better behaved NRGIT model. 
We start with the simple compactification $\PP(\CC \oplus J_k(n,m))$, and we construct a $\diff_k$-equivariant blow-up $\mathrm{Jet}_k$ such that the $\diff_k$-semistable locus $\mathrm{Jet}_k^{ss}$ admits a $\diff_k$-invariant morphism $\mathrm{Jet}_k^{ss}\to \CHilb^{k+1}(\CC^n)$: 
\begin{equation}\label{strategy}
\xymatrix{**[r] \mathrm{Jet}_k  \ar[d]^\pi & **[l] \mathrm{Jet}_k^{ss} \ar@{_{(}->}[l] \ar[rd]^-{\tilde{\phi}} & & \\
 \PP(\CC \oplus J_k(1,n)) \ar@{.>}[rr]^-{\phi} & & \CHilb^{k+1}_0(\CC^n) \ar@{^{(}->}[r] &  \grass_k(J_k(1,n)^*) } 
\end{equation}

With such a diagram in hand we will be in good position: the NRGIT quotient is categorical,  hence $\tilde{\phi}$ induces a morphism $\varphi: \mathrm{Jet}_k/\!/\diff_k \to \CHilb^{k+1}_0(\CC^n)$. The Thom polynomials and multisingularity classes will be expressed as an integral on the NRGIT quotient $\mathrm{Jet}_k/\!/\diff_k$:
\[\epd{\Sigma_k(n,m),J_k(n,m)}=\int_{\mathrm{Jet}_k/\!/\diff_k}\varphi^*\Euler(\cale \otimes \CC^m).\]
Our equivariant localisation strategy can be summarised as follows:

1) The non-reductive Jeffrey-Kirwan residue formula developed in \cite{bkcoh} can be applied to get
\[\epd{\Sigma_k(n,m),J_k(n,m)}=\res_{z=\infty} \sum_{F\in (\mathrm{Jet}_k^{\CC^*})_{\min}}\int_F\frac{(k-1)!z^{k-1}i_F^*\varphi^*\Euler(\cale \otimes \CC^m)}{\mathrm{Euler}(\mathcal{N}_{F})(z)} dz\]
Here the sum is taken over those $\CC^* \subset \diff_k$-fixed point components on the blown up $\mathrm{Jet}_k$ (these will be isolated fixed points in our construction) which have minimal $\CC^*$-weight. 

2) We will perform blow-ups at smooth centers, invariant under the $\diff_k$ action. Starting from the projective space $\PP(\CC \oplus J_k(n,m))$, this way we will end up in having a smooth iterated blow-up $\mathrm{Jet}_k$ with an induced $\diff_k$ action. In order to follow this plan, we need to start with an {\em initial master blow-up}: instead of  $\PP(\CC \oplus J_k(n,m))$, we start the blow-up process with a fiberwise compactification $\tilde{\PP}$ of the fibration over the complete flag $J_k(1,n)^{nondeg} \to J_k(1,n)^{nondeg}/B=\flag_k(\CC^n)$. This space is a partial blow-up of $\PP(\CC \oplus J_k(n,m))$, and the crucial advantage of using it is that  we can apply equivariant localisation on $\flag_k(\CC^n)$ with respect to the $\GL(n)$ action (which commutes with the $\CC^*$ action coming from $\diff_k$) and rewrite the localisation formula as an iterated residue w.r.t $k$ variables $z_1,\ldots, z_k$. 

3) Let $X_0=\tilde{\PP}$ and let $X_i$ denote the space we get after performing $i$ blow-ups. The center of each blow-up will be given by an ideal of the form $(\b_1,\ldots, \b_r)$ where the $\b_j$'s are some local affine coordinates on some $X_i$. Hence the blow-up process will be described by a \textit{rooted tree}, whose nodes are labeled by \textit{clusters of affine coordinates} on the intermediate spaces, and the \textit{edges are labeled by variables} indicating the affine chart we pick. Fixed points on $\mathrm{Jet}_k$ correspond to \textit{paths on the tree from the root to some leave} (leave=node of valency $1$). 

4) Due to a residue vansihing theorem proved in \S \ref{subsec:residuevanishing}, if a leaf (i.e torus fixed point on $\mathrm{Jet}_k$) has nonzero contribution to the residue formula, then it must be mapped by $\tilde{\phi}$ to the distinguished subset $\Lambda_k = \{\xi \in \CHilb^{k+1}(\CC^n): \calo_\xi \simeq (x_1,\ldots, x_k)/(x_1,\ldots, x_{k})^2\}$. This is a very strong condition, which allows us to drop most of the leaves from the blow-up tree. The yellow labels in the blow-up diagrams in our examples in \S \ref{sec:examples} have zero contribution, and only the red leaves have nonzero contribution.  

Before stating the main result, we fix some terminology. We will refer to fixed points on $\mathrm{Jet}_k$ as leaves of the blow-up tree. There is a unique path from the root to the leaf $L$, and such a path has the form 
\[\xymatrix{\calc_1 \ar[r]^{\b_1} & \calc_2 \ar[r]^{\b_2} & \ldots \ar[r]^{\b_{r-1}} & \calc_{\max}}\]
where $\calc_i$ are clusters of affine coordinates, that is, $\calc_i \subset \mathbf{B}$ for a fixed set of coordinates $\mathbf{B}$. The edges are labeled by variables such that $\b_i \in \calc_i$ for all $i$. The set of associated paths is denoted my $\Gamma_k$. Each variable of $\mathbf{B}$ at each level is endowed with a $T^k \times \CC^*$ weight in this path; the weight of $\b \in \mathbf{B}$ after $i$ blow-ups is denoted by $\wt^L_i(\b)$. The simple rule of forming these weights from some initial weights $\wt_1(\b)$ is the following:
\begin{equation}\label{changeofweights}
\wt^L_{i+1}(\b)=\begin{cases} \wt^L_{i}(\b) & \text{ if } \b=\b_i \text{ or } \b \notin \calc_{i+1}\\
                                        \wt^L_{i}(\b)-\wt^L_i(\b_i) & \text{ if } \b \in \calc_{i+1}\setminus \{\b_i\} \\
                                               \end{cases}
\end{equation}                                                                                
We will use the shorthand notation $\wt^L(\b)=\wt^L_{\max}(\b)$. We are ready to formulate the main theorem, which we will restate as Theorem \ref{maintheorem} at the end of this paper. 
\begin{maintheorem} For arbitrary integers $k\ll n \le m$ the Thom polynomial for the $A_k$-singularity with
$n$-dimensional source space and $m$-dimensional target space is given by the following iterated residue formula:
\[\Tp_k^{n,m}=\res_{z_1=\infty}\ldots \res_{z_k=\infty}\cdot \sum_{L \in \tilde{\phi}^{-1}(\bgg)} \frac{(k-1)!z^{k-1} (z_1\ldots z_k)^{m-n}\prod_{i<j}(z_i-z_j)}{\prod_{\b \in \mathbf{B}} \wt^L(\b)} \prod_{i=1}^k c_{TM-TN}(1/z_i)d\bz\]
where we sum over those leaves of the blow-up tree, which are mapped to the distinguished point $\bgg=[e_1\wedge \ldots \wedge e_k]\in \PP(\wedge^k \symdot)$ under the morphism $\tilde{\phi}$. 
\end{maintheorem}

\noindent \textbf{Examples} For $k=3$ we get back the formula of \cite{bsz}:
\[\Tp_3=\res_{z_1,z_2,z_3=\infty}\frac{(z_1-z_2)(z_1-z_3)(z_2-z_3)}{(2z_1-z_2)(2z_1-z_3)(z_1+z_2-z_3)}\prod_{i=1}^3 c_{TM-TN}(1/z_i)\]
For the blow-up tree and the details of the computation see \S \ref{subsec:k=3}. However, for $k=4$ we already get something different:
{\footnotesize \begin{align*}
\Tp_4 =\res_{\bz}\frac{(z_1z_2z_3z_4)^{N-n}\prod_{i<j}(z_i-z_j)\prod_{i=1}^4 c_{TM-TN}(1/z_i)d\mathbf{z}}{(z_2+z_3-z_1-z_4)(z_1+z_2-z_4)(2z_1-z_3)(2z_1-z_4)}\cdot \left(\frac{1}{(2z_1-z_2)(z_1+z_2-z_3)}+\frac{1}{(z_4-z_1-z_3)(2z_2-z_4)}\right) 
\end{align*}}
The details of the computation including the blow-up tree and the weight table is given in \S \ref{subsec:k=4}. Recall formula from \cite{bsz}:
{\footnotesize \begin{align*}\label{Tp4bsz}
\Tp_4^{BSZ}& =\res_{z_1,\ldots,z_4=\infty}\frac{(2z_1+z_2-z_4)(z_1z_2z_3z_4)^{N-n}\prod_{i<j}(z_i-z_j)d\bz}{(2z_1-z_2)(2z_1-z_3)(z_1+z_2-z_3)(z_1+z_2-z_4)(2z_1-z_4)(z_1+z_3-z_4)(2z_2-z_4)}\cdot \prod_{i=1}^4 c_{TM-TN}(1/z_i)
\end{align*}}
The two formula give the same result, and in particular the new formula provides a nontrivial partial fraction decomposition  
\begin{multline}\nonumber \frac{1}{(z_2+z_3-z_1-z_4)(2z_1-z_2)(z_1+z_2-z_3)}+\frac{1}{(z_2+z_3-z_1-z_4)(z_4-z_1-z_3)(2z_2-z_4)}=\\
=\frac{2z_1+z_2-z_4}{(2z_1-z_2)(z_1+z_2-z_3)(z_1+z_3-z_4)(2z_2-z_4)}
\end{multline}
For $k=5$ the blow-up tree has $15$ leaves mapped to $\Xi_5$, and hence in the residue formula we have the sum of $15$ rational expression. Two of these have 0 contribution, and  in short, our formula replaces the rational expression 
\[\frac{Q_5(z_1,\ldots, z_5)}{\prod_{1\le i+j \le l\le k} (z_i+z_j-z_l)}\]
where 
\[\QQ_5(z_1,z_2,z_3,z_4,z_5)=
(2z_1+z_2-z_5)(2z_1^2+3z_1z_2-2z_1z_5+2z_2z_3-z_2z_4-z_2z_5-z_3z_4+z_4z_5).
\]
which appears in \cite{bsz} with the sum of $13$ rational expressions. This might look as a bad deal, but the real benefit comes for $k\ge 6$; for $k=6$ we know that $Q_6$ is a is a degree-$7$ polynomial in $6$ variables, with more than $1000$ terms, but our blow-up tree has $\sim 60$ leaves. And for $k\ge 7$ there is no computing capacity to cope with the relations and hence $Q_k$ is unknown, whereas our algorithm is purely toric, working entirely with monomial ideals. 

\section{Hilbert schemes and tautological integrals}

\subsection{Tautological integrals} Let $X$ be a smooth projective variety of dimension $n$ and let $F$ be a rank $r$ bundle (loc. free sheaf) on $X$.  
Let 
\[\Hilb^k(X)=\{\xi \subset X:\dim(\xi)=0,\mathrm{length}(\xi)=\dim H^0(\xi,\calo_\xi)=k\}\]
denote the Hilbert scheme of $k$ points on $X$ parametrizing length $k$ subschemes of $X$ and $F^{[k]}$ the corresponding rank $rk$ bundle on $\Hilb^k(X)$ whose fibre over $\xi \in \Hilb^k(X)$ is $F \otimes \calo_{\xi}=H^0(\xi,F|_\xi)$. 
  
 Equivalently, $F^{[k]}=q_*p^*(F)$ where $p,q$ denote the projections from the universal family of subschemes $\mathcal{U}$ to $X$ and $\Hilb^k(X)$ respectively: 
 \[\Hilb^k(X)\times X \supset \xymatrix{\mathcal{U} \ar[r]^-q \ar[d]^-p & \Hilb^k(X) \\ X & }.\]

In this paper we work with singular homology and cohomology with rational coefficients. For a smooth manifold $X$ the degree of a class $\eta \in H_*(X)$ means its pushforward to $H_*(pt)=\QQ$. By choosing  $\alpha_\eta \in \Omega^{top}(X)$,  a closed compactly supported differential form representing the cohomology class $\eta$ this degree is equal to the integral  
\[\eta \cap [X]=\int_{X} \alpha_\eta.\]

Let $\calz \subset \Hilb^k(X)$ be a geometric subset with closure $\overline{\calz}$ and $M(c_1,\ldots, c_{rk})$ be a monomial in the Chern classes $c_i=c_i(F^{[k]})$ of weighted degree equal to $\dim \overline{\calz}$ where the weight of $c_i$ is $i$. By choosing $\alpha_M \in \Omega^*(\Hilb^k(X))$, a closed compactly supported differential form representing the cohomology class of $M(c_1,\ldots, c_{rk})$, the degree  
\begin{equation}\label{integral}
[\overline{\calz}]\cap M(c_1,\ldots, c_{rk})=\int_{\overline{\calz}} \alpha_M
\end{equation}
is called a tautological integral of $F^{[k]}$. 

\subsection{Curvilinear Hilbert schemes}

Let $X$ be a smooth projective variety of dimension $n$ and for $p\in X$ let 
\[\Hilb^k_p(X)=\{ \xi \in \Hilb^k(X): \mathrm{supp}(\xi)=p\}\]
denote the punctual Hilbert scheme consisiting of subschemes supported at $p$. If $\rho: \Hilb^k(X) \to S^kX$, $\xi \mapsto \sum_{p\in X}\mathrm{length}(\calo_{\xi,p})p$ denotes the Hilbert-Chow morphism then $\Hilb^k_p(X)=\rho^{-1}(kp)$.
\begin{definition}
A subscheme $\xi \in \Hilb^k_p(X)$ is called curvilinear if $\xi$ is contained in some smooth curve $C\subset X$. Equivalently, $\xi$ is curvilinear if $\calo_\xi$ is isomorphic  to the $\CC$-algebra $\CC[z]/z^{k}$.
The punctual curvilinear locus at $p\in X$ is the set of curvilinear subschemes supported at $p$: 
\[\mathrm{Curv}^k(X)_p=\{\xi \in \Hilb^k_p(X): \xi \subset \mathcal{C}_p \text{ for some smooth curve } \mathcal{C} \subset X\}=
\{\xi:\calo_\xi \simeq \CC[z]/z^{k}\}\]
\end{definition}

For surfaces ($n=2$) $\mathrm{Curv}^k_p(X)$ is an irreducible quasi-projective variety of dimension $k-1$ which is an open dense subset in $\Hilb^k_p(X)$ and therefore its closure is the full punctual Hilbert scheme at $p$, that is, $\overline{\mathrm{Curv}}^{[k]}_p(X)=\Hilb^k_p(X)$. When $n\ge 3$ the punctual Hilbert scheme $\Hilb^k_p(X)$ is not necessarily irreducible or reduced, but the closure of the curvilinear locus is one of its irreducible components:  
\begin{lemma} $\overline{\mathrm{Curv}}^{[k]}_p$ is an irreducible component of the punctual Hilbert scheme $\Hilb^k_p(X)$ of dimension $(n-1)(k-1)$. We call this punctual component the curvilinear Hilbert scheme (supported at $p$) and denote it by $\CHilb^k_p(X)$.
\end{lemma}
\begin{proof}
Note that $\xi \in \Hilb^{[k]}_0(\CC^n)$ is not curvilinear if and only if $\calo_\xi$ does not contain elements of degree $k-1$, that is, after fixing some local coordinates $x_1,\ldots, x_n$ of $\CC^n$ at the origin we have
\[\calo_\xi \simeq \CC[x_1,\ldots, x_n]/I \text{ for some } I\supseteq (x_1,\ldots, x_n)^{k-1}.\]
This is a closed condition and therefore curvilinear subschemes can't be approximated by non-curvilinear subschemes in $\Hilb^{[k]}_0(\CC^n)$. The dimension of $\CHilb^k_p(X)$ comes from the description of it as a non-reductive quotient in the previous section.  
\end{proof}

Note that any curvilinear subscheme contains only one subscheme for any given smaller length and any small deformation of a curvilinear subscheme is again locally curvilinear.

\begin{remark}\label{remark:embed}
Fix coordinates $x_1,\ldots, x_n$ on $\CC^n$. Recall that the defining ideal $I_\xi$ of any subscheme $\xi \in \Hilb^{k+1}_0(\CC^n)$ is a codimension $k$ subspace in the maximal ideal $\mathfrak{m}=(x_1,\ldots, x_n)$. The dual of this is a $k$-dimensional subspace $S_\xi$ in $\mathfrak{m}^*\simeq \symdot$ giving us a natural embedding $\rho: X^{[k+1]}_p \hookrightarrow \grass_k(\symdot)$. The test curve model of Theorem \ref{bszmodel} gives an explicit  parametrization of this embedding.
\end{remark}
 
\subsection{Tautological bundles over punctual components}\label{subsec:taurestrict}

Over the punctual Hilbert scheme $\Hilb^k_0(X)$ the tautological integral can be described as follows. Let $F$ be a rank $r$ vector bundle over $X$. The fibre of the corresponding rank $rk$ tautological bundle $F^{[k]}$ on $\CHilb^k(X)$ at the point $\xi$ is 
\[F^{[k]}_{\xi}=H^0(\xi, F|_{\xi})=H^0(\calo_{\xi} \otimes F).\]
On the level of bundles we have the following.
\begin{lemma}[\cite{b0} Lemma 3.15]\label{lemma:tauoncurvi} There is an isomorphism of topological vector bundles 
\[F^{[k]}|_{\CHilb^k_0(X)}\simeq \calo_{X}^{[k]} \otimes \pi^*F\]
 where $\pi: \CHilb^k_0(X) \to X$ is the projection. 
\end{lemma}
The natural embedding $\rho: \CHilb^{k+1}(X) \hookrightarrow \grass_{k}(J_k(n,1)^*)$ 
in Remark \ref{remark:embed} identifies the fibres of $\calo_{X}^{[k+1]}$ over $\xi\in \CHilb^{k+1}_p(X)$ with
$H^0(\calo_\xi) \simeq \calo_{X,p} \oplus \cale_{\rho(\xi)}$
where $\cale$ is the tautological rank $k$ bundle over $\grass_k(J_k(n,1)^*)$.
Hence the total Chern class of $F^{[k+1]}$ can be written as 
\begin{equation}\label{chernroots}
c(F^{[k+1]})=\prod_{j=1}^r(1+\theta_j)\prod_{i=1}^k\prod_{j=1}^r(1+\eta_i+\theta_j)
\end{equation}
where $c(F)=\prod_{j=1}^r(1+\theta_j)$ and $c(\cale)=\prod_{i=1}^k(1+\eta_i)$ are the Chern classes for the corresponding bundles. In particular the Chern class
\begin{equation}\label{chernclasses}
c_i(F^{[k+1]})=\mathcal{C}_i(c_1(\cale),\ldots c_k(\cale),c_1(F),\ldots, c_r(F))
\end{equation}
can be expressed as a polynomial function $\mathcal{C}_i$ in Chern classes of $\cale$ and $F$.

\section{The test curve model of Morin singularities}

In this section we describe the so-called test jet model of the curvilinear locus of the punctual Hilbert scheme $\Hilb_0^k(\CC^n)$. This open locus can be described as the moduli of holomorphic map jets $\CC \to \CC^n$ of order $k$, which has a canonical embedding into a Grassmannian. The compactification in this Grassmannian is the curvilinear component. 
 
\subsection{The test jet model of Berczi-Szenes \cite{bsz}}\label{subsec:jetdiff}

This model was developed as a compactification of the Morin contact singularity class of type $A_k$, and localisation on this model resulted in new formulas for Thom polynomials in \cite{bsz}. 

\subsubsection{Jets of holomorphic maps}\label{subsubsec:holmaps} If $u,v$ are positive integers let $J_k(u,v)$ denote the vector
space of $k$-jets of holomorphic maps $(\CC^u,0) \to (\CC^v,0)$ at
the origin, that is, the set of equivalence classes of maps
$f:(\CC^u,0) \to (\CC^v,0)$, where $f\sim g$ if and only if
$f^{(j)}(0)=g^{(j)}(0)$ for all $j=1,\ldots ,k$. This is a finite-dimensional complex
vector space, which one can identify with $J_k(u,1) \otimes \CC^v$; hence
$\dim J_k(u,v) =v \binom{u+k}{k}-v$.  We will call the elements of
$J_k(u,v)$ {\em map-jets of order} $k$, or simply map-jets. 

Eliminating the terms of degree $k+1$ results in an algebra homomorphism
$J_k(u,1) \twoheadrightarrow J_{k-1}(u,1)$, and the chain
$J_k(u,1) \twoheadrightarrow J_{k-1}(u,1)
\twoheadrightarrow \ldots \twoheadrightarrow J_1(u,1)$ induces
an increasing filtration on $J_k(u,1)^*$:
\begin{equation}\label{filtration} 
J_1(u,1)^* \subset J_2(u,1)^* \subset \ldots \subset J_k(u,1)^*
\end{equation}
\begin{remark}\label{remark:diffop} The space $J_i(u,1)^*$ may be interpreted as set of
differential operators on $\CC^u$ of degree at most
$i$, and in particular, by taking symbols, we have
\begin{equation}\label{eq:dual}
  J_k(u,1)^* \cong \Sym^{\le k}\CC^u \overset{\mathrm{def}}=
\oplus_{l=1}^k \sym^l\CC^u,
\end{equation}
where $\sym^l$ stands for the symmetric tensor product and the
isomorphism is that of filtered $\GL(n)$-modules. 
Given a regular $k$-jet $f: (\CC,0) \to (\CC^n,0)$ in $J_k^{\reg}(1,n)$ we may push forward the differential operators of order $k$ on $\CC$ (with constant coefficients) to $\CC^n$ along $f$ which gives us a map 
\[\tilde{f}:J_k(1,1)^* \to \grass(k,J_k(n,1)^*).\]

\end{remark}

Choosing coordinates on $\CC^u$ and $\CC^v$ a $k$-jet $f \in J_k(u,v)$ can be identified with the set of derivatives at the
origin, that is the vector $(f'(0),f''(0),\ldots, f^{(k)}(0))$, where
$f^{(j)}(0)\in \mathrm{Hom}(\mathrm{Sym}^j\CC^u,\CC^v)$. This way we
get the equality
\begin{equation}\label{identification}
J_k(u,v) \simeq J_k(u,1) \otimes \CC^v \simeq \oplus_{j=1}^k\mathrm{Hom}(\mathrm{Sym}^j\CC^u,\CC^v).
\end{equation}
One can compose map-jets via substitution and elimination of terms
of degree greater than $k$; this leads to the composition map
\begin{equation}
  \label{comp}
\jetk uv \times \jetk vw \to \jetk uw,\;\;  (\Psi_1,\Psi_2)\mapsto
\Psi_2\circ\Psi_1 \mbox{ modulo terms of degree $>k$ }.
\end{equation}
When $k=1$, $J_1(u,v)$ may be identified with $u$-by-$v$ matrices,
and \eqref{comp} reduces to multiplication of matrices.

The $k$-jet of a curve $(\CC,0) \to (\CC^n,0)$ is simply an element
of $J_k(1,n)$. We call such a curve $\gamma$ {\em regular}, if
$\gamma'(0)\neq 0$; introduce the notation $\jetreg 1n$ for the set
of regular curves:
\[\jetreg 1n=\left\{\g \in \jetk 1n; \g'(0)\neq 0 \right\}\]
Note that $\jetreg uu$ with the composition map \eqref{comp} has a natural group structure and we will often use the notation
\[\mathrm{Diff}_k(u)=\jetreg uu\]
and refer to this set as the {\em $k$-jet diffeomorphism group} to underline this property.

\subsubsection{Test curves for curvilinear subschemes}

Let $\xi \in \CHilb^{k+1}_0(\CC^n)$ be a curvilinear subscheme supported at the origin. Then $\xi$ is (scheme theoretically) contained in a smooth curve germ $\mathcal{C}_p$ in $\CC^n$:
\[\xi \subset \mathcal{C}_p \subset \CC^n.\]
Let $f_{\xi}:(\CC,0)\to (\CC^n,0)$ be a $k$-jet of germ parametrising $\mathcal{C}_p$. 
Then $f_{\xi}\in J^\reg_k(1,n)$ is determined only up to polynomial reparametrisation germs $\phi: (\CC,0)\to (\CC,0)$ and therefore we get the following lemma. 
\begin{lemma}\label{lemma:quotient} The punctual curvilinear locus $\CHilb^{[k+1]}_0(\CC^n)$ is equal (as a set) to the set of $k$-jet of regular germs at the origin, modulo polynomial reparametrisations: 
\begin{equation}\nonumber
\CHilb^{k+1}_0(\CC^n)=\{\text{regular }k\text{-jets } (\CC,0)\to (\CC^n,0)\}/\{\text{regular }k\text{-jets } (\CC,0)\to (\CC,0)\}=J^\reg_k(1,n)/\diff_k(1).
\end{equation}
\end{lemma}

We can explicitely write out the reparametrisation action (defined in \eqref{comp}) of $\diff_k(1)$ on $\jetreg 1n$ as follows. Let $f_{\xi}(z)=z f'(0)+\frac{z^2}{2!}f''(0)+\ldots +\frac{z^k}{k!}f^{(k)}(0) \in \jetreg 1n$ the $k$-jet of a germ at the origin (i.e no constant term) in $\CC^n$ with $f^{(i)}\in \CC^n$ such that $f' \neq 0$ and let  
$\varphi(z)=\alpha_1z+\alpha_2z^2+\ldots +\alpha_k z^k \in \jetreg 11$ with $\alpha_i\in \CC, \alpha_1\neq 0$.  
 Then 
 \[f \circ\varphi(z)
=(f'(0)\alpha_1)z+(f'(0)\alpha_2+
\frac{f''(0)}{2!}\alpha_1^2)z^2+\ldots
+\left(\sum_{i_1+\ldots +i_l=k}
\frac{f^{(l)}(0)}{l!}\alpha_{i_1}\ldots \alpha_{i_l}\right)z^k=\]
\begin{equation}\label{jetdiffmatrix}
=(f'(0),\ldots, f^{(k)}(0)/k!)\cdot 
\left(
\begin{array}{ccccc}
\alpha_1 & \alpha_2   & \alpha_3          & \ldots & \alpha_k \\
0        & \alpha_1^2 & 2\alpha_1\alpha_2 & \ldots & 2\alpha_1\alpha_{k-1}+\ldots \\
0        & 0          & \alpha_1^3        & \ldots & 3\alpha_1^2\alpha_ {k-2}+ \ldots \\
0        & 0          & 0                 & \ldots & \cdot \\
\cdot    & \cdot   & \cdot    & \ldots & \alpha_1^k
\end{array}
 \right)
 \end{equation}
where the $(i,j)$ entry is $p_{i,j}(\bar{\alpha})=\sum_{a_1+a_2+\ldots +a_i=j}\alpha_{a_1}\alpha_{a_2} \ldots \alpha_{a_i}.$

\begin{remark}\label{naturalembedding}
The linearisation of the action of $\diff_k(1)$ on $\jetreg 1n$ given as the matrix multiplication in \eqref{jetdiffmatrix} represents $\diff_k(1)$ as a upper triangular matrix group in $\GL(n)$. This is a non-reductive group so Mumford's reductive GIT is not applicable to study the geometry of the quotient $\jetreg 1n/ \diff_k(1)$, see Doran-Kirwan for details. 
Note that our matrix group is parametrised along its first row with the free parameters $\alpha_1,\ldots, \alpha_k$ and the other entries are certain (weighted homogeneous) polynomials in these free parameters. It is a $\CC^*$ extension of its maximal  unipotent radical
\[\diff_k(1)=U \rtimes \CC^*\]
where $U$ is the subgroup we get via substituting $\alpha_1=1$ and the diagonal $\CC^*$ acts with weights $0,1\ldots, n-1$ on the Lie algebra $\mathrm{Lie}(U)$. In B\'erczi and Kirwan \cite{bk} and B\'erczi, Doran, Hawes and Kirwan \cite{bdhk,bdhk2} we study actions of groups of this type in a more general context. 
\end{remark}

Fix an integer $N\ge 1$ and define
\[\Theta_k=\left\{\Psi\in J_k(n,N):\exists \g \in \jetreg 1n: \Psi \circ \g=0
\right\},\]
that is, $\Theta_k$ is the set of those $k$-jets of germs on $\CC^n$ at the origin which vanish on some regular curve. By definition, $\Theta_k$ is the image
of the closed subvariety of $\jetk nN \times \jetreg 1n$ defined by
the algebraic equations $\Psi \circ \g=0$, under the projection to
the first factor. If $\Psi \circ \gamma=0$, we call $\g$ a {\em test
curve} of $\Theta$. 

\begin{remark}
The subset $\Theta_k$ is the closure of an important singularity class in the jet space $J_k(n,N)$. These are called Morin singularities and the equivariant dual of $\Theta_k$ in $J_k(n,N)$ is called the Thom polynomial of Morin singularities, see B\'erczi and Szenes \cite{bsz} and Feh\'er and Rim\'anyi \cite{rf} for details.
\end{remark}

Test curves of germs are generally not unique. A basic but crucial observation is the following. If $\g$ is a test
curve of $\Psi \in \Theta_k$, and $\vp \in \diff_k(1)$ is a 
holomorphic reparametrisation of $\CC$, then $\g \circ \vp$ is,
again, a test curve of $\Psi$:
\begin{displaymath}
\label{basicidea}
\xymatrix{
  \CC \ar[r]^\vp & \CC \ar[r]^\g & \CC^n \ar[r]^{\Psi} & \CC^N}
\end{displaymath}
\[\Psi \circ \g=0\ \ \Rightarrow \ \ \ \Psi \circ (\g \circ \vp)=0\]

In fact, we get all test curves of $\Psi$ in this way if the
following property, open and dense in $\theta_k$, holds: the linear part of $\Psi$ has
$1$-dimensional kernel. Before stating this in Theorem 
\ref{embedgrass} below, let us write down the equation $\Psi \circ
\g=0$ in coordinates in an illustrative case. Let
$\g=(\g',\g'',\ldots, \g^{(k)})\in \jetreg 1n$ and
$\Psi=(\Psi',\Psi'',\ldots, \Psi^{(k)})\in \jetk nN$ be the
$k$-jets of the test curve $\g$ and the map $\Psi$ respectively. Using the chain rule and the notation $v_i=\g^{(i)}/i!$, the equation $\Psi \circ \g=0$ reads
as follows for $k=4$:
\begin{eqnarray} \label{eqn4}
& \Psi'(v_1)=0 \\ \nonumber & \Psi'(v_2)+\Psi''(v_1,v_1)=0 \\
\nonumber
& \Psi'(v_3)+2\Psi''(v_1,v_2)+\Psi'''(v_1,v_1,v_1)=0 \\
&
\Psi'(v_4)+2\Psi''(v_1,v_3)+\Psi''(v_2,v_2)+
3\Psi'''(v_1,v_1,v_2)+\Psi''''(v_1,v_1,v_1,v_1)=0
\nonumber
\end{eqnarray}


\begin{lemma}[Gaffney \cite{gaffney}, B\'erczi-Szenes \cite{bsz}]\label{explgp} Let
$\g=(\g',\g'',\ldots, \g^{(k)})\in \jetreg 1n$ and
$\Psi=(\Psi',\Psi'',\ldots, \Psi^{(k)})\in \jetk nN$ be $k$-jets.
Then substituting $v_i=\g^{(i)}/i!$, the equation $\Psi\circ \g=0$ is equivalent to
  the following system of $k$ linear equations with values in
  $\CC^N$:
\begin{equation}
  \label{modeleq}
\sum_{\ell \in \mathcal{P}(m)} \Psi(\bv_\ell)=0,\quad m=1,2,\dots, k
\end{equation}
Here $\mathcal{P}(m)$ denotes the set of partitions $\ell=1^{\ell_1}\ldots m^{\ell_m}$ of $m$ into nonnegative integers and $\bv_\ell=v_1^{\ell_1}\cdots v_{m}^{\ell_m}$. 
\end{lemma}
For a given $\g \in \jetreg 1n$ and $1\le i \le k$ let $\mathcal{S}^{i,N}_{\g}$ denote the set of 
solutions of the first $i$ equations in \eqref{modeleq}, that is,
\begin{equation}\label{solutionspace}
\mathcal{S}^{i,N}_\g=\left\{\Psi \in \jetk nN, \Psi \circ \g=0 \text{ up to order } i \right\}.
\end{equation}
The equations \eqref{modeleq} are linear in $\Psi$, hence
\[\mathcal{S}^{i,N}_\g \subset \jetk nN\]
is a linear subspace of codimension $iN$, i.e a point of $\grass_{\mathrm{codim}=iN}(J_k(n,N))$, whose orthogonal, $(\mathcal{S}^{i,N}_{\g})^\perp$, is an $iN$-dimensional subspace of $J_k(n,N)^*$. These subspaces are invariant under the reparametrization of $\g$. In fact, $\Psi \circ \gamma$ has $N$ vanishing coordinates and therefore
\[(\mathcal{S}^{i,N}_{\g})^\perp=(\mathcal{S}^{i,1}_{\g})^\perp \otimes \CC^N\]
holds. 

For $\Psi \in J_k(n,N)$ let $\Psi^1 \in \Hom(\CC^n,\CC^N)$ denote the linear part. When $N\ge n$ then the subset 
\[\tilde{\mathcal{S}}^{i,N}_{\g}=\{\Psi \in \mathcal{S}^{i,N}_{\g}: \dim \ker \Psi^1=1\}\]
is an open dense subset of the subspace $\mathcal{S}^{i,N}_{\g}$. In fact it is not hard to see that the complement $\tilde{\mathcal{S}}^{i,N}_{\g}\setminus \mathcal{S}^{i,N}_{\g}$where the kernel of $\Psi^1$ has dimension at least two is a closed subvariety of codimension $N-n+2$.


\begin{theorem}\label{embedgrass} The map
\[\phi: \jetreg 1n \rightarrow \grass_k(J_k(n,1)^*)\]
defined as  $\gamma  \mapsto (\mathcal{S}^{k,1}_\g)^\perp$
is $\diff_k(1)$-invariant and induces an injective map on the $\diff_k(1)$-orbits into the Grassmannian 
\[\phi^\grass: \jetreg 1n /\diff_k(1) \hookrightarrow \grass_k(J_k(n,1)^*).\]
Moreover, $\phi$ and $\phi^\grass$ are $\GL(n)$-equivariant with respect to the standard action of $\GL(n)$ on $\jetreg 1n \subset \Hom(\CC^k,\CC^n)$ and the induced action on $\grass_k(J_k(n,1)^*)$.
\end{theorem}

\begin{proof}
For the first part it is enough to prove that for $\Psi \in \Theta_k$ with $\dim \ker \Psi^1=1$ and $\gamma,\delta \in \jetreg 1n$
\[\Psi \circ \gamma=\Psi \circ \delta=0 \Leftrightarrow \exists
\Delta \in \jetreg 11 \text{ such that } \gamma=\delta
\circ \Delta.\]
We prove this statement by induction. Let $\gamma=v_1t+\dots +v_kt^k$ and
$\delta=w_1t+\dots+ w_kt^k$. Since $\dim \ker \Psi^1=1$, $v_1=\lambda w_1$, for some
$\lambda\neq0$. This proves the $k=1$ case. 

Suppose the statement is true for $k-1$. Then, using the appropriate
order-($k-1$) diffeomorphism, we can assume that $v_m=w_m$, $m=1\ldots
k-1.$ It is clear then from the explicit form \eqref{modeleq}
(cf. \eqref{eqn4}) of the equation
$\Psi\circ\gamma=0$, that  $\Psi^1(v_k)=\Psi^1(w_k)$, and hence
$w_k=v_k-\lambda v_1$ for some $\lambda\in\CC$. Then
$\gamma=\Delta\circ\delta$ for $\Delta=t+\lambda t^k$, and the proof
is complete.
\end{proof}

\begin{remark}\label{remark:orbit}
For a point $\gamma\in J_k^\reg(1,n)$ let $v_i=\frac{\g^{(i)}}{i!}\in \CC^n$ denote the normed $i$th derivative. Then from Lemma \ref{explgp} immediately follows that for $1\le i \le k$ (see \cite{bsz}):
\begin{equation}\label{sgamma}
(\mathcal{S}^{i,1}_\g)^\perp=\mathrm{Span}_\CC (v_1,v_2+v_1^2,\ldots, \sum_{i_1+\ldots +i_r=k}v_{i_1}\ldots v_{i_r})\subset \symdot
\end{equation} 
\end{remark}

Recall from Remark \ref{remark:diffop} that $J_k(n,1)^*=\symdot$. Note that the image of $\phi$ and the image of $\rho$ defined in Remark \ref{remark:embed} of the previous section coincide in $\grass_k(\symdot)$:
\[\mathrm{im}(\phi)=\mathrm{im}(\rho)\subset \grass_k(J_k(n,1)^*).\]
Hence their closure is the same, which is the first part of the next theorem. The second part follows from Remark \ref{remark:orbit} because $\phi$ is $\GL(n)$-equivariant.
\begin{theorem}[B\'erczi-Szenes model for $\CHilb^k_0(\CC^n)$, \cite{bsz}]\label{bszmodel}
\begin{enumerate}
\item For any $k,n$ we have 
\[\CHilb^{k+1}_0(\CC^n)=\overline{\mathrm{im}(\phi)} \subset \grass_k(J_k(n,1)^*).\]
\item Let $\{e_1,\ldots, e_n\}$ be a basis of $\CC^n$. For $k\le n$ the $\GL(n)$-orbit of 
\[p_{k,n}=\phi(e_1,\ldots, e_k)=\mathrm{Span}_\CC (e_1,e_2+e_1^2,\ldots, \sum_{i_1+\ldots +i_r=k}e_{i_1}\ldots e_{i_r})\] 
forms a dense subset of the image $\jetreg 1n$ and therefore (see e.g Remark \ref{remark:embed})
\[\CHilb^{k+1}_0(\CC^n)=\overline{\mathrm{GL_n} \cdot p_{k,n}}.\] 
\end{enumerate}
\end{theorem}

\subsection{Tautological integrals and Thom polynomials}


In this subsection we recall from \cite{bsz} how  the test curve model explained in the previous section reinterprets Thom polynomials as certain equivariant tautological integrals over the punctual curvilinear Hilbert scheme $\CHilb^{k+1}_0(\CC^n)$. Recall from Theorem \ref{bszmodel} that 
\[\CHilb^{k+1}_0(\CC^n)=\overline{\mathrm{im}(\phi)} \subset \grass_k(\wedge^k \symdot)\]
and therefore there are two canonically defined bundles on $\CHilb^{k+1}_0(\CC^n)$: 
\begin{enumerate}
\item The tautological rank $k+1$ bundle $\calo_{\CC^n}^{[k+1]}$ associated to the trivial bundle on $\CC^n$.  
\item The tautological rank $k$ bundle $\cale$, which is the restriction of the tautological bundle over $\grass_k(\wedge^k \symdot)$. 
\end{enumerate}
These fit into the short exact sequence 
\[\xymatrix{0 \ar[r] & \calo_\grass \ar[r] & \calo_{\CC^n}^{[k+1]} \ar[r] & \cale \ar[r] & 0}.\]
Let $\CC^m$ be the trivial $\GL(m)$-equivariant bundle over $\grass_k(J_k(n,1)^*)$. By Lemma \ref{lemma:tauoncurvi}, by tensoring with $\CC^m$ we obtain the short exact sequence over $\CHilb^{k+1}_0(\CC^n)$
\[\xymatrix{0 \ar[r] & \CC^m  \ar[r] & (\CC^m)^{[k+1]} \ar[r] & \cale \otimes \CC^m \ar[r] & 0}.\]
The equivariant Euler class of $\cale \otimes \CC^m$ can be written as 
\begin{equation}\label{chernroots}
\Euler(\cale \otimes \CC^m)=\prod_{i=1}^k\prod_{j=1}^m(\eta_i+\theta_j)
\end{equation}
where $c^T(\cale)=\prod_{j=1}^m(1+\theta_j)$ and $c^T(\cale)=\prod_{i=1}^k(1+\eta_i)$ are the equivariant Chern classes for the corresponding bundles.
\begin{theorem}[\cite{bsz}] \label{thm:thomtau} Thom polynomials of Morin singularities can be expressed as equivariant tautological integrals over curvilinear Hilbert schemes as follows:
\[\Tp_k^{n,m}=\int_{\CHilb^{k+1}_0(\CC^n)}\Euler(\cale \otimes \CC^m)(TN,f^*(TM)),\]
obtained by substituting the Chern roots of $TN,TM$ into the equivariant integral, which sits in the ring $\CC[x_1,\ldots, x_n,y_1,\ldots, y_m]^{S_n\times S_m}$ of bi-symmetric polynomials on $n+m$ variables.
\end{theorem}

\section{Non-reductive geometric invariant theory}\label{sec:nrgit}

In \cite{bdhk2} an extension of Mumford's classical GIT is developed for linear actions of a non-reductive linear algebraic group with internally graded unipotent radical  over an algebraically closed field $\kk$ of characteristic 0.

\begin{definition}\label{def:gradedunipotent} 
We say that a linear algebraic group $H = U \rtimes R$  has {\em internally graded unipotent radical} $U$ if there is a central one-parameter subgroup $\l:\GG_m \to Z(R)$ of the Levi subgroup $R$ of $H$ 
 such that the adjoint action of $\GG_m$ on the Lie algebra of $U$ has all its weights strictly positive. 
 Then $\hat{U} = U \rtimes \l(\GG_m)$ is a normal subgroup of $H$ and $H/\hU \cong R/\l(\GG_m)$ is reductive.
\end{definition}

Let $H=U \rtimes R$ be a linear algebraic group with internally graded unipotent radical $U$ acting linearly with respect to an ample line bundle $L$ on a projective variety  $X$; that is, the action of $H$ on $X$ lifts to an action on $L$ via automorphisms of the line bundle. When $H=R$ is reductive, using Mumford's classical geometric invariant theory (GIT)  \cite{git}, we can define $H$-invariant open subsets $X^s \subseteq X^{ss}$ of $X$ (the stable and semistable loci for the linearisation) with a geometric quotient $X^s/H$ and  projective completion $X/\!/H \supseteq X^s/H$ which is the projective variety associated to the algebra  of invariants
 $\bigoplus_{k \geq 0} H^{0}(X,L^{\otimes k})^H$. 
The variety $X/\!/H$ is the image
of a surjective morphism $\phi$ from the open subset $X^{ss}$ of $X$  such that if $x,y \in X^{ss}$ then $\phi(x) = \phi(y)$ if and only if the closures of the $H$-orbits of $x$ and $y$ meet in $X^{ss}$. 
Furthermore the subsets $X^s$ and $X^{ss}$ can be described using the Hilbert--Mumford criteria for stability and semistability. 

Mumford's GIT 
 does not have an immediate extension  to actions of non-reductive linear algebraic groups $H$, since the algebra of invariants $\bigoplus_{k \geq 0} H^{0}(X,L^{\otimes k})^H$ is not necessarily finitely generated as a graded algebra when $H$ is not reductive. 
It is still possible to define semistable and stable subsets $X^{ss}$ and $X^s$, with a geometric quotient $X^s/H$ which is an open subset of a so-called enveloping quotient $X\env H$ with an $H$-invariant morphism $\phi: X^{ss} \to X\env H$, and if the algebra of invariants 
 $\bigoplus_{k \geq 0} H^{0}(X,L^{\otimes k})^H$ is finitely generated then $X\env H$ is the associated projective variety
 \cite{BDHK0,dorankirwan}. But in general the enveloping quotient $X\env H$ is not necessarily projective, the morphism $\phi$ is not necessarily surjective (and its image may be only a constructible subset, not a subvariety, of $X\env H$). In addition there are in general no obvious analogues of the Hilbert--Mumford criteria.

However when $H = U \rtimes R$ has internally graded unipotent radical $U$ and acts linearly on a projective variety $X$, then provided that we are willing to modify the linearisation of the action by replacing the line bundle $L$ by a sufficiently divisible tensor power and multiplying by a suitable character of $H$ (which will not change the action of $H$ on $X$), many of the  key features of classical GIT still apply.

Let such an $H$ act linearly on an irreducible projective variety $X$ with respect to a very ample line bundle $L$.
Let $\chi: H \to \GG_m$ be a character of $H$. Its kernel contains $U$, and its restriction to $\hU$ can be identified with an integer so that the integer 1 corresponds to the character of $\hU$ which fits into the exact sequence $U \hookrightarrow \hat{U} \to \l(\GG_m)$. Let $\weight_{\min}$ be the minimal weight for the $\l(\GG_m)$-action on
$V:=H^0(X,L)^*$ and let $V_{\min}$ be the weight space of weight $\weight_{\min}$ in
$V$. Suppose that $\weight_{\min}=\weight_0 < \weight_{1} < 
\cdots < \weight_{\max} $ are the weights with which the one-parameter subgroup $\l: \GG_m \leq \hU \leq H$ acts on the fibres of the tautological line bundle $\calo_{\PP((H^0(X,L)^*)}(-1)$ over points of the connected components of the fixed point set $\PP((H^0(X,L)^*)^{\GG_m}$ for the action of $\GG_m$ on $\PP((H^0(X,L)^*)$; since $L$ is very ample $X$ embeds in $\PP((H^0(X,L)^*)$ and the line bundle $L$ extends to the dual $\calo_{\PP((H^0(X,L)^*)}(1)$ of the tautological line bundle on $\PP((H^0(X,L)^*)$.
Note that we can assume that there exist at least two distinct such weights since otherwise the action of the unipotent radical $U$ of $H$ on $X$ is trivial, and so the action of $H$ is via an action of the reductive group $R=H/U$.
\begin{definition}\label{def:welladapted}
Let $c$ be a positive integer such that 
$$ 
\frac{\chi}{c} = \weight_{\min} + \epsilon$$
where $\epsilon>0$ is sufficiently small; we will call rational characters $\chi/c$  with this property {\it well adapted } to the linear action of $H$, and we will call the linearisation well adapted if $\weight_{\min} <0\leq \weight_{\min} + \epsilon$ for sufficiently small $\epsilon>0$. How small $\epsilon$ is required to be will depend on the situation; more precisely, we will say that some property P holds for well adapted linearisations if there exists $\epsilon(P) > 0$ such that property P holds for any linearisation for which $\weight_{\min} <0\leq \weight_{\min} + \epsilon(P)$.
\end{definition}

\begin{remark}
In \cite{bkgrosshans} it is shown that under hypotheses which will be satisfied in our situation it suffices to take $0< \epsilon < 1$.
\end{remark}

 The linearisation of the action of $H$ on $X$ with respect to the ample line bundle $L^{\otimes c}$ can be twisted by the character $\chi$ so that the weights $\weight_j$ are replaced with $\weight_jc-\chi$;
let $L_\chi^{\otimes c}$ denote this twisted linearisation. 
Let $X^{s,\GG_m}_{\min+}$ denote the stable subset of $X$ for the linear action of $\GG_m$ with respect to the linearisation $L_\chi^{\otimes c}$; by the theory of variation of (classical) GIT \cite{Dolg,Thaddeus}, if $L$ is very ample then  $X^{s,\GG_m}_{\min+}$ is the stable set for the action of $\GG_m$ with respect to any rational character $\chi/c$ such that 
$\weight_{\min} < \chi/c < \weight_{\min + 1}$.
Let
\[
Z_{\min}:=X \cap \PP(V_{\min})=\left\{
\begin{array}{c|c}
\multirow{2}{*}{$x \in X$} & \text{$x$ is a $\GG_m$-fixed point and} \\ 
 & \text{$\GG_m$ acts on $L^*|_x$ with weight $\weight_{\min}$} 
\end{array}
\right\}
\]
and
\[
X_{\min}^0:=\{x\in X \mid p(x)  \in Z_{\min}\}  \quad \mbox{ where } \quad  p(x) =  \lim_{\substack{ t \to 0\\ t \in \GG_m }} t \cdot x \quad \mbox{ for } x \in X.
\]   

\begin{definition} \label{cond star}(
\textit{cf.}\ \cite{bdhk2}) 
With this notation, we define the following condition for the $\hU$-action on $X$:
\begin{equation}\label{star}
 \stab_{U}(z) = \{ e \}  \textrm{ for every } z \in Z_{\min}. \tag{$*$}
\end{equation}
Note that \eqref{star} holds if and only if we have $\stab_{U}(x) = \{e\}$ for all $x \in X^0_{\min}$. This  is also referred to as the condition that \lq semistability coincides with stability' for the action of $\hU$ (or, when $\lambda:\GG_m \to R$ is fixed, for the linear action of $U$); see Definition \ref{def:s=ss} below. 
\end{definition}

\begin{definition}\label{def min stable}   
When (\ref{star}) holds for a well adapted action of $\hU$ the min-stable locus for the $\hU$-action is 
\[ X^{s,{\hU}}_{\min+}= X^{ss,{\hU}}_{\min+}= \bigcap_{u \in U} u X^{s,\lambda(\GG_m)}_{\min+} = X^0_{\min} \setminus U Z_{\min}.  \]
\end{definition}



\begin{definition}\label{def:welladaptedaction} A {\it well-adapted linear action} of the linear algebraic group $H$ on an irreducible projective variety consists of the data $(X,L,H,\hat{U},\chi)$ where
\begin{enumerate}
\item $H$ is a linear algebraic group with internally graded unipotent radical $U$,  
\item $H$ acts linearly on $X$ with respect to a very ample line bundle $L$, while $\chi: H \to \GG_m$ is a character of $H$ and $c$ is a positive integer such that the rational character $\chi/c$ is well adapted for the linear action of $\hU =U \rtimes \GG_m$ on $X$.
\end{enumerate}
We will often refer to this set-up simply as a well-adapted action of $H$ on $X$.
\end{definition}

\begin{theorem}[\cite{bdhk2}] \label{mainthm} 
Let $(X,L,H,\hat{U},\chi)$ be a well-adapted linear action satisfying condition (\ref{star}). Then 
\begin{enumerate}
\item the algebras of invariants 
\[\oplus_{m=0}^\infty H^0(X,L_{m\chi}^{\otimes cm})^{\hat{U}} \text{ and } \oplus_{m=0}^\infty H^0(X,L_{m\chi}^{\otimes cm})^{H} = (\oplus_{m=0}^\infty H^0(X,L_{m\chi}^{\otimes cm})^{\hat{U}})^{R}\]
are finitely generated;   
\item the enveloping quotient $X\env \hat{U}$ is the projective variety associated to the algebra of invariants $\oplus_{m=0}^\infty H^0(X,L_{m\chi}^{\otimes cm})^{\hat{U}}$ and is a geometric quotient of the open subset $X^{s,\hU}_{\min+}$ of $X$ by $\hat{U}$;
  

\item the enveloping quotient $X\env H$ is the projective variety associated to the algebra of invariants $\oplus_{m=0}^\infty H^0(X,L_{m\chi}^{\otimes cm})^{{H}}$ and is the classical GIT quotient of $X \env \hat{U}$ by the induced action of $R/\l(\GG_m)$ with respect to the linearisation induced by a sufficiently divisible tensor power of $L$. 
\end{enumerate}
\end{theorem}

\begin{definition}\label{def:s=ss} Let $X$ be a projective variety 
which has a well adapted linear action of a linear algebraic group $H=U\rtimes R$ with internally graded unipotent radical $U$. 
When (\ref{star}) holds we denote by $X^{s,{{H}}}_{\min+}$ and $X^{ss,{{H}}}_{\min+}$ the pre-images in $X^{s,{{\hU}}}_{\min+}=X^{ss,{{\hU}}}_{\min+}$ of the stable and semistable loci for the induced linear action of the reductive group $H/\hU = R/\l(\GG_m)$ on 
$X\env \hat{U} = X^{s,\hU}_{\min+}/\hat{U}$.

By $H$-stability=$H$-semistability we mean that (\ref{star}) holds and $X^{s,{{H}}}_{\min+}=X^{ss,{{H}}}_{\min+}$. The latter is equivalent to the requirement that $\Stab_H(x)$ is finite for all $x \in X^{ss,{{H}}}_{\min+}$; then the projective variety $X\env H$ is a geometric quotient of the open subset $X^{s,{{H}}}_{\min+}=X^{ss,{{H}}}_{\min+}$ of $X$ by the action of $H$.

\end{definition}

\begin{remark}
When the conditions of Theorem \ref{mainthm} hold, we call $X \env H$ (respectively $X \env \hU$) the GIT quotient and we denote it by $X/\!/H$ (respectively $X /\!/ \hU$). 
\end{remark}

It is shown in \cite{BDHK0} that if $(X,L,H,\hat{U},\chi)$ is a well-adapted linear action satisfying $H$-stability=$H$-semistability, then  
\begin{enumerate}
\item
there is a sequence of blow-ups of $X$ along $H$-invariant projective subvarieties  resulting in a projective variety $\hat{X}$ with a well adapted linear action of $H$ which satisfies the  condition \eqref{star}, so that Theorem \ref{mainthm} applies, giving us a projective geometric quotient
$$\hat{X} /\!/ \hat{U} = \hat{X}^{s,\hU}_{\min+}/\hU$$
and its (reductive) GIT quotient $\hat{X} /\!/ H = (\hat{X} /\!/ \hat{U} ) /\!/ R  = (\hat{X} /\!/ \hat{U} ) /\!/ R$ where $R=H/U$;
\item
there is a sequence of further blow-ups along $H$-invariant projective subvarieties  resulting in a projective variety $\tilde{X}$ satisfying the same conditions as $\hat{X}$ and in addition $\tilde{X} /\!/ H = \Proj(\oplus_{m=0}^\infty H^0(X,L_{m\chi}^{\otimes cm})^{H})$ is the geometric quotient by $H$ of  the  $H$-invariant open subset $\tilde{X}^{s,H}_{\min+}$. 
\end{enumerate}

\section{Moment maps and cohomology of non-reductive quotients}\label{NRGITlocalisation}

In this section we briefly summarise the results of \cite{bkcoh}, which generalise results of the second author \cite{francesthesis} and Martin \cite{SM} to the cohomology of GIT quotients by non-reductive groups with internally graded unipotent radicals.

First let us recall the reductive picture. Let $X$ be a nonsingular complex projective variety acted on by a complex reductive group $G$ with respect to an ample linearisation. Then we can choose a maximal compact subgroup $K$ of $G$ and a $K$-invariant Fubini--Study K\"ahler metric on $X$ with corresponding moment map $\mu:X\to \mathfrak{k}^*$, where $\mathfrak{k}$ is the Lie algebra of $K$ and $\mathfrak{k}^* = \Hom_{\RR}(\mathfrak{k},\RR)$ is its dual. $\mathfrak{k}^*$ embeds naturally in the complex dual $\mathfrak{g}^* = \Hom_{\CC}(\mathfrak{k},\CC)$ of the 
Lie algebra $\mathfrak{g} = \mathfrak{k} \otimes \CC$ of $G$, as $\mathfrak{k}^* = \{\xi \in \mathfrak{g}^*: \xi(\mathfrak{k}) \subseteq \RR\}$; using this identification we can regard $\mu: X \to \mathfrak{g}^*$ 
 as a \lq moment map'  for the action of $G$, although of course it is not a moment map for $G$ in the traditional sense of symplectic geometry. 

In \cite{francesthesis} it is shown that the norm-square $f=||\mu||^2$ of the moment map $\mu:X\to \mathfrak{k}^*$ induces an equivariantly perfect Morse stratification of $X$ such that that the open stratum which retracts equivariantly onto the zero level set $\mu^{-1}(0)$ of the moment map coincides with the GIT semistable locus $X^{ss}$ for the linear action of $G$ on $X$. In particular this tells us that the restriction map
$$H^*_G(X;\QQ) \to H^*_G(X^{ss};\QQ)$$
is surjective; we also have an isomorphism (of vector spaces  though not of algebras) 
$H^*_G(X;\QQ)  \cong H^*(X;\QQ) \otimes H^*(BG;\QQ).$
Moreover, $\mu^{-1}(0)$ is $K$-invariant and its inclusion in $X^{ss}$ induces a homeomorphism 
\begin{equation}\label{fundamental}
\mu^{-1}(0)/K  \cong X/\!/G.
\end{equation} 
 When $X^{s}=X^{ss}$ the $G$-equivariant rational cohomology of $X^{ss}$ coincides with the ordinary rational cohomology of its geometric quotient $X^{ss}/G$, which is the GIT quotient $X/\!/G$, and we get expressions for the Betti numbers of $X/\!/G$ in terms of the equivariant Betti numbers of the unstable GIT strata, which can be described inductively, and of $X$ \cite{francesthesis}. In order to describe the ring structure on the rational cohomology of $X/\!/G$, the surjectivity of the composition
 $$\kappa: H^*_G(X;\QQ) \to H^*_G(X^{ss};\QQ) \cong H^*(X/\!/G;\QQ)$$
 can be combined with Poincar\'e duality on $X/\!/G$ and the nonabelian localisation formulas for intersection pairings on $X/\!/G$ given in \cite{jeffreykirwan}.

Martin \cite{SM} used \eqref{fundamental} to obtain formulas for the  intersection pairings on the quotient $X/\!/G$ in a different way, by relating these pairings to intersection pairings on the associated quotient $X/\!/T_\CC$, where $T_\CC\subseteq G$ is a maximal torus. He proved a formula expressing the rational cohomology ring of $X/\!/G$ in terms of the rational cohomology ring of $X/\!/T_\CC$ and an integration formula relating intersection pairings on the cohomology of $X/\!/G$ to corresponding pairings on $X/\!/T_\CC$. This integration formula, combined with methods from abelian localisation, leads to 
residue formulas for pairings on $X/\!/G$ which are closely related to those of \cite{jeffreykirwan} (see also \cite{vergne}).

In \cite{bkcoh} similar results are obtained for non-reductive actions.  Let $X$ be a nonsingular complex projective variety with a linear action of a complex linear algebraic group $H=U \rtimes R$ with internally graded unipotent radical $U$ with respect to an ample line bundle $L$; 
then the Levi subgroup $R$ is the complexification of a maximal compact subgroup $Q$ of $H$. 
 The unipotent radical $U$ of $H$ is internally graded by a central 1-parameter subgroup
$\l: \CC^* \to Z(R)$  of  $R$. Let $\hU=U \rtimes \l(\CC^*) \subseteq H$; then $\l(S^1) \subseteq \l(\CC^*) \subseteq \hU$ is a maximal compact subgroup of $\hU$. Assume also that semistability coincides with stability for the $\hU$-action, in the sense of Definition \ref{def:s=ss}. 
Using the embedding $X \subseteq \PP^n$ defined by a very ample tensor power of $L$, and a corresponding Fubini--Study K\"ahler metric invariant under the maximal compact subgroup $Q$ of $H$,  an $H$-moment map $\mu_{H}:X\to \mathrm{Lie}H^* = \Hom_\CC(\mathrm{Lie}H,\CC)$ is defined in \cite{bkcoh} by composing  the $G=\GL(n+1)$-moment map $\mu_{G}:X\to \mathfrak{g}^*$ with the map of complex duals $\mathfrak{g}^* \to \mathrm{Lie}(H)^*$ coming from the representation $H \to \GL(n+1)$. It is shown in \cite{bkcoh} that if the linearisation of the action of  $H$ on $X$ is well-adapted (which can be achieved by adding a suitable central constant to the moment map) and if  $H$-stability=$H$-semistability (see Definition \ref{def:s=ss}), then $H\mu_{H}^{-1}(0) = X^{s,H} = X^{ss,H}$ and the embedding  of $\mu_{H}^{-1}(0) $ in $X^{ss,H}$ induces a homeomorphism 
\[\mu_{H}^{-1}(0)/Q \simeq X/\!/H = X^{s,H}/H.\] 
In particular when $H=\hU=U \rtimes \CC^*$, this tells that the embedding $\mu_{\hU}^{-1}(0) \hookrightarrow X^{ss,\hU}$ induces a homeomorphism $\mu_{\hU}^{-1}(0)/S^1 \simeq X/\!/\hU$. Indeed  to have an embedding  of $\mu_{H}^{-1}(0) $ in $X^{ss,H}$ and an induced homeomorphism 
$ \mu_{H}^{-1}(0)/Q \simeq X/\!/H$, the condition that $H$-stability=$H$-semistability can be weakened to the requirement that $\hU$-stability=$\hU$-semistability. 

Similar results hold more generally when $X$ is compact K\"ahler but not necessarily projective. Suppose that $Y$ is a compact K\"ahler manifold acted on by a complex reductive Lie group $G$ such that $G$ is the complexification of a maximal compact subgroup $K$, so their Lie algebras satisfy $\lieg=\liek \oplus i\liek$. Let $B \subseteq G$ be a Borel subgroup such that $G=KB$ and $K\cap B=T$ is a maximal torus in $K$.  We fix $\hat{U}=U\rtimes \l(\CC)^* \subseteq B$ where $\l: \CC^* \to T_\CC$ grades the unipotent subgroup $U$ of $B$; then $K \cap \hat{U}=S^1$ is a maximal compact subgroup of $\hU$. The Lie algebra of $\hU$ decomposes as a real vector space as 
\begin{equation}\label{decomp}
\hat{\lieu}=\RR \oplus i\RR \oplus \lieu
\end{equation}
where $\mathrm{Lie}(K \cap \hU)=\RR$ and $\lieu$ is the Lie algebra of the complex unipotent group $U$. The set of positive roots $\Delta^+ \subseteq \Delta$ contains the weights of the adjoint action of $G$ on the Lie algebra of the unipotent radical of $B$, that is, the Cartan decomposition has the form 
\[\lieg=\lieg^- \oplus \liet_{\CC} \oplus \lieg^+ \text { where } \mathfrak{b}=\liet_{\CC} \oplus \lieg^+ \text{ and } \lieg^{\pm}=\oplus_{\a \in \Delta^{\pm}}e(\a).\]
Suppose that $H = U \rtimes R \subseteq G$ where $R$ is the complexification of $Q = K \cap H$ and $\l(\CC^*)$ is central in $R$, so that $H$ has internally graded unipotent radical with $\hU \subseteq H$. Suppose also that $X\subseteq Y$ is a compact complex submanifold invariant under the $H$ action, and  that $K$ 
preserves the K\"ahler structure on $Y$, so $S^1=K\cap \hU$ and $Q = K \cap H$ preserve the induced K\"ahler structure on $X$. 
Then we can define (generalised) moment maps $\mu_{\hU}$ and $\mu_H$ from $X$ to the complex duals of the Lie algebras of ${\hU}$ and $H$ by composing the restriction maps from $\mathfrak{g}^*$ to these duals with the $G$-moment map on $Y$ and the inclusion of $X$ in $Y$.

In the present paper we will work with actions of the diffeomorphism group (see \S \ref{subsec:jetdiff})
\[\hU=\diff_k(1)=\left\{ \left(
\begin{array}{ccccc}
\alpha_1 & \alpha_2   & \alpha_3          & \ldots & \alpha_k \\
0        & \alpha_1^2 & 2\alpha_1\alpha_2 & \ldots & 2\alpha_1\alpha_{k-1}+\ldots \\
0        & 0          & \alpha_1^3        & \ldots & 3\alpha_1^2\alpha_ {k-2}+ \ldots \\
0        & 0          & 0                 & \ldots & \cdot \\
\cdot    & \cdot   & \cdot    & \ldots & \alpha_1^k
\end{array}
 \right)  \begin{array}{c} : \alpha_1,\ldots,\alpha_k \in \CC, \\ \alpha_1 \neq 0 \end{array} \right\}
\]
on projective varieties. Therefore we only state the results of \cite{bkcoh} for the $H=\hU$ case. In this situation the
 symplectic description of the GIT quotient $X/\!/\hU$ as 
$X/\!/ \hU=\mu_{\hU}^{-1}(0)/S^1$
 fits into the diagram  
\begin{equation}\label{diagrammartin2}
\xymatrix{\mu_K^{-1}(0)/S^1  \ar@{^{(}->}[r]^-{j} & \mu_{\hU}^{-1}(0)/S^1=X/\!/\hU  \ar@{^{(}->}[r]^-{i} & \mu_{S^1}^{-1}(0)/S^1=X/\!/\CC^* 
}
\end{equation}
\begin{definition} For a weight $\a$ of $\CC^* \subseteq \hU$, let $\CC_\a$  denote the corresponding $1$-dimensional complex representation of $\CC^*$ and let 
\[L_\a:=\mu_{S^1}^{-1}(0)\times_{S^1} \CC_\a \to X/\!/\CC^*,\]
denote the associated line bundle whose Euler class is denoted by $e(\a) \in H^2(X/\!/\CC^*)\simeq H_\CC^2(X)$.  
For a $\CC^*$-invariant complex subspace $\mathfrak{a} \subseteq \mathfrak{b}$ let 
\[V_{\mathfrak{a}}= \mu_{S^1}^{-1}(0) \times_{S^1} \mathfrak{a} \to X/\!/\CC^*\]
denote the corresponding vector bundle. 
\end{definition}
Then we have 
\begin{proposition}[\cite{bkcoh}, Proposition 5.11]\label{propmartin}
\begin{enumerate}
\item The vector bundle $V_\lieu^* \to X/\!/\CC^*$ has a $C^{\infty}$-section $s$ which
is transverse to the zero section 
 and whose zero set is the submanifold
$\mu_{\hU}^{-1}(0)/\!/S^1 \subseteq X/\!/\CC^*$. Therefore the $\CC^*$-equivariant normal bundle is
\[\mathcal{N}(i)\simeq V_\lieu^*. \]
\item 
Let 
$\lieb=\hat{\lieu} \oplus \mathfrak{v}$
be a decomposition invariant under the adjoint $\CC^*$ action. Then the complex vector bundle $V^*_{\liev} \to X/\!/\CC^*$ has a transversal section whose zero set is the submanifold
$\mu_K^{-1}(0)/\!/S^1$. Therefore the $\CC^*$-equivariant normal bundles are
\[\mathcal{N}(j)\simeq V^*_{\liev} \text{ and } \mathcal{N}(i \circ j)\simeq V^*_{\liev \oplus \lieu}\]
\end{enumerate}
\end{proposition}

This leads us to the following theorems:

\begin{theorem}[\cite{bkcoh}, Theorem 5.12]\label{thm:cohomologyrings} 
Let $X$ be a smooth projective variety endowed with a well-adapted action of $\hU=U \rtimes \CC^*$ such that $\hU$-stability=$\hU$-semistability holds. Then there is a natural ring isomorphism
\[H^*(X/\!/\hU,\QQ)\simeq \frac{H^*(X/\!/\CC^*,\QQ)}{ann(\mathrm{Euler}(V_\lieu)}.\]
Here $\mathrm{Euler}(V_\lieu) \in H^*(X/\!/\CC^*)$ is the Euler class of the bundle $V_\lieu$ and  
\[\mathrm{ann}(\mathrm{Euler}(V_\lieu))=\{c \in H^*(X/\!/\CC^*,\QQ)| c \cup \mathrm{Euler}(V_\lieu)=0\} \subseteq H^*(X/\!/\CC^*,\QQ).\]
is the annihilator ideal. 
\end{theorem}
\begin{theorem}[\cite{bkcoh}, Theorem 5.13]\label{thm:integration} 
Let $X$ be a smooth projective variety endowed with a well-adapted action of $\hU=U \rtimes \CC^*$ such that $\hU$-stability=$\hU$-semistability holds.
 Assume that the stabiliser in $\hU$ of a generic $x \in X$ is trivial. Given a cohomology class $a \in H^*(X/\!/\hU)$  with a lift $\tilde{a}\in H^*(X/\!/\CC^*)$, then 
\[\int_{X/\!/\hU}a=\int_{X/\!/\CC^*}\tilde{a} \cup \mathrm{Euler}(V_\lieu),\]
where $\mathrm{Euler}(V_\lieu)$ is the cohomology class defined in Theorem \ref{thm:cohomologyrings}. Here we say that $\tilde{a}\in H^*(X/\!/\CC^*)$ is a lift of $a\in H^*(X/\!/\hU)$ if $a=i^*\tilde{a}$. 
\end{theorem}

\begin{remark}
Theorem \ref{thm:integration} can be generalised to allow the triviality assumption for the stabiliser in $\hU$ of a generic $x \in X$ to be omitted; then the sizes of the stabilisers in $\hU$ and $\CC^*$ of a generic $x \in X$ are included in the formula for $\int_{X/\!/\hU}a$.
\end{remark}

Finally, we have residue formulas for the intersection pairings on the quotient $X/\!/\CC^*$. There are two surjective ring homomorphisms 
\[\kappa_{\CC^*}: H_{S^1}^*(X;\QQ) \to H^*(X/\!/\CC^*;\QQ) \text{  and  }  \kappa_{\hU}: H_{\hU}^*(X;\QQ)=H_{S^1}^*(X;\QQ) \to H^*(X/\!/\hU;\QQ)\]
from the $S^1$-equivariant cohomology of $X$ to the ordinary cohomology of the corresponding GIT quotients. The fixed points of the maximal compact subgroup $S^1$ of $\hU$ on $X\subseteq \PP^n$ correspond to the weights of the $\CC^*$ action on $X$, and since this action is well-adapted, these weights satisfy
\[\omega_{\min}=\omega_0<0<\omega_1 <\ldots <\omega_{n}.\]
We can represent elements of $H_{\hU}^*(X;\QQ)=H_{S^1}^*(X;\QQ)$ as polynomial functions on the Lie algebra of $\CC^*$ whose coefficients are differential forms on $X$ and which are equivariantly closed.


\begin{theorem}[\cite{bkcoh}, Theorem 5.14 and Corollary 5.15]
\label{jeffreykirwannonred}
Let $X$ be a smooth projective variety endowed with a well-adapted action of $\hU=U \rtimes \CC^*$ such that $\hU$-stability=$\hU$-semistability holds (in the sense of Definitions \ref{cond star} and \ref{def:welladaptedaction}).  
Let $z$ be the standard coordinate on the Lie algebra of $\CC^*$. Given any $\hU$-equivariant cohomology class $\eta$ on $X$ represented by an equivariant differential form $\eta(z)$ whose degree is the dimension of $X/\!/\hU$, we have 
\[\int_{X/\!/\hU} \kappa_{\hU} (\eta) 
=n_{\CC^*}\res_{z=\infty} \int_{F_{\min}}\frac{i_{F_{\min}}^* (\eta(z) \cup \mathrm{Euler}(V_\lieu)(z))}{\mathrm{Euler}(\mathcal{N}_{F_{\min}})(z)} dz \]
where $F_{\min}$ is the union of those connected components of the fixed point locus $X^{\CC^*}$ on which the $S^1$-moment map takes its minimum value $\omega_{\min}$, and $n_{\hU}$ is the positive integer which is the order of the stabiliser in $\hU$ of a generic $x \in X$. 
\end{theorem}

\section{Equivariant localisation and multidegrees}\label{sec:equiv}

This section is a brief introduction to equivariant cohomology and localisation. For
more details, we refer the reader to Berline--Getzler--Vergne \cite{bgv} and B\'erczi--Szenes \cite{bsz}. 

Let $\kt\cong U(1)^n$ be the maximal compact subgroup of
$T\cong(\CC^*)^n$, and denote by $\mathfrak{t}$ the Lie algebra of $\kt$.  
Identifying $T$ with the group $\CC^n$, we obtain a canonical basis of the weights of $T$:
$\lambda_1,\ldots ,\lambda_n\in\mathfrak{t}^*$. 

For a manifold $M$ endowed with the action of $\kt$, one can define a
differential $d_\kt$ on the space $S^\bullet \
mathfrak{t}^*\otimes
\Omega^\bullet(M)^\kt$ of polynomial functions on $\mathfrak{t}$ with values
in $\kt$-invariant differential forms by the formula:
\[   
[d_\kt\alpha](X) = d(\alpha(X))-\iota(X_M)[\alpha(X)],
\]
where $X\in\mathfrak{t}$, and $\iota(X_M)$ is contraction by the corresponding
vector field on $M$. A homogeneous polynomial of degree $d$ with
values in $r$-forms is placed in degree $2d+r$, and then $d_\kt$ is an
operator of degree 1.  The cohomology of this complex--the so-called equivariant de Rham complex, denoted by $H^\bullet_T(M)$, is called the $T$-equivariant cohomology of $M$. Elements of $H_T^\bullet (M)$ are therefore polynomial functions $\mathfrak{t} \to \Omega^\bullet(M)^K$ and there is an integration (or push-forward map) $\int: H_T^\bullet(M) \to H_T^\bullet(\mathrm{point})=S^\bullet \mathfrak{t}^*$ defined as  
\[(\int_M \alpha)(X)=\int_M \alpha^{[\mathrm{dim}(M)]}(X) \text{ for all } X\in \mathfrak{t}\]
where $\alpha^{[\mathrm{dim}(M)]}$ is the differential-form-top-degree part of $\alpha$. The following proposition is the Atiyah-Bott-Berline-Vergne localisation theorem in the form of \cite{bgv}, Theorem 7.11. 
\begin{theorem}[(Atiyah-Bott \cite{atiyahbott}, Berline-Vergne \cite{berlinevergne})]\label{abbv} Suppose that $M$ is a compact complex manifold and $T$ is a complex torus acting smoothly on $M$, and the fixed point set $M^T$ of the $T$-action on M is finite. Then for any cohomology class $\a \in H_T^\bullet(M)$
\[\int_M \alpha=\sum_{f\in M^T}\frac{\a^{[0]}(f)}{\mathrm{Euler}^T(T_fM)}.\]
Here $\mathrm{Euler}^T(T_fM)$ is the $T$-equivariant Euler class of the tangent space $T_fM$, and $\alpha^{[0]}$ is the differential-form-degree-0 part of $\alpha$. 
\end{theorem}

The right hand side in the localisation formula considered in the fraction field of the polynomial ring of $H_T^\bullet (\mathrm{point})=H^\bullet(BT)=S^\bullet \mathfrak{t}^*$ (see more on details in Atiyah--Bott \cite{atiyahbott} and \cite{bgv}). Part of the statement is that the denominators cancel when the sum is simplified.

\subsection{Equivariant Poincar\'e duals and multidegrees}
\label{subsec:epdmult} 

Restricting the equivariant de Rham complex to compactly supported (or quickly
decreasing at infinity) differential forms, one obtains the compactly
supported equivariant cohomology groups $ H^\bullet_{\kt,\mathrm{cpt}}(M)
$. Clearly $H^\bullet_{\kt,\mathrm{cpt}}(M) $ is a module over
$H^\bullet_\kt(M)$. For the case when $M=W$ is an $N$-dimensional
complex vector space, and the action is linear, one has
$H^\bullet_\kt(W)= S^\bullet\mathfrak{t}^*$ and $ H^\bullet_{\kt,\mathrm{cpt}}(W) $ is
a free module over $H^\bullet_\kt(W)$ generated by a single element of
degree $2N$:
\begin{equation}
  \label{thomg}
   H^\bullet_{\kt,\mathrm{cpt}}(W) = H^\bullet_{\kt}(W)\cdot\mathrm{Thom}_{\kt}(W)
\end{equation}

Fixing coordinates $y_1,\dots,y_N$ on $W$, in which the $T$-action is
diagonal with weights $\eta_1,\ldots,  \eta_N$, one can write an explicit
representative of  $\mathrm{Thom}_{\kt}(W)$ as follows:
\[   \mathrm{Thom}_{\kt}(W) = 
e^{-\sum_{i=1}^N|y_i|^2}\sum_{\sigma\subset\{1,\ldots , N\}}
\prod_{i\in\sigma}\eta_i/2\cdot\prod_{i\notin \sigma}dy_i\,d\bar y_i
\]

We will say that an algebraic variety has dimension $d$ if its
maximal-dimensional irreducible components are of dimension $d$.  A
$T$-invariant algebraic subvariety $\Sigma$ of dimension $d$ in $W$
represents $\kt$-equivariant $2d$-cycle in the sense that
\begin{itemize}
\item a compactly-supported equivariant form $\mu$ of degree $2d$ is
  absolutely integrable over the components of maximal dimension of
  $\Sigma$, and $\int_\Sigma\mu\in S^\bullet \mathfrak{t}$;
\item if $d_\kt\mu=0$, then $\int_\Sigma\mu$ depends only on the class
  of $\mu$ in $ H^\bullet_{\kt,\mathrm{cpt}}(W) $,
\item and $\int_\Sigma\mu=0$  if $\mu=d_\kt\nu$ for a
  compactly-supported equivariant form $\nu$.
\end{itemize}

\begin{definition} \label{defepd} Let $\Sigma$ be an $T$-invariant algebraic
  subvariety of dimension $d$ in the vector space $W$. Then the
  equivariant Poincar\'e dual of $\Sigma$ is the polynomial on $\mathfrak{t}$
  defined by the integral
\begin{equation}
 \label{vergneepd}
 \epd\Sigma = \frac1{(2\pi)^d}\int_\Sigma\mathrm{Thom}_{\kt}(W).
\end{equation}  
\end{definition}
\begin{remark}
  \begin{enumerate}
  \item An immediate consequence of the definition is that for an equivariantly
closed differential form $\mu$ with compact support, we have
\[  \int_\Sigma\mu = \int_W \epd\Sigma\cdot\mu.
\]
This formula serves as the motivation for the term {\em equivariant
  Poincar\'e dual.}
\item This definition naturally extends to the case of an analytic
  subvariety of $\CC^n$  defined in the neighborhood of the origin, or
  more generally, to any $T$-invariant cycle in $\CC^n$.
  \end{enumerate}
\end{remark}

Another terminology for the equivariant Poincar\'e dual is {\em multidegree}, which is close in spirit to the original
construction of Joseph \cite{joseph}. Let  $\Sigma \subset W$ be a $T$-invariant
subvariety. Then we have
\[       \epd{\Sigma,W}_T=\mdeg{I(\Sigma),\CC[y_1,\ldots , y_N]}.
\] 

Some basic properties of the equivariant Poincar\'e dual are listed in \cite{bsz}, these are: Positivity, Additivity, Deformation invariance, Symmetry and a formula for complete intersections. Using these properties one can easily describe an algorithm for
computing $\mdeg{I,S}$ as follows (see Miller--Sturmfels \cite[\S8.5]{milsturm}, Vergne \cite{voj} and \cite{bsz} for details). 

An ideal $M\subset S$ generated by a set of monomials in
$y_1,\ldots, y_N$ is called a \emph{monomial ideal}. Since
$\mathrm{in}_<(I)$ is such an ideal, by the deformation invariance
it is enough to compute $\mdeg{M}$ for monomial ideals $M$. If the
codimension of $\Sigma(M)$ in $W$ is $s$, then the maximal
dimensional components of $\Sigma(M)$ are codimension-$s$ coordinate
subspaces of $W$. Such subspaces are indexed by subsets
$\mathbf{i}\in\{1\ldots  N\}$ of cardinality $s$; the corresponding
associated primes are $\mathfrak{p}[\mathbf{i}]=\langle y_i:i\in
\mathbf{i} \rangle$. Then 
\begin{equation*}
\label{primemon} \mult(\mathfrak{p}[\mathbf{i}],M)=
\left|\left\{\mathbf{a}\in\ZZ_+^{[{\mathbf{i}}]};\;
\mathbf{y}^{\mathbf{a}+\mathbf{b}}\notin M\text{ for all }
\mathbf{b}\in\ZZ_+^{[\hat{\mathbf{i}}]}\right\}\right|,
\end{equation*}
where $\ZZ_+^{[\mathbf{i}]}=\{\mathbf{a}\in \ZZ_+^N;a_i=0 \text{ for }
i\notin \mathbf{i}\}$, $\hat{\mathbf{i}}=\{1\ldots 
N\}\setminus\mathbf{i}$, and $|\cdot|$, as usual, stands for the
number of elements of a finite set. By the normalization and additivity axiom we have
\begin{equation}\label{mdegformula}
\mdeg{M,S} =
\sum_{|\mathbf{i}|=s}\mult(\mathfrak{p}[\mathbf{i}],M)
\prod_{i\in\mathbf{i}}\eta_i.
\end{equation}
By definition, the weights $\eta_1,\ldots \eta_N$ on $W$ are linear forms of $\l_1,\ldots \l_r$, the basis of $(\CC^*)^r$, and we denote the coefficient of $\l_j$ in $\eta_i$ by $\coeff(\eta_i,j)$, $1\le i\le N, 1\le j \le r$, and introduce 
\[\deg(\eta_1,\ldots, \eta_N;m)=\#\{i;\;\coeff(\eta_i,m)\neq 0\}\}.\]
It is clear from the formula \eqref{mdegformula} that 
\begin{equation*}
\deg_{\l_m}\mdeg{I,S} \le \deg(\eta_1,\ldots, \eta_N;m)
\end{equation*}
holds for any $1\le m \le r$. 
We need a slightly stronger result in the next section which we formulate and prove here.

\begin{proposition}\label{vanishlemma}
Let $W$ be an $N$-dimensional complex
vector space with coordinates $y_1,\dots,y_N$ endowed with an diagonal action of $(\CC^*)^r$ acting with weights $\eta_1\ldots \eta_N$. Let $X \subset W$ be a $(\CC^*)^r$-invariant irreducible subvariety not contained in the coordinate hyperplanes $\{y_i=0\}$ for $1\le i \le N$. Let $I=I(X)\subset S$ be its $(\CC^*)^r$-invariant prime ideal. Then 
\begin{equation*}
\deg_{\l_m}\mdeg{I,S} \le \deg(\eta_1,\ldots, \eta_N;m)-1
\end{equation*}
\end{proposition}

\proof
By the positivity property of the multidegree $\mdeg{I,S}$ is indeed a polynomial of the weights $\eta_i,i=1,\ldots, N$. Let 
\[\coeff(\eta_i,m)\neq 0 \text{ for } 1 \le i \le s;\ \coeff(\eta_{s+1},m)=\ldots =\coeff(\eta_{N},m)=0.\]
The idea of the proof is to choose an appropriate monomial order on the polynomial ring $S=\CC[y_1,\ldots, y_N]$ to ensure that $y_1$ does not appear in the corresponding initial ideal. 

To that end recall, that a weight function is a linear map $\rho: \ZZ^N \to \ZZ$. This defines a partial order $>_\rho$ on the monomials of $S$, called the weight order associated to $\rho$, by the rule $m=y^a >_\rho n=y^b$ iff $\rho(a)>\rho(b)$. Here $a=(a_1,\ldots ,a_N),b=(b_1,\ldots ,b_N)$ are arbitrary multiindices. Any weight order can be extended to a compatible monomial order $>$ (see Eisenbud \cite[Ch 15.2]{eisenbud}), which means that $m>_\rho n$ implies $m>n$. 
For our purposes define 
\begin{equation*}
\rho(y_1)=-1, \rho(y_2)=\ldots =\rho(y_N)=0
\end{equation*}    
and let $>$ denote arbitrary compatible monomial order on $S$. By definition for a monomial $m \in S$
\begin{equation}\label{weights}
\rho(m)<0 \Longleftrightarrow y_1|m.
\end{equation} 
Let $p=m_1+\ldots +m_t \in I$ be a polynomial where the $m_i$'s are monomials in the variables $y_1,\ldots, y_N$ of the same weighted degree. We claim that if $in_{<}(p)$ is a basis element of the initial ideal $in_<(I)$ then not all monomials $m_i$ of $p$ are divisible by $y_1$. Indeed, if $y_1|m_i$ for $i=1,\ldots, t$ then $y_1|p$. But $p\neq y_1$ by assumption because this would mean that $X\subset \{y_1=0\}$. Therefore $p'=p/y_1\in I$ and $y_1\cdot in_{>}(p')=in_{>}(p)$ holds and therefore $in_<(p)$ is not among the generators of $in_>(I)$. 
So there is a monomial of $p$ not containing $y_1$, and by \eqref{weights} the weight of this monomial is strictly bigger to the weight of any other containing $y_1$.  
Consequently, $y_1$ does not divide any of the generators of $in_>(I)$, and by \eqref{mdegformula}, $\mdeg{I,S}$ does not depend on $\eta_1$. The only possible variables containing $\l_m$ are therefore $\eta_2,\ldots ,\eta_s$, giving a maximum total degre $s-1$.  
\qed 

\subsection{The Rossman formula} \label{subsec:rossman} 

The Rossmann equivariant localisation formula is an improved version of the Atiyah-Bott/Berline-Vergne localisation for singular varieties sitting in a smooth ambient space. 
Let $Z$ be a complex manifold with a holomorphic $T$-action, and let
$M\subset Z$ be a $T$-invariant analytic subvariety with an isolated
fixed point $p\in M^T$. Then one can find local analytic coordinates
near $p$, in which the action is linear and diagonal. Using these
coordinates, one can identify a neighborhood of the origin in $\TT_pZ$
with a neighborhood of $p$ in $Z$. We denote by $\tc_pM$ the part of
$\TT_pZ$ which corresponds to $M$ under this identification;
informally, we will call $\tc_pM$ the $T$-invariant {\em tangent cone}
of $M$ at $p$. This tangent cone is not quite canonical: it depends on
the choice of coordinates; the equivariant dual of
$\Sigma=\tc_pM$ in $W=\TT_pZ$, however, does not. Rossmann named this
 the {\em equivariant multiplicity of $M$ in $Z$ at $p$}:
\begin{equation}\label{emult}
   \emu_p[M,Z] \overset{\mathrm{def}}= \epd{\tc_pM,\TT_pZ}.
\end{equation}

\begin{remark}
In the algebraic framework one might need to pass to the {\em
tangent scheme} of $M$ at $p$ (cf. Fulton \cite{fulton}). This is canonically
defined, but we will not use this notion.
\end{remark}
The analog of the Atiyah-Bott formula for singular subvarieties of smooth ambient manifolds is the following statement.
\begin{proposition}[Rossmann's localisation formula \cite{rossmann}]\label{rossman} Let $\mu \in H_T^*(Z)$ be an equivariant class represented by a holomorphic equivariant map $\mathfrak{t} \to\Omega^\bullet(Z)$. Then 
\begin{equation}
  \label{rossform}
  \int_M\mu=\sum_{p\in M^T}\frac{\emu_p[M,Z]}{\mathrm{Euler}^T(\TT_pZ)}\cdot\mu^{[0]}(p),
\end{equation}
where $\mu^{[0]}(p)$ is the differential-form-degree-zero component
of $\mu$ evaluated at $p$.  
\end{proposition}

\section{NRGIT quotients versus curvilinear Hilbert schemes}

The moduli of $k$-jets in $\CC^n$ is a quasi-affine non-reductive quotient. The NRGIT compactification and $\CHilb^{k+1}(\CC^n)$ are two different compactifications. In this section we explain the strategy of our argument: we find a blow-up of the NRGIT model which admits a morphism to $\CHilb^{k+1}(\CC^n)$, so tautological integrals over the curvilinear Hilbert scheme can be pulled back to equivariant integrals over the NRGIT model.
We set up the notations here, and explain the initial blow-up: fibration over the flag manifold, and introduce the coordinates $\b_{ij}$.

Recall the test curve model from Theorem \ref{bszmodel}, which says that for any $k,n$  
\[\CHilb^{k+1}(\CC^n)=\overline{\mathrm{im}(\phi)} \subset \grass_k(J_k(n,1)^*)\]
with the $\diff_k$-invariant morphism $\phi: J_k^{reg}(1,n) \to \grass_k(J_k(n,1)^*)$ (which we will refer to in short as the \textit{test curve morphism}) defined as
\[\phi(v_1,\ldots, v_k)=v_1 \wedge (v_2+v_1^2) \wedge (v_3+2v_1v_2+v_1^3) \wedge \ldots \wedge (\sum_{i_1+\ldots +i_r=k}v_{i_1}\ldots v_{i_r})\]
Let $k\le n$ and fix a basis $\{e_1,\ldots, e_n\}$ of $\CC^n$. Then, according to Theorem \ref{bszmodel}, the curvilinear Hilbert scheme is the closure of the $\GL(n)$ orbit of $p_{k,n}=\phi(e_1,\ldots, e_k)$:
\[\overline{\im(\phi)}=\overline{\mathrm{GL_n} \cdot p_{k,n}}.\]
In order to avoid working with singular centers of blow-ups we slightly modify the strategy outlined in \S \ref{sec:strategy}. 
Let $P_{n,k} \subset \mathrm{GL}_n$ denote the parabolic subgroup which preserves the flag 
\[\mathbf{f}=(\mathrm{Span}(e_1)   \subset \mathrm{Span}(e_1,e_2) \subset \ldots \subset \mathrm{Span}(e_1,\ldots, e_k) \subset \CC^n).\] 
\begin{definition}\label{def:xktilde} Define the partial desingularization 
\[\widetilde{\CHilb}^{k+1}(\CC^n)=\mathrm{GL}_n \times_{P_{n,k}} \overline{P_{n,k} \cdot p_{k,n}}\]
with the resolution map $\rho: \widetilde{\CHilb}^{k+1}(\CC^n) \to \CHilb^{k+1}(\CC^n)$ given by $\rho(g,x)=g\cdot x$. In short, this is the blow-up of the curvilinear Hilbert scheme at the linear part. Note that the blown-up space fibers over the complete flag manifold on $\CC^n$:
\[\pi: \widetilde{\CHilb}^{k+1}(\CC^n)\to \mathrm{GL}_n/P_{n,k}=\flag_{1,2,\ldots, k}(\CC^n)\]
\end{definition} 

Equivalently, let $\jetnondeg 1n \subset \jetreg 1n$ be the set of test curves with $\g',\ldots, \g^{(k)}$ linearly independent. These correspond to the nonsingular $n \times k$ matrices in $\Hom(\CC^k,\CC^n)$, and they fibre over the set of complete flags in $\CC^n$:
\begin{equation}\label{proj2}
\jetnondeg 1n/\diff_k(1) \to \Hom(\CC^k,\CC^n)/B_k=\flag_{1,\ldots, k}(\CC^n)
\end{equation}
where $B_k \subset \GL(k)$ is the upper Borel. The image of the fibres under $\phi$ are isomorphic to $P_{n,k} \cdot p_{k,n}$, and therefore $ \widetilde{\CHilb}^{k+1}(\CC^n)$ is the fibrewise compactification of $\jetnondeg 1n$ over $\flag_{1\ldots k}(\CC^n)$. 

The strategy outlined in \S \ref{sec:strategy} will be refined as follows: 
\begin{equation}\label{strategy1}
\xymatrix{**[r] \mathrm{Jet}_k  \ar[d]^\pi & **[l] \mathrm{Jet}_k^{ss} \ar@{_{(}->}[l] \ar[rd]^-{\tilde{\phi}} & &\\
 \GL(n)\times_{P_{n,k}} \PP(\CC \oplus \Hom^\ff(\CC^k,\CC^n)) \ar[rd] \ar@{.>}[rr]^-{\phi} & & \widetilde{\CHilb}^{k+1}(\CC^n) \ar[ld] \\
 & \flag_{1\ldots k}(\CC^n) & &}
\end{equation}
Hence we need to describe the iterated blow-up process of the fiber over $\ff$. Here $\Hom^\ff(\CC^k,\CC^n)$ are formed by $n\times k$ matrices whose column vectors $v_1,\ldots, v_k$ sit over $\mathbf{f}$, that is they have the form
\begin{align}\label{vectors}
v_1=& \vv_{11}e_1 \\ \nonumber
v_2= & \vv_{22}e_2+\vv_{12}e_1\\ \nonumber
\cdots\\ \nonumber
v_k= &\vv_{kk}e_k+\vv_{{k-1}k}e_{k-1}+\ldots +\vv_{1k}e_1\\ \nonumber
\end{align}
The rational map
\[\phi:\PP[\CC \oplus \Hom^\ff (\CC^k,\CC^n)] \dasharrow \PP(\wedge^k \symdot)\]
on the compactified fiber over $\mathbf{f}$ reads as
\[[x:v_1,\ldots v_k] \mapsto [v_1 \wedge (xv_2+v_1^2) \wedge (x^2v_3+xv_1v_2+v_1^3) \wedge \ldots \wedge (\sum_{i_1+\ldots +i_r=k}x^{k-r}v_{i_1}\ldots v_{i_r})]\]

\section{Equivariant localisation and the Residue Vanishing Theorem}\label{sec:loc}
In this subsection we develop a two step equivariant localisation method on $\widetilde{\CHilb}^{k+1}(\CC^n)$ which is a stronger version of our iterated residue in \cite{bsz}. We need an important restriction on the parameters to make this method work, namely we assume that $k\le n$ in this section.
The partial resolution $\rho: \widetilde{\CHilb}^{k+1}(\CC^n) \to \CHilb^{k+1}(\CC^n)$ fibers over the flag manifold $\flag_k(\CC^n)$ 
\begin{equation}\label{diagram}
\xymatrix{\widetilde{\CHilb}^{k+1}(\CC^n) \ar[r]^-{\rho} \ar[d]^{\mu} & \CHilb^{k+1}(\CC^n) \subset \grass_k(\symdot) \\
\Hom^{\nondeg}(\CC^k,\CC^n)/B_k=\flag_k(\CC^n)& } 
\end{equation}
where the fibres of $\mu$ are isomorphic to $\overline{P_{k,n}\cdot p_{k,n}}\subset \grass_k(\symdot)$.

Let $e_1,\ldots, e_n \in \CC^n$ be an eigenbasis of $\CC^n$ for the $T^n \subset \GL(n)$ action with weights $\l_1,\ldots, \l_n\in \mathfrak{t}^*$ and let $\ff=(\langle e_1 \rangle \subset \langle e_1,e_2 \rangle \subset \ldots \subset \langle e_1,\ldots,e_k \rangle \subset \CC^n)$
denote the standard flag in $\CC^n$ fixed by the upper Borel as before.


Since $\widetilde{\CHilb}^{k+1}(\CC^n)$ fibers over the flag manifold $\flag_k(\CC^n)$, the ABBV localisation formula of Proposition \ref{abbv} reads as   
\begin{equation} \label{flagloc}
\int_{\widetilde{\CHilb}^{k+1}(\CC^n)}\alpha= \sum_{\sigma\in\sg n/\sg{n-k}}
\frac{\alpha_{\sigma(\ff)}}{\prod_{1\leq m \leq
k}\prod_{i=m+1}^n(\lambda_{\sigma\cdot
    i}-\lambda_{\sigma\cdot m})},
\end{equation}
where 
\begin{itemize}
\item $\sigma$ runs over the ordered $k$-element subsets of $\{1,\ldots, n\}$ labeling the fixed flags $\sigma(\ff)=(\langle e_{\sigma(1)} \rangle \subset \ldots \subset \langle e_{\sigma(1)},\ldots, e_{\sigma(k)} \rangle \subset \CC^n)$ in $\CC^n$.
\item $\prod_{1\leq m \leq k}\prod_{i=m+1}^n(\lambda_{\sigma(i)}-\lambda_{\sigma(m)})$ is the equivariant Euler class of the tangent space of $\flag_k(\CC^n)$ at $\s(\ff)$.
\item if $X_{\sigma(\ff)}=\mu^{-1}(\sigma(\ff))$ denotes the fibre then $\alpha_{\sigma(\ff)}=(\int_{X_{\sigma(\ff)}} \alpha)^{[0]}(\sigma(\ff))\in S^\bullet \mathfrak{t}^*$ is the differential-form-degree-zero part evaluated at $\sigma(\ff)$.
\end{itemize}
In particular, the Chern roots of the tautological bundle over $\grass_k(\symdot)$ at the fixed point $\sigma(\ff)$ are represented by $\l_{\s(1)}, \ldots ,\l_{\s(k)}\in \mathfrak{t}^*$ and therefore if $\l$ is a Chern polynomial of the tautological bundle then 
\begin{equation}\label{alphasigmaf}
\alpha_{\s(\ff)}=\sigma \cdot \alpha_\ff=\alpha_\ff(\l_{\s(1)}, \ldots ,\l_{\s(k)})\in S^\bullet \mathfrak{t}^*,
\end{equation}
is the $\sigma$-shift of the polynomial $\alpha_{\ff}=(\int_{X_{\ff}}\alpha)^{[0]}(\ff)\in S^\bullet \mathfrak{t}^*$ corresponding to the distinguished fixed flag $\ff$.

\subsection{Transforming the localisation formula into iterated residue}\label{subsec:transform}
We transform the right hand side of \eqref{flagloc} into an iterated residue motivated by B\'erczi--Szenes \cite{bsz}. This step turns out to be crucial in handling the combinatorial complexity of the fixed point data in the Atiyah-Bott localisation formula and condense the symmetry of this fixed point data in an efficient way which enables us to prove the vanishing of the contribution of all but one of the fixed points. 

To describe this formula, we will need the notion of an {\em iterated
  residue} (cf. e.g. \cite{szenes}) at infinity.  Let
$\omega_1,\dots,\omega_N$ be affine linear forms on $\CC^k$; denoting
the coordinates by $z_1,\ldots, z_k$, this means that we can write
$\omega_i=a_i^0+a_i^1z_1+\ldots + a_i^kz_k$. We will use the shorthand
$h(\bz)$ for a function $h(z_1\ldots z_k)$, and $\dbz$ for the
holomorphic $n$-form $dz_1\wedge\dots\wedge dz_k$. Now, let $h(\bz)$
be an entire function, and define the {\em iterated residue at infinity}
as follows:
\begin{equation}
  \label{defresinf}
 \ires \frac{h(\bz)\,\dbz}{\prod_{i=1}^N\omega_i}
  \overset{\mathrm{def}}=\left(\frac1{2\pi i}\right)^k
\int_{|z_1|=R_1}\ldots
\int_{|z_k|=R_k}\frac{h(\bz)\,\dbz}{\prod_{i=1}^N\omega_i},
 \end{equation}
 where $1\ll R_1 \ll \ldots \ll R_k$. The torus $\{|z_m|=R_m;\;m=1 \ldots
 k\}$ is oriented in such a way that $\res_{z_1=\infty}\ldots
 \res_{z_k=\infty}\dbz/(z_1\cdots z_k)=(-1)^k$.
We will also use the following simplified notation: $\sires \overset{\mathrm{def}}=\ires.$

In practice, one way to compute the iterated residue \eqref{defresinf} is the following algorithm: for each $i$, use the expansion
 \begin{equation}
   \label{omegaexp}
 \frac1{\omega_i}=\sum_{j=0}^\infty(-1)^j\frac{(a^{0}_i+a^1_iz_1+\ldots
   +a_{i}^{q(i)-1}z_{q(i)-1})^j}{(a_i^{q(i)}z_{q(i)})^{j+1}},
   \end{equation}
   where $q(i)$ is the largest value of $m$ for which $a_i^m\neq0$,
   then multiply the product of these expressions with $(-1)^kh(z_1\ldots
   z_k)$, and then take the coefficient of $z_1^{-1} \ldots z_k^{-1}$
   in the resulting Laurent series.

\begin{proposition}[{\rm B\'erczi--Szenes \cite{bsz}, Proposition 5.4}] For any homogeneous polynomial $Q(\bz)$ on $\CC^k$ we have
\begin{equation}\label{flagres}
\sum_{\sigma\in\sg n/\sg{n-k}}
\frac{Q(\lambda_{\sigma(1)},\ldots ,\lambda_{\sigma(k)})}
{\prod_{1\leq m\leq k}\prod_{i=m+1}^n(\lambda_{\sigma\cdot
    i}-\lambda_{\sigma\cdot m})}=\sires
\frac{\prod_{1\leq m<l\leq k}(z_m-z_l)\,Q(\bz)\dbz}
{\prod_{l=1}^k\prod_{i=1}^n(\lambda_i-z_l)}.
\end{equation}
\end{proposition}

\begin{remark}
  Changing the order of the variables in iterated residues, usually,
  changes the result. In this case, however, because all the poles are
  normal crossing, formula \eqref{flagres} remains true no matter in
  what order we take the iterated residues.
\end{remark}

This together with \eqref{flagloc} and \eqref{alphasigmaf} gives
\begin{proposition}\label{propflag} Let $k\le n$ and $\alpha(\eta_1,\ldots, \eta_k)$ be a Chern polynomial in the Cherns roots of the tautological rank $k$ bundle $\cale$ over $\grass_k(\symdot)$. Then 
\begin{equation*}
\int_{\widetilde{\CHilb}^{k+1}(\CC^n)}\alpha(u)=\sires
\frac{\prod_{1\leq m<l\leq k}(z_m-z_l)\,\alpha_{\ff}(z_1, \ldots ,z_k)\dbz}
{\prod_{l=1}^k\prod_{i=1}^n(\lambda_i-z_l)}
\end{equation*}
\end{proposition}
To calculate $\alpha_\ff(z_1,\ldots ,z_k)=\int_{X_\ff} \a$ on the fiber
 \[X_\ff=\mu^{-1}(\ff)\simeq \overline{P_{k,n}\cdot p_{k,n}} \subset \flag_k(\sym^{\le k}\CC^n)\]
 we proceed a second equivariant localisation on $X_\ff$. Note that 
 \[X_\ff=\overline{\phi(\Hom^\ff(\CC^k,\CC^n)}\]
 where 
 \[\Hom^\ff(\CC^k,\CC^n)=\{\psi \in \Hom(\CC^k,\CC^n): \psi(e_i)\subset \CC_{[i]} \text{ for } i=1,\ldots, k\}\]
and $\CC_{[i]} \subset \CC^n$ is the subspace spanned by $e_1,\ldots, e_i$. The $\GL(n)$ action on $\Hom^\ff(\CC^k,\CC^n)$ reduces to $\GL(\CC_{[k]}) \subset \GL(n)$, and 
we perform a second $\GL(\CC_{[k]})$-equivariant localisation on $X_\ff$ to calculate $\a_\ff$. In the residue formula the Chern roots of the torus $T^k \subset \GL(\CC_{[k]})$ of $\cale$ on $X_\ff$ are  $z_1,\ldots, z_k$. 
 
We perform this second localisation on the blown-up fibers of $\mathrm{Jet}_k$. We arrive to this space after an iterated blow-up process encoded by a rooted blow-up tree, and the localization formula will give a sum over all fixed points, where fixed points correspond to leaves of the tree. For each leaf $L$ we will have a cluster of weights $\{z^L(\b): \b \in B\}$ associated to a fixed finite parametrising set $B$, and these will give the tangent weights of the blown-up space at this fixed point. For a leaf $L\in \call_k$, the image $\tilde{\phi}$ under the blown-up morphism is a fixed point on $\grass_k(\symdot)$ with tautological weights $z^L_1,\ldots, z^L_k$. Both the $z^L_i$ and $z^L(\b)$ weights are linear form in the $z_i$'s. We arrive at


\begin{proposition}\label{propint} Let $k\le n$ and $\alpha(\eta_1,\ldots, \eta_k)$ be a Chern polynomial in the Cherns roots of the tautological rank $k$ bundle over $\grass_k(\symdot)$. Let $\call_k$ denote the set of leaves in the blow-up tree. Then 
\begin{equation}\label{intnumberone} 
\int_{\mathrm{Jet}_k/\!/\diff_k}\tilde{\phi}^*\alpha=\sum_{L \in \call_k} \sires \frac{
(k-1)!z^{k-1}\prod_{m<l}(z_m-z_l) \alpha(z^L_1,\ldots ,z^L_k)}{
\Euler^L(\mathrm{Jet}_k)  \prod_{l=1}^k\prod_{i=1}^n(\lambda_i-z_l)} \,\dbz
  \end{equation}
where $\Euler^L(\mathrm{Jet}_k)=\prod_{\b \in B}z^L(\b)$ is the equivariant Euler class of the tangent space at $L$.  
\end{proposition} 
\begin{remark}\label{remark:explanation}
The formula \eqref{intnumberone} contains a few unexplained ingredients: the set of 'leaves' , and the weights $z_i,z_i^L$. Still, below we will be able to prove the residue vanishing theorem using only the following minimal information on these data:
\begin{enumerate}
\item $z_i^L, z^L(\b) \text{ are linear forms in } z_1,\ldots, z_k \text{ for all } 1\le i\le k, L\in \call_k, \b\in B$
\item Since $\tilde{\phi}$ is equivariant, is sends the $T^k$-fixed point $L$ on $\mathrm{Jet}_k$ to a $T^k$ fixed point on $\PP(\wedge^k \symdot)$. Assume that  $\tilde{\phi}(L)=e_\bipi=e_1 \wedge e_{\pi_2} \ldots \wedge e_{\pi_k} \in \PP(\wedge^k \symdot)$ is the torus fixed point corresponding to the sequence of partitions $\bipi=(\pi_1=(1),\pi_2,\ldots, \pi_k)$. Then 
\[z_i^L=z_{\pi_i}=\Sigma_{p \in \pi_i}z_p \text{ for } i=1,\ldots, k.\]
\item The cardinality of $B$ is $\dim(\mathrm{Jet}^\ff_k)=k(k+1)/2$ and therefore $\Euler^L(\mathrm{Jet}_k)$ is the product of $k(k+1)/2$ linear forms for all $L$.
\end{enumerate}
\end{remark}
 
\subsection{The residue vanishing theorem}\label{subsec:residuevanishing}

In this section we prove that all but one term on the right hand side of \eqref{intnumberone} vanish. This key feature of the iterated residue has already appeared in B\'erczi-Szenes \cite{bsz} but here we need to prove a slightly different version. We devote the rest of this section to the proof of the following theorem. 

\begin{theorem}[\textbf{(The Residue Vanishing Theorem)}]\label{vanishtheorem} Let $k\le n$ and $m \le 2n-(k+1)k/2$. Let
\[\mathrm{Euler}(\cale \otimes \CC^m)=\Pi_{j=1}^m\Pi_{i=1}^k (\theta_j+z_i)\]
be the Euler class in Theorem \ref{thm:thomtau}. Let $\call_k$ denote the set of leaves in the blow-up tree. Then all terms but the ones corresponding to leaves $L\in \call_k$ which are mapped to $\bgg=([1],[2],\ldots, [k]) \in \grass_k(\symdot)$ vanish in \eqref{intnumberone}. For these leaves $z^L_i=z_i$ for $i=1,\ldots k$,  leaving us with
\begin{equation}\label{intnumbertwo} 
\int_{\mathrm{Jet}_k/\!/\diff_k} \Euler(\cale \otimes \CC^m)=\sum_{L \in \tilde{\phi}^{-1}(\bgg)} \sires \frac{
(k-1)!z^{k-1}\prod_{m<l}(z_m-z_l) \Pi_{j=1}^m\Pi_{i=1}^k (\theta_j+z_i)}{
\Euler^L(\mathrm{Jet}_k)  \prod_{l=1}^k\prod_{i=1}^n(\lambda_i-z_l)} \,\dbz
  \end{equation}
where $\Euler^L(\mathrm{Jet}_k)=\prod_{\b \in B}z^L(\b)$ is the equivariant Euler class of the tangent space at $L$.   
\end{theorem} 
\begin{remark} Since the Thom polynomial $\Tp_k^{n,m}$ only depends on $k$ and $m-n$, the numerative conditions on these parameters in the Residue Vanishing Theorem do not affect our formula: for any fixed codimension $m-n$ it is possible to choose $n$ big enough so that these conditions are met. 
\end{remark}

For the preparation of the proof, we describe following \cite{bsz} $\S6.2$ the conditions under which iterated
residues  of the type appearing in the sum in \eqref{intnumberone}
vanish and we prove Theorem \ref{vanishtheorem}.

We start with the 1-dimensional case, where the residue at infinity
is defined by \eqref{defresinf} with $d=1$. By bounding the integral
representation along a contour $|z|=R$ with $R$ large, one can
easily prove the following lemma.
\begin{lemma}\label{1lemma}
  Let $p(z),q(z)$ be polynomials of one variable. Then
\[\res_{z=\infty}\frac{p(z)\,dz}{q(z)}=0\quad\text{if }\deg(p(z))+1<\deg(q).
\]
\end{lemma}
Consider now the multidimensional situation. Let $p(\bz),q(\bz)$ be
polynomials in the $k$ variables $z_1\ldots z_k$, and assume that
$q(\bz)$ is the product of linear factors $q=\prod_{i=1}^N L_i$, as
in \eqref{intnumberone}. We continue to use the notation $\dbz=dz_1\dots dz_k$.
We would like to formulate conditions under which the iterated
residue
\begin{equation}
  \label{ires}
\ires\frac{p(\bz)\,\dbz}{q(\bz)}
\end{equation}
vanishes. Introduce the following notation:
\begin{itemize}\label{notations}
\item For a set of indices $S\subset\{1\ldots k\}$, denote by $\deg(p(\bz);S)$
  the degree of the one-variable polynomial $p_S(t)$ obtained from $p$
  via the substitution $z_m\to
  \begin{cases}
t\text{ if }m\in S,\\ 1\text{ if }m\notin S.
  \end{cases}$. 
\item For a nonzero linear form $L=a_0+a_1z_1+\ldots +a_kz_k$,
  denote by $\coeff(L,z_i)=a_i$ the coefficient $a_i$.
 \item Finally, for $1\leq m\leq k$, set
  \[\lead(q(\bz);m)=\#\{i;\;\max\{l;\;\coeff(L_i,z_l)\neq0\}=m\}\]
which is the number of those factors $L_i$ in which the coefficient
of $z_m$ does not vanish, but the coefficients of $z_{m+1},\ldots,
z_k$ are $0$.
\end{itemize}
We can group the $N$ linear factors of $q(\bz)$ according to the
nonvanishing coefficient with the largest index; in particular, for
$1\leq m\leq k$ we have
\[   \deg(q(\bz);m)\geq\lead(q(\bz);m))  \text{ and } \sum_{m=1}^k\lead(q(\bz);m)=N.
\]

Now applying Lemma \ref{1lemma} to the first residue in
\eqref{ires}, we see that
\[ \res_{z_d=\infty}\frac{p(z_1,\ldots,z_{d-1},z_d)\dbz}{q(z_1,\ldots,z_{d-1},z_d)}=0
\]
whenever $\deg(p(\bz);d)+1<\deg(q(\bz),d)$; in this case, of course,
the entire iterated residue \eqref{ires} vanishes.

Now we suppose the residue with respect to $z_d$ does not vanish,
and we look for conditions of vanishing of the next residue:
\begin{equation}
  \label{2res}
\res_{z_{k-1}=\infty}\res_{z_k=\infty}\frac{p(z_1,\ldots,z_{k-2},z_{k-1},z_k)\dbz}
{q(z_1,\ldots,z_{k-2},z_{k-1},z_k)}.
\end{equation}
 Now the condition $\deg(p(\bz);k-1)+1<\deg(q(\bz),k-1)$ will
{\em insufficient};  for example,
\begin{equation}
\res_{z_{k-1}=\infty}\res_{z_k=\infty}\frac{dz_{k-1}dz_k}{z_{k-1}(z_{k-1}+z_k)}=
\res_{z_{k-1}=\infty}\res_{z_k=\infty}\frac{dz_{k-1}dz_k}{z_{k-1}z_k}
\left(1-\frac{z_{k-1}}{z_k}+\ldots\right)=1.
\end{equation}
After performing the expansions \eqref{omegaexp} to $1/q(\bz)$, we
obtain a Laurent series with terms $z_1^{-i_1}\ldots z_k^{-i_k}$ such that
$i_{k-1}+i_k\geq\mathrm{deg}(q(z);k-1,k)$, hence the condition
\begin{equation}\label{toprove}
\deg(p(\bz);k-1,k)+2<\deg(q(\bz);k-1,k)
\end{equation}
will suffice for the vanishing of \eqref{2res}. This argument easily
generalizes to the following statement.

 \begin{proposition}[\cite{bsz} Proposition 6.3]
  \label{vanishprop}
Let $p(\bz)$ and $q(\bz)$ be polynomials in the variables $z_1\ldots 
z_k$, and assume that $q(\bz)$ is a product of linear factors:
$q(\bz)=\prod_{i=1}^NL_i$; set $\dbz=dz_1\dots dz_k$. Then
\[ \ires\frac{p(\bz)\dbz}{q(\bz)} = 0
\]
if for some $l\leq k$, the following holds:
\[\deg(p(\bz);d,d-1,\dots,l)+d-l+1<\deg(q(\bz);d,d-1,\dots,l)\]
\end{proposition}
Note that the equality
$\deg(q(\bz);l)=\lead(q(\bz);l)$ means that
  \begin{equation}
    \label{op2cond}
\text{ for each }i=1\ldots N\text{ and }m>l)(
\coeff(L_i,z_{l})\neq0\text{ implies }\coeff(L_i,z_{m})=0.
  \end{equation}

We are ready to proof the Residue Vanishing Theorem. Recall that our goal is to show that all the terms of the sum in \eqref{intnumberone} vanish except for the ones corresponding to leaves L such that $\tilde{\phi}(L)=\bipi_\dist=e_1 \wedge \ldots \wedge e_k$. The plan is to apply Proposition \ref{vanishprop} in stages to show that the itrated residue vanishes unless $z^L_i=[i]$ holds, starting with $i=k$ and going backwards. 

Fix a sequence $\bipi=(\pi_1,\dots,\pi_k)\in\Bipi$, and consider the
iterated residue on the right hand side of
\eqref{intnumberone} corresponding to a leaf $L$ with $\tilde{\phi}(L)=e_\bipi$ The expression under the residue is the product of
two fractions:
\[\frac{p(\bz)}{q(\bz)}=\frac{p_1(\bz)}{q_1(\bz)}\cdot\frac{p_2(\bz)}{q_2(\bz)}\]
where
\begin{equation}
  \label{pq}
\frac{p_1(\bz)}{q_1(\bz)}=\frac{
(k-1)!z^{k-1}\,\prod_{m<l}(z_m-z_l) }{\prod_{\b \in B}z^L(\b)
}\text{ and
  }
\frac{p_2(\bz)}{q_2(\bz)} = \frac{
\Pi_{j=1}^m\Pi_{i=1}^k (\theta_j+z_{\pi_i})} {
\prod_{l=1}^k\prod_{i=1}^n(\lambda_i-z_l)}.
\end{equation}

Note that $p(\bz)$ is a polynomial, while $q(\bz)$ is a product of
linear forms.  As a first step we show that if $\pi_k \neq [k]$, then already the
first residue in the corresponding term on the right hand side of
\eqref{intnumberone} -- the one with respect to $z_k$ -- vanishes.
Indeed, if $\pi_k\neq[k]$, then $\deg(q_2(\bz);k)=n$, while
$z_k$ does not appear in $p_2(\bz)$. On the other hand, $\deg(p_1(\bz);k)=k-1$, hence
\begin{equation}
\deg(p(\bz);k)=k-1 \text{ and } \deg(q(\bz);k)\ge n
\end{equation}
and $k \le n$, so $\deg(p(\bz))\le \deg(q(\bz))-2$ holds and we can apply Lemma \ref{1lemma}.

We can thus assume that $\pi_k=[k]$, and proceed to the next step and take the residue 
with respect to $z_{k-1}$. Assume that $\pi_{k-1}\neq[k-1]$. Then  
\begin{equation}\label{k-1}
\deg(q_2(\bz),k-1,k)=2n, \deg(p_2(\bz);k-1,k)=m.
\end{equation}
and 
\begin{equation}\label{k}
\deg(q_1(\bz),k-1,k)\le \#B=(k+1)k/2, \deg(p_1(\bz);k-1,k)=2k-3.
\end{equation}
Hence for $m<2n+2k-5-(k+1)k/2$ we can apply Proposition \ref{vanishprop} with $l=k-1$ to deduce the vanishing of the residue with respect to $k-1$.

Continuing this argument in the same way provides an inductive proof of the the Residue Vanishing Theorem.

\subsection{Initial blow-up to achieve stability-semistability for the $\diff_k$ action}\label{subsection:initial}
Recall the rational map
\[\phi: \PP=\PP[\CC \oplus \Hom^\ff (\CC^k,\CC^n)] \dasharrow \PP(\wedge^k \symdot)\]
on the compactified fiber over $\mathbf{f}$ has reads as
\[[x:v_1,\ldots v_k] \mapsto [v_1 \wedge (xv_2+v_1^2) \wedge (x^2v_3+xv_1v_2+v_1^3) \wedge \ldots]\]
Here $v_1,\ldots, v_k \in \CC^n$ are vectors of the form \eqref{vectors} below, representing the columns of a matrix $M \in \Hom^\ff(\CC^k, \CC^n)$ and $x$ is the compactifying coordinate, while $\diff_k$ acts via the right action
\[ [x:M]\cdot \hat{u}=[x:M\hat{u}] \text{ for } \hat{u} \in \diff_k\]
or equivalently via the left action
\[\hat{u}\cdot [x:M]=[x:M(\hat{u})^{-1}] \text{ for } \hat{u} \in \diff_k.\]
We want to apply the results of non-reductive GIT described in $\S$ \ref{sec:nrgit}, which are stated for left actions. For this we need a one-parameter subgroup $\lambda:\CC^* \to \diff_k$ whose adjoint action on the Lie algebra of $U \subset \diff_k$ has only strictly positive weights; we can take
\[\lambda(t)=\left(
\begin{array}{ccccc}
t^{-1} &0   & 0          & \ldots & 0 \\
0        & t^{-2} & 0 & \ldots & 0 \\
0        & 0          & t^{-3}        & \ldots & 0 \\
0        & 0          & 0                 & \ldots & \cdot \\
\cdot    & \cdot   & \cdot    & \ldots &t^{-k}
\end{array}
 \right),
\]
and then 
\[\lambda(t)  \cdot [x:M]=[x:M  \left(
\begin{array}{ccccc}
t &0   & 0          & \ldots & 0 \\
0        & t^{2} & 0 & \ldots & 0 \\
0        & 0          & t^{3}        & \ldots & 0 \\
0        & 0          & 0                 & \ldots & \cdot \\
\cdot    & \cdot   & \cdot    & \ldots &t^{k}
\end{array}
 \right) ].
\]
 The weights of this (left)  action of the one-parameter subgroup of $\diff_k$ defined by $\lambda$  are $\{0, 1, 2, \ldots, k\}$. The minimal weight space is the point
\[Z_{\min}=\{ [1:0:\ldots :0] \} \]
and the $U$-stabiliser of this point is $
U$. Thus the non-reductive GIT blow-up process described at the end of $\S$3 starts with blowing up the projective space $\PP$ along $Z_{\min}$ to  get 
\[\tPP=\blow_{[1:0:\ldots :0]}\PP=\{([x:v_1,\ldots, v_k],[w_1,\ldots w_k]):w_i\otimes v_j=w_j\otimes v_i \mbox{ for }  1\le i<j\le k\}\]
embedded in $\PP^{{k \choose 2}} \times \PP^{{k \choose 2}-1} \subset \PP^{({k \choose 2}+1){k \choose 2}-1}$. Here 
\begin{align}\label{vectors}
v_1=& \b_{11}e_1 \\ \nonumber
v_2= & \b_{22}e_2+\b_{12}e_1\\ \nonumber
\cdots\\ \nonumber
v_k= &\b_{kk}e_k+\b_{{k-1}k}e_{k-1}+\ldots +\b_{1k}e_1 \nonumber
\end{align}
We fix the ample linearisation $L=\calo_{\PP^{{k \choose 2}}}(1) \otimes \calo_{\PP^{{k \choose 2}-1}}(1)$ on $\PP^{{k \choose 2}} \times \PP^{{k \choose 2}-1}$ and restrict it to $\tPP$. The minimal weight space for the action of $\lambda(\CC^*)$ on $\tPP$ is the intersection $\tZmin$ of the exceptional divisor $E$ and the strict transform of $\PP[x:v_1:0:\cdots:0] \subset \PP$:
\[\tZmin=\{([1:0:\ldots :0],[w_1:0:\ldots:0]): w_1 \in \CC^n,  w_1 \neq 0 \}\subset E \subset \tPP.\] 
and its contracting set is
\[\tPP_{\min}^0=\{([x:v_1:\ldots :v_k],[w_1:w_2:\ldots: w_k]): w_1 \neq 0, \text{ and not all of } w_2,\ldots, w_k \text{ is zero}\}\subset \tPP\]
The $U$-stabiliser of any point in $\tZmin$ is trivial, and hence stability coincides with semistability for the induced $\diff_k$ action on $\tPP$. This remains valid after blow-ups, hence semistability=stability holds for all spaces in this blow-up process.

Due to the form \eqref{vectors} of the vectors, on $\tPP^{ss}=\tPP_{\min}^0\setminus \tZmin$ we have $w_1=\b_{11}e_1$ and on the affine chart $U_0 \subset \tPP$, where $\b_{11} \neq 0$, we set $\b_{11}=1$. Note that $\tPP^{ss} \subset U_0$. We get relations $\b_{11} w_i=v_i$, and hence the affine coordinates on $U_0$ are $\b_{ij}$ for $1\le i\le j\le k$. We will use the notation $t=\b_{11}$ for this distinguished coordinate. The exceptional divisor is $\{t=0\} \subset U_0$ and the rational map  
\[\phi:\tPP^{ss} \dasharrow \PP(\wedge^k \symdot)\]
can be written as  
\begin{multline}\label{phiw}
\phi(t:w_1,\ldots w_k) = [te_1 \wedge (tw_2+(tw_1)^2) \wedge (tw_3+t^2w_1w_2+(tv_1)^3) \wedge \ldots]=\\
e_1 \wedge (w_2+tw_1^2) \wedge \ldots \wedge (\sum_{i_1+\ldots +i_r=k}t^{r-1}w_{i_1}\ldots w_{i_r})
\end{multline}

\section{Examples: Thom polynomials for $k\le 5$}\label{sec:examples}
In this section we work out the strategy for small examples, and we compute Thom polynomials of $A_3, A_4$ and $A_5$ singularities to demonstrate the key ingredients of out new approach. 
\subsection{Thom polynomials of $A_3$ singularities}\label{subsec:k=3}
The affine coordinates over the flag $\ff=(e_1,e_2,e_3) \in \flag_3(\CC^n)$ are
\begin{align*}
w_1=& e_1 \\
w_2= &\b_{22}e_2+\b_{12}e_1\\
w_3= &\b_{33}e_3+\b_{23}e_2+\b_{13}e_1
\end{align*}
We sort the map $\phi_3:U_0=\mathrm{Spec}(\CC[t,\b_{ij}]) \to \PP(\sym^{\le 3}\CC^n)$ by the degree of $t$:
\begin{align}\label{mapk=3}\nonumber
\phi_3(t,w_1, w_2, w_3)=& w_1 \wedge (w_2+tw_1^2) \wedge (w_3 + tw_1w_2+ t^2w_1^3)= \\ \nonumber
& [\b_{22}\b_{33} e_1 \wedge e_2 \wedge e_3] +\\ \nonumber
t & [\b_{33}e_1 \wedge e_1^2 \wedge e_3 + (\b_{23}-\b_{12}\b_{22}) e_1 \wedge e_1^2 \wedge e_2+ \b_{22}^2 e_1 \wedge e_2 \wedge e_1e_2]+\\
t^2 & [\b_{22} e_1 \wedge e_2 \wedge e_1^3 + \b_{22} e_1 \wedge e_1^2 \wedge e_1e_2]+\\ \nonumber
t^3 & [e_1 \wedge e_1^2 \wedge e_1^3] \nonumber
\end{align}

The ideal of indeterminacy locus of $\phi_3$ is generated by the coefficients:
\[I(\phi_3)=(\b_{22}\b_{33}, \b_{33}t, \b_{22}^2t, (\b_{23}-\b_{12}\b_{22})t, \b_{22}t^2, t^3)\] 
We blow-up along the indeterminacy locus of the monomial map 
\[\phi_3^{mon}(t,w_1, w_2, w_3)=[\b_{22}\b_{33}, \b_{33}t, \b_{22}^2t, \b_{23}t, \b_{12}\b_{22}t, \b_{22}t^2, t^3] \in \PP^6\]
formed by all monomials in the $\phi_3$. The indeterminacy locus is $\mathrm{Spec}(M(\phi_3))$ with 
\[M(\phi_3)=(\b_{22}\b_{33}, \b_{33}t, \b_{22}^2t, \b_{23}t, \b_{12}\b_{22}t, \b_{22}t^2, t^3] \in \PP^6)\]
is generated by the coordinates all monomials in $\phi_3$. Now $\mathrm{Spec}(M(\phi_3))$ is the union of coordinate subspaces with multiplicities, determined by the primary decomposition of the monomial ideal $M(\phi_3)$:
\[M(\phi_3)=(t,\b_{22}) \cap (t,b_{33}) \cap (t^3,\b_{22},\b_{23},\b_{33}) \cap (t^2,\b_{12},\b_{22}^2,\b_{23},\b_{33})\]
We proceed by blowing up along one of the components which is i) reduced for the blow-up to be smooth ii) $\diff_3$-invariant and iii) of maximal dimension. Then we choose the affine chart corresponding to the minimal weight space. We pick $(\underline{t},\b_{33})$, and the affine chart $t\neq 0$ as the only choice as the weight of $t$ is $1$ and the weight of $\b_{33}$ is $2$. Let $\tilde{\phi}_3:\blow_{(t,b_33),t}U_0 \to \PP(\sym^{\le 3}\CC^n)$ denote the blown-up rational map, then its monomial ideal is
\[M(\tilde{\phi}_3)=(t\b_{22}\tilde{\b}_{33}, t^2\tilde{\b}_{33}, t\b_{22}^2, t\b_{23}, t\b_{12}\b_{22}, t^2\b_{22}, t^3)\]
which defines the same monomial map to $\PP(\sym^{\le 3}\CC^n)$ as the one by dividing all monomials by $t$, and we take the primary decomposition again:
\[M(\tilde{\phi}_3)=(\b_{22}\tilde{\b}_{33}, t\tilde{\b}_{33}, \b_{22}^2, \b_{23}, \b_{12}\b_{22}, t\b_{22}, t^2)=(t,\b_{22},\b_{23})\cap (t^2,\b_{22},\b_{23},\tilde{\b}_{33}) \cap (t,\b_{12},\b_{22}^2,\b_{23},\tilde{\b}_{33})\]
Then we pick the ideal $(t,\b_{22},\b_{23})$ and note that the weight of $t$ and $\b_{22}$ is $1$ and the weight of $\b_{23}$ is $2$. So we have two minimal weight charts, and the blow-up maps have the following monomial ideals:
\begin{align*}
M(\tilde{\phi}_3^{\b_{22}})&= (\tilde{\b}_{33}, \tilde{t}\tilde{\b}_{33}, \b_{22}, \tilde{\b}_{23}, \b_{12}, \tilde{t}\b_{22}, \tilde{t}^2\b_{22})=(\tilde{\b}_{33},\b_{22}, \tilde{\b}_{23}, \b_{12})\\
M(\tilde{\phi}_3^{t})&=(\tilde{\b}_{22}\tilde{\b}_{33}, \tilde{\b}_{33}, t\tilde{\b}_{22}^2, \tilde{\b}_{23}, \b_{12}\tilde{\b}_{22}, t\tilde{\b}_{22}, t)
\end{align*}
On the first affine chart our blow-up process terminates: there is only one reduced component, and we can write the blown-up map on this chart as
\begin{align*}
\tilde{\tilde{\phi}}_3(\tilde{t},\b_{12},\b_{13},\b_{22},\tilde{\b}_{23},\tilde{\b}_{33})= \tilde{\b}_{33} e_1 \wedge e_2 \wedge e_3 +
\tilde{t}\tilde{\b}_{33}e_1 \wedge e_1^2 \wedge e_3 + (\tilde{\b}_{23}-\b_{12}) e_1 \wedge e_1^2 \wedge e_2 + \\
+\b_{22} e_1 \wedge e_2 \wedge e_1e_2+
+\tilde{t}\b_{22} e_1 \wedge e_2 \wedge e_1^3 + \tilde{t}\b_{22} e_1 \wedge e_1^2 \wedge e_1e_2+
\tilde{t}^2\b_{22}e_1 \wedge e_1^2 \wedge e_1^3
\end{align*}
All coefficients are linear combinations of the generators of $M(\tilde{\phi}_3^{\b_{22}})$, and the indeterminacy ideal of this map is (which is the ideal generated by the coefficients) is
\[(\tilde{\b}_{33},\b_{22}, \tilde{\b}_{23}-\b_{12})\]
which is not a monomial ideal, but a complete intersection, hence the blown-up at this is smooth, and $\phi_3$ is well defined on the blown-up. 

It remains to check the other affine chart, where $t\neq 0$, but in fact the points which are not covered by the $\b_{22}\neq 0$ affine chart is the locus $\tilde{\b}_{22}=0$. The map on this set has the form
\begin{align*}
\phi_3(t,\b_{12},\b_{13},\tilde{\b}_{22},\tilde{\b}_{23},\tilde{\b}_{33})= \tilde{\b}_{33}e_1 \wedge e_1^2 \wedge e_3 + \tilde{\b}_{23} e_1 \wedge e_1^2 \wedge e_2 +
te_1 \wedge e_1^2 \wedge e_1^3
\end{align*}
The image does not contain the point $e_1\wedge e_2 \wedge e_3$, hence by the Residue Vanishing Theorem the contribution to the residue formula of this part is zero. 
In short, the map $\phi_3$ is well-defined on the $\diff_3$-semistable=stable  set of the iterated blown-up space 
\[\blow_{(\tilde{\b}_{33},\b_{22}, \tilde{\b}_{23}-\b_{12})}\blow_{(t,\b_{22},\b_{23})}\blow_{(t,\b_{33})}U_0\]
The blow-up process can be summarized in the following blow-up tree.

\begin{figure}[h]
{\tiny 
\centering
\begin{tikzpicture}[node distance=2cm]
\node (I) [boxes] {$(t^1,\b_{33}^2)$};
\node (II) [boxes, below of=I,yshift=1cm] {$(t^1,\b_{22}^1,\b_{23}^2)$};
\node (nocontrII) [nocontr, left of=II,xshift=0cm ] {No contr.};
\node (III) [boxes, below of=II,yshift=1cm] {$(\b_{22}^1,\b_{23}^1-\b_{12}^1,\b_{33}^1)$};
\node (contrIII) [contr, below of=III,yshift=1cm] {Contr.};

\draw [arrow] (I) -- node[anchor=east] {$t^1$} (II);
\draw [arrow] (II) -- node[anchor=east] {$\b_{22}^1$} (III);
\draw [arrow] (II) -- node[anchor=south] {$t^1$} (nocontrII);
\draw [arrow] (III) -- node[anchor=east] {$\b_{33}^0$} (contrIII);
\end{tikzpicture}
}
\caption{The blow-up tree for $k=3$}
\end{figure}
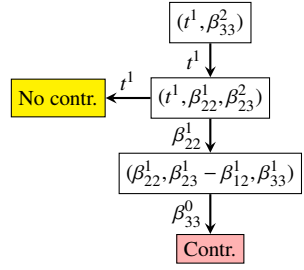

The $\CC^* \times \GL(3)$ weights of the affine coordinates change after each blow-up as follows:
\begin{align*}
\text{ Initial weights: }  &  \g(t)=z+z_1,\  \g(\b_{ij})=(j-1)z+z_i-z_1\\
\text{After first blow-up: } & \g(\tilde{\b}_{33})=\g(\b_{33})-\g(t)=2z+z_3-z_1-(z+z_1)=z+z_3-2z_1\\
\text{After second blow-up: } & \g(\tilde{\b}_{23})=\g(\b_{23})-\g(\b_{22})=z\\
& \g(\tilde{t})=\g(t)-\g(\b_{22})=z+z_1-(z+z_2-z_1)=z_2-2z_1\\
\text{After final blow-up: } & \g(\tilde{\tilde{\b}}_{23})=\g(\tilde{\b}_{23})-\g(\tilde{b}_{33})=2z_1-z_3\\
& \g(\tilde{\b}_{22})=\g(\b_{22})-\g(\tilde{\b}_{33})=(z+z_2-z_1)-(z+z_3-2z_1)=z_1+z_2-z_3
\end{align*}
The following table collects the weights in the process:

{\tiny \begin{tabular}{l|c|c|c|c|c|c|}
Ideal & $t$ & $\b_{12}$ & $\b_{13}$ & $\b_{22}$ & $\b_{23}$ & $\b_{33}$ \\
\hline 
-  &  $z+z_1$ & $z$ & $2z$ & $z+z_2-z_1$ & $2z+z_2-z_1$ & $2z+z_3-z_1$ \\
\hline
$(\underline{t},\b_{33})$ & $z+z_1$ & $z$ & $2z$ & $z+z_2-z_1$ & $2z+z_2-z_1$ & \makecell{$(2z+z_3-z_1)-(z+z_1)$\\ $=z+z_3-2z_1$}\\
\hline
$(\underline{t},\b_{22},\b_{23})$ & \makecell{$(z+z_1)-(z+z_2-z_1)$\\ $=2z_1-z_2$} & $z$ & $2z$ & $z+z_2-z_1$ & \makecell{$(2z+z_2-z_1)-$\\$(z+z_2-z_1)=z$} & $z+z_3-2z_1$ \\
\hline
$(\b_{22},\b_{23},\underline{\b}_{33})$ & $2z_1-z_2$ & $z$ & $2z$ & \makecell{$(z+z_2-z_1)-(z+z_3-2z_1)$\\$=z_1+z_2-z_3$} & \makecell{$z-(z+z_3-2z_1)$\\=$2z_1-z_3$} & $z+z_3-2z_1$
\end{tabular}}

The residue formula has the form 
\begin{multline}\nonumber
\Tp_3=\res_{z_1=\infty}\res_{z_2=\infty}\res_{z_3=\infty}\res_{z=\infty}\frac{2z^2(z_1-z_2)(z_1-z_3)(z_2-z_3)\prod_{i=1}^N\prod_{j=1}^3(\theta_i-z_j)d\bz}{z(2z)(2z_1-z_2)(2z_1-z_3)(z_1+z_2-z_3)(z+z_3-2z_1)\prod_{i=1}^n\prod_{j=1}^3(\lambda_i-z_j)}=\\
=\res_{z_1=\infty}\res_{z_2=\infty}\res_{z_3=\infty}\frac{(z_1-z_2)(z_1-z_3)(z_2-z_3)d\bz}{(2z_1-z_2)(2z_1-z_3)(z_1+z_2-z_3)}\prod_{i=1}^3 c_{TM-TN}(1/z_i) 
\end{multline}
which gives back the formula in \cite{bsz}.

\subsection{Thom polynomials of $A_4$ singularities}\label{subsec:k=4} The affine coordinates over the flag $\ff=(e_1,e_2,e_3,e_4) \in \flag_4(\CC^n)$ are
\begin{align*}
w_1=& e_1 \\
w_2= &\b_{22}e_2+\b_{12}e_1\\
w_3= &\b_{33}e_3+\b_{23}e_2+\b_{13}e_1\\
w_4= &\b_{44}e_4+\b_{34}e_3+\b_{24}e_2+\b_{14}e_1
\end{align*}
The rational map $\phi: U_0=\mathrm{Spec}(\CC[t,\b_{ij}]) \to \PP(\sym^{\le 4}\CC^n)$ is given as
{\tiny \begin{align*}
\phi(t,w_1,\ldots, w_4) = & w_1 \wedge w_2 \wedge w_3 \wedge w_4 + \\
t &[(w_1 \wedge w_1^2 \wedge w_3 \wedge w_4) + (w_1 \wedge w_2 \wedge w_1w_2 \wedge w_4) + (w_1 \wedge w_2 \wedge w_3 \wedge (w_2^2+w_1w_3))]+\\
t^2 & [(w_1 \wedge w_2 \wedge w_3 \wedge w_1^2w_2) + (w_1 \wedge w_2 \wedge w_1w_2 \wedge (w_1w_3 +w_2^2))+(w_1 \wedge w_2 \wedge w_1^3 \wedge w_4)+ \\
& +(w_1 \wedge w_1^2 \wedge w_3 \wedge (w_1w_3 +w_2^2))+(w_1 \wedge w_1^2 \wedge w_1w_2 \wedge w_4)]+\\
t^3 & [(w_1 \wedge w_1^2 \wedge w_1^3 \wedge w_4) +( w_1 \wedge w_1^2 \wedge w_1w_2 \wedge (w_1w_3+w_2^2))+\\
& +(w_1 \wedge w_1^2 \wedge w_3 \wedge w_1^2w_2)+(w_1 \wedge w_2 \wedge w_1w_2 \wedge w_1^2w_2)]+\\
t^4 & [(w_1 \wedge w_1^2 \wedge w_1^3 \wedge (w_1w_3+w_2^2)) + (w_1 \wedge w_1^2 \wedge w_1w_2+ \wedge w_1^2w_2)+(w_1\wedge w_2 \wedge w_1^3 \wedge w_1^2w_2) ]+\\
t^5 & [(w_1 \wedge w_1^2 \wedge w_1^3 \wedge w_1^2w_2) ]+\\
t^6 & [w_1 \wedge w_1^2 \wedge w_1^3 \wedge w_1^4]
\end{align*}}
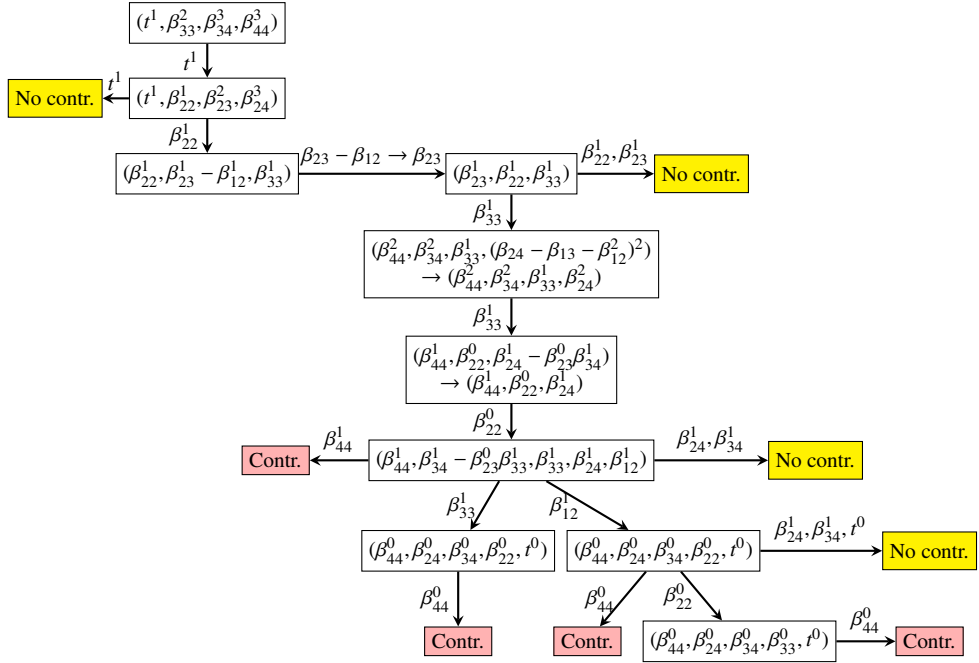
\begin{figure}[h!]
\centering
{\tiny \begin{tikzpicture}[node distance=2cm]
\node (II) [boxes] {$(t^1,\b_{33}^2,\b_{34}^3,\b_{44}^3)$};
\node (III) [boxes, below of=II,yshift=1cm] {$(t^1,\b_{22}^1,\b_{23}^2,\b_{24}^3)$};
\node (nocontrIII) [nocontr, left of=III,xshift=0cm ] {No contr.};
\node (IV) [boxes, below of=III,yshift=1cm] {$(\b_{22}^1,\b_{23}^1-\b_{12}^1,\b_{33}^1)$};
\node (V) [boxes, right of=IV, xshift=2cm] {$(\b_{23}^1,\b_{22}^1,\b_{33}^1)$};
\node (nocontrV) [nocontr, right of=V,xshift=0.5cm] {No contr.};
\node (VI) [boxes, below of=V, yshift=0.8cm] {\makecell{$(\b_{44}^2,\b_{34}^2,\b_{33}^1,(\b_{24}-\b_{13}-\b_{12}^2)^2)$ \\ $\to (\b_{44}^2,\b_{34}^2,\b_{33}^1,\b_{24}^2)$}};
\node (VII) [boxes, below of=VI, yshift=0.6cm] {\makecell{$(\b_{44}^1,\b_{22}^0,\b_{24}^1-\b_{23}^0\b_{34}^1)$ \\ $\to (\b_{44}^1,\b_{22}^0,\b_{24}^1)$}};
\node (VIII) [boxes, below of=VII,yshift=0.8cm] {$(\b_{44}^1,\b_{34}^1-\b_{23}^0\b_{33}^1,\b_{33}^1,\b_{24}^1,\b_{12}^1)$};
\node (nocontrVIII) [nocontr, right of=VIII,xshift=2cm] {No contr.};
\node (IX) [boxes, below of=VIII,xshift=-0.7cm,yshift=0.8cm] {$(\b_{44}^0,\b_{24}^0,\b_{34}^0,\b_{22}^0,t^0)$};
\node (contrIX) [contr, below of=IX,yshift=0.8cm] {Contr.};
\node (contrVIII) [contr, left of=VIII,xshift=-1.1cm ] {Contr.};
\node (X) [boxes, below of=VIII,xshift=2cm,yshift=0.8cm] {$(\b_{44}^0,\b_{24}^0,\b_{34}^0,\b_{22}^0,t^0)$};
\node (nocontrX) [nocontr, right of=X,xshift=1.5cm] {No contr.};
\node (XI) [boxes, below of=X,xshift=1cm,yshift=0.8cm] {$(\b_{44}^0,\b_{24}^0,\b_{34}^0,\b_{33}^0,t^0)$};
\node (contrX) [contr, below of=X,xshift=-1cm,yshift=0.8cm] {Contr.};
\node (contrXI) [contr, right of=XI,xshift=0.5cm] {Contr.};

\draw [arrow] (II) -- node[anchor=east] {$t^1$} (III);
\draw [arrow] (III) -- node[anchor=east] {$\b_{22}^1$} (IV);
\draw [arrow] (III) -- node[anchor=south] {$t^1$} (nocontrIII);
\draw [arrow] (IV) -- node[anchor=south] {$\b_{23}-\b_{12}  \to \b_{23}$} (V);
\draw [arrow] (V) -- node[anchor=south] {$\b_{22}^1,\b_{23}^1$} (nocontrV);
\draw [arrow] (V) -- node[anchor=east] {$\b_{33}^1$} (VI);
\draw [arrow] (VI) -- node[anchor=east] {$\b_{33}^1$} (VII);
\draw [arrow] (VII) -- node[anchor=east] {$\b_{22}^0$} (VIII);
\draw [arrow] (VIII) -- node[anchor=south] {$\b_{24}^1,\b_{34}^1$} (nocontrVIII);
\draw [arrow] (VIII) -- node[anchor=east] {$\b_{33}^1$} (IX);
\draw [arrow] (IX) -- node[anchor=east] {$\b_{44}^0$} (contrIX);
\draw [arrow] (VIII) -- node[anchor=south] {$\b_{44}^1$} (contrVIII);
\draw [arrow] (VIII) -- node[anchor=east] {$\b_{12}^1$} (X);
\draw [arrow] (X) -- node[anchor=south] {$\b_{24}^1,\b_{34}^1,t^0$} (nocontrX);
\draw [arrow] (X) -- node[anchor=east] {$\b_{44}^0$} (contrX);
\draw [arrow] (X) -- node[anchor=east] {$\b_{22}^0$} (XI);
\draw [arrow] (XI) -- node[anchor=south] {$\b_{44}^0$} (contrXI);

\end{tikzpicture}}
\caption{The blow-up tree for $k=4$}\label{tree4}
\end{figure}
Figure \ref{tree4} summarises the blowing up process. We make some general explanatory comments before we describe the first few steps of this procedure in detail.
\begin{itemize}
\item The initial $\CC^*$-weight of $t$ is $1$, and the weight of $\b_{ij}$ is $j-1$. The upper index of the variables indicate the $z$-weight at each stage.
\item At each stage, after suitable change of affine coordinates if necessary, we perform a blow-up of an affine space along a coordinate subspace. The ideal of this subspace is generated by a cluster of affine coordinates, and vertices of the tree are labeled by these cluster. In general let $x_1,\ldots, x_m$ be coordinates on the affine space $\AAA^m$.  For a non-empty cluster $C \subset \{1,\ldots ,m\}$ and an element $i\in C$ we define the map 
\[\pi_{C,i}: \AAA^m \to \AAA^m\] 
by 
\[\pi_{C,i}^*(x_j)=\begin{cases} x_j & j \notin C \setminus \{i\} \\ x_ix_j & j \in C \setminus \{i\} \end{cases}\]
Let $\AAA_C=\mathrm{Spec}(x_j:j\in C) \subset \AAA^m$ be the coordinate subspace defined by the equations $x_j=0, j\in C$; we have
$\pi_{C,i}^{-1}(\AAA_C)=\AAA_{{i}}$ and 
\[\pi_{C,i}: \AAA^m \setminus \AAA_{{i}} \to \AAA^m \setminus \AAA_C\] is an isomorphism. This is the restriction of the blow-up morphism $\pi_{C}: \blow_{C}\AAA^m \to \AAA^m$ to the ith coordinate affine chart $\blow_{C,i}\AAA^m$. We will keep the same notation $x_i$ for the affine coordinate $\pi_{C,i}^*(x_i)$ on $\blow_{C,i}\AAA^m$, keeping in mind the transformation. In the blowing-up process summarised in the blow-up tree above, in each step we choose a cluster $C \subset \{t,\b_{ij}\}$ and a $c \in C$, and look at the blowing up of the map $\phi$ on the affine chart $\blow_{C,c}\AAA^m$. The node is labeled by the cluster $C$, and for each $c \in C$ with minimal $\CC^*$-weight, there is an edge from $C$ labeled by $c$.  
\item At each step of the blowing-up process we need to make sure that 
\begin{enumerate}
\item We blow-up at $\diff_k$-invariant center, that is, the ideal $I=(C)$ generated by the cluster $C \subset \{t,\b_{ij}\}$ is invariant under the $\diff_k$-action. Note that the coordinates which generate the ideal are not necessarily $\diff_k$-invariant, but the ideal must be. 
\item The chosen $c \in C$ belongs to a coordinate in the minimal weight space $Z_{\min}$, that is, it has minimal $\CC^*$-weight. The weights of the coordinates after each blow up change following the simple transformation $\pi_{C,i}^*$ described in Eq. \eqref{changeofweights}
 \end{enumerate}
\item Torus fixed points correspond to the leaves of the tree. However, not all leaves have contribution to the final residue formula, as some of these sit in the kernel of the residue operator. Due to the residue vanishing theorem, which we prove in the next section, only those fixed points have nonzero contribution to the localisation formula, which are mapped by $\tilde{\phi}$ to the point $e_1\wedge \ldots \wedge e_k$.  These leaves are colored red, and we have 4 of them in Figure \ref{tree4}. However, two of these have zero residue due to simple degree reasons explained below. Finally, the yellow 'No contr' boxes show a branch (i.e affine chart) of the tree whose leaves are not mapped to $e_1\wedge \ldots \wedge e_k$, as explained below. 
\end{itemize}
Note that (2) and (3) are two very powerful, crucial requirements, which says that we have to worry only about a small portion of the leaves, namely those, which contain a fixed point mapped to $e_1\wedge \ldots \wedge e_k$.

We start with the ideal $I_1=(t^1,\b_{33}^2,\b_{34}^3,\b_{44}^3)$, corresponding to the coordinate subspace defined by the cluster $C_1=\{t,\b_{33},\b_{34},\b_{44}\}$. The only minimal weight variable is $c=t$ whose $\CC^*$-weight $1$. After the first blow-up each coordinate of the rational map $\tilde{\phi}_1:\blow_{C_1,t}\Spec(\CC[t,\b_{i}]) \to \PP(\wedge^4 \sym^{\le 4}\CC^n)$ can be divided by $t^3$.

Next we blow up $\AAA^m\simeq \blow_{C_1,t}\Spec(\CC[t,\b_{i}])$ at $I_2=(t^1, \b_{22}^1, \b_{23}^2,\b_{24}^3)$, because $I_2$ is a $\diff_4$-invariant subspace of the indeterminacy locus of the monomial map $\tilde{\phi}^{mon}$. Here we get two minimal weight coordinates in $C_2=\{t^1, \b_{22}^1, \b_{23}^2,\b_{24}^3\}$, namely $t$ and $\b_{22}$, corresponding to two edges of the tree, and we need to study these two branches separately. \\
\underline{$c=t$:} Look at the coefficients of $e_1\wedge \ldots \wedge e_4$ and $e_1 \wedge e_1^2 \wedge e_3 \wedge e_4$. These are transformed under the first two blow-up as
\[(\pi_{C_2,t}^* \circ \pi_{C_1,t}^*)(\b_{22}\b_{33}\b_{44})=t \b_{22}\b_{33}\b_{44} \text{ and }
(\pi_{C_2,t}^* \circ \pi_{C_1,t}^*)(t\b_{33}\b_{44})=t^3 \b_{33}\b_{44}\]
All coefficients (not just these two) of the blown-up map $\tilde{\phi}_2$ are divisible by $t^3$, hence we can divide by this all projective coordinates. Since the first is a multiple of the second, on this chart it is not possible to single-out the coefficient of $e_1\wedge \ldots \wedge e_4$, hence this chart (and its blow-ups) will not contain fixed points mapped to $e_1\wedge \ldots \wedge e_4$. By the Residue Vanishing Theorem this affine chart has zero contribution to the residue formula.\\ 
\underline{$c=\b_{22}$:}  The affine chart is $\blow_{C_2,c}\AAA^m$, the coefficients of the blown-up map $\tilde{\phi}_2:\blow_{C_2,c}\AAA^m \to \PP(\wedge^4 \sym^{\le 4}\CC^n)$ are divisible by $(t^{(1)})^2 b_{22}$, hence we can divide by this all projective coordinates of $\tilde{\phi}_2$.

Then after a change of affine coordinate $\b_{23}\to \b_{23}-\b_{12}$ (we do not change the other coordinates) we continue with $C_3=\{\b_{22}^1,\b_{23}^1,\b_{33}^1\}$, which is $\diff_4$-invariant component of the monomial ideal $M(\tilde{\phi}_2^{mon})$, generated by the monomials of $\tilde{\phi}_2$. All variables have minimal weight $1$, but only the $c=\b_{33}$ chart can contain points mapped to $e_1\wedge \ldots \wedge e_4$.  Indeed:\\ 
\underline{$c=\b_{22}$:} Look at the coefficients of $e_1\wedge \ldots \wedge e_5$ and $e_1 \wedge e_2 \wedge e_1e_2 \wedge e_4 \wedge e_5$. These are transformed under the first $3$ blow-up to 
\[(\pi_{C_3,\b_{22}}^* \circ \pi_{C_2,\b_{22}}^* \circ \pi_{C_1,t}^*)(\b_{33}\b_{44})=\b_{22}\b_{33}\b_{44}  \text{  and  } (\pi_{C_3,\b_{22}}^* \circ \pi_{C_2,\b_{22}}^* \circ \pi_{C_1,t}^*)(\b_{22}\b_{44})=\b_{22}\b_{44}\]
All coefficients (not just these two) of the blown-up map $\tilde{\phi}_3$ are divisible by $\b_{22}$, hence we can divide by this all projective coordinates. Again, the first is a multiple of the second, hence this chart (and its blow-ups) will not contain fixed points mapped to $e_1\wedge \ldots \wedge e_5$. \\
\underline{ $c=\b_{23}$:} Similar argument shows that this affine coordinate does not contribute to the residue. Hence we continue with $c=\b_{33}$ and proceed with the blow-up algorithm following the blow-up tree in Figure \ref{tree4}.

The following table collects the weights in the process:

\begin{table}[ht]
\centering
\scalebox{0.49}{
\begin{tabular}{l|c|c|c|c|c|c|c|c|}
Ideal & $t$ & $\b_{12}$ & $\b_{22}$ & $\b_{23}$ & $\b_{24}$ & $\b_{33}$ & $\b_{34}$ & $\b_{44}$\\
\hline \hline
-  &  $\wt_t$ & $\wt_{12}$ &$\wt_{22}$ & $\wt_{23}$ & $\wt_{24}$ & $\wt_{33}$ & $\wt_{34}$ & $\wt_{44}$\\
\hline
$(\underline{t},\b_{44})$ & $\wt_t$ & $\wt_{12}$ & $\wt_{22}$ & $\wt_{23}$ & $\wt_{24}$ & $\wt_{33}$ & $\wt_{34}$ & $\wt_{44}-\wt_t$ \\
\hline
$(\underline{t},\b_{33},\b_{34})$ & $\wt_t$ & $\wt_{12}$ & $\wt_{22}$ & $\wt_{23}$ & $\wt_{24}$ & $\wt_{33}-\wt_t$ & $\wt_{34}-\wt_t$ & $\wt_{44}-\wt_t$\\
\hline
$(t^1, \underline{\b}_{22}^1,\b_{23}^2,\b_{24}^3)$ & $\wt_t-\wt_{22}$ & $\wt_{12}$ & $\wt_{22}$ & $\wt_{23}-\wt_{22}$ & $\wt_{24}-\wt_{22}$ & $\wt_{33}-\wt_t$ & $\wt_{34}-\wt_t$ & $\wt_{44}-\wt_t$\\
\hline 
$(\b_{23}^1,\b_{22}^1,\underline{\b}_{33}^1)$ & $\wt_t-\wt_{22}$ & $\wt_{12}$ & $\wt_{22}-\wt_{33}+\wt_t$ & $\wt_{23}-\wt_{22}-\wt_{33}+\wt_t$ & $\wt_{24}-\wt_{22}$ & $\wt_{33}-\wt_t$ & $\wt_{34}-\wt_t$ & $\wt_{44}-\wt_t$\\
\hline
$(\b_{44}^2,\b_{34}^2,\underline{\b}_{33}^1,\b_{24}^2)$ & $\wt_t-\wt_{22}$ & $\wt_{12}$ & $\wt_{22}-\wt_{33}+\wt_t$ & $\wt_{23}-\wt_{22}-\wt_{33}+\wt_t$ & $\wt_{24}-\wt_{22}-\wt_{33}+\wt_t$ & $\wt_{33}-\wt_t$ & $\wt_{34}-\wt_{33}$ & $\wt_{44}-\wt_{33}$\\
\hline 
$(\b_{44}^1,\underline{\b}_{22}^0,\b_{24}^1)$ & $\wt_t-\wt_{22}$ & $\wt_{12}$ & $\wt_{22}-\wt_{33}+\wt_t$ & $\wt_{23}-\wt_{22}-\wt_{33}+\wt_t$ & $\wt_{24}-2\wt_{22}$ & $\wt_{33}-\wt_t$ & $\wt_{34}-\wt_{33}$ & $\wt_{44}-\wt_{22}-\wt_t$\\
\hline 
\rowcolor{red!30}
$(\underline{\b}_{44}^1,\b_{34}^1,\b_{33}^1,\b_{24}^1,\b_{12}^1)$ & $\wt_t-\wt_{22}$ & $\wt_{12}-\wt_{44}+\wt_{22}+\wt_t$ & $\wt_{22}-\wt_{33}+\wt_t$ & $\wt_{23}-\wt_{22}-\wt_{33}+\wt_t$ & $\wt_{24}-\wt_{22}-\wt_{44}+\wt_t$ & $\wt_{33}-\wt_{44}+\wt_{22}$ & $\wt_{34}-\wt_{33}-\wt_{44}+\wt_{22}+\wt_t$ & $\wt_{44}-\wt_{22}-\wt_t$\\
\hline
$(\b_{44}^1,\b_{34}^1,\underline{\b}_{33}^1,\b_{24}^1,\b_{12}^1)$ & $\wt_t-\wt_{22}$ & $\wt_{12}-\wt_{33}+\wt_t$ & $\wt_{22}-\wt_{33}+\wt_t$ & $\wt_{23}-\wt_{22}-\wt_{33}+\wt_t$ & $\wt_{24}-2\wt_{22}-\wt_{33}+\wt_t$ & $\wt_{33}-\wt_t$ & $\wt_{34}-2\wt_{33}+\wt_t$ & $\wt_{44}-\wt_{22}-\wt_{33}$\\
\hline
$(\b_{44}^1,\b_{34}^1,\b_{33}^1,\b_{24}^1,\underline{\b}_{12}^1)$ & $\wt_t-\wt_{22}$ & $\wt_{12}$ & $\wt_{22}-\wt_{33}+\wt_t$ & $\wt_{23}-\wt_{22}-\wt_{33}+\wt_t$ & $\wt_{24}-2\wt_{22}-\wt_{12}$ & $\wt_{33}-\wt_{12}-\wt_t$ & $\wt_{34}-\wt_{33}-\wt_{12}$ & $\wt_{44}-\wt_{22}-\wt_{12}-\wt_t$\\
\hline
\rowcolor{red!30}
$(\underline{\b}_{44}^0,\b_{24}^0,\b_{34}^0,\b_{22}^0,t^0)$ (a) & $\wt_t-\wt_{44}+\wt_{33}$ & $\wt_{12}-\wt_{33}+\wt_t$ & $2\wt_{22}+\wt_t-\wt_{44}$ & $\wt_{23}-\wt_{22}-\wt_{33}+\wt_t$ & $\wt_{24}-\wt_{22}-\wt_{44}+\wt_t$ & $\wt_{33}-\wt_t$ & $\wt_{34}-\wt_{33}-\wt_{44}+\wt_{22}+\wt_t$ & $\wt_{44}-\wt_{22}-\wt_{33}$\\
\hline
\rowcolor{red!30}
$(\underline{\b}_{44}^0,\b_{24}^0,\b_{34}^0,\b_{22}^0,t^0)$ (b) & $2\wt_t-\wt_{44}+\wt_{12}$ & $\wt_{12}$ & $2\wt_{22}+2\wt_t-\wt_{44}-\wt_{33}+\wt_{12}$ & $\wt_{23}-\wt_{22}-\wt_{33}+\wt_t$ & $\wt_{24}-\wt_{22}-\wt_{44}+\wt_t$ & $\wt_{33}-\wt_{12}-\wt_t$ & $\wt_{34}-\wt_{33}-\wt_{44}+\wt_{22}+\wt_t$ & $\wt_{44}-\wt_{22}-\wt_{12}-\wt_t$\\

\hline
$(\b_{44}^0,\b_{24}^0,\b_{34}^0,\underline{\b}_{22}^0,t^0)$ & $\wt_{33}-2\wt_{22}$ & $\wt_{12}$ & $\wt_{22}-\wt_{33}+\wt_t$ & $\wt_{23}-\wt_{22}-\wt_{33}+\wt_t$ & $\wt_{24}-3\wt_{22}+\wt_{33}-\wt_{12}-\wt_t$ & $\wt_{33}-\wt_{12}-\wt_t$ & $\wt_{34}-\wt_{12}-\wt_{22}-\wt_t$ & $\wt_{44}+\wt_{33}-2\wt_{22}-\wt_{12}-2\wt_t$\\
\hline
\rowcolor{red!30}
$(\underline{\b}_{44}^0,\b_{24}^0,\b_{34}^0,\b_{33}^0,t^0)$ & $\wt_{12}+2\wt_t-\wt_{44}$ & $\wt_{12}$ & $\wt_{22}-\wt_{33}+\wt_t$ & $\wt_{23}-\wt_{22}-\wt_{33}+\wt_t$ & $\wt_{24}-\wt_{22}-\wt_{44}+\wt_t$ & $2\wt_{22}+\wt_{t}-\wt_{44}$ & $\wt_{34}-\wt_{33}-\wt_{44}+\wt_{22}+\wt_t$ & $\wt_{44}+\wt_{33}-2\wt_{22}-\wt_{12}-2\wt_t$\\

\end{tabular}}
\end{table}
The initial $\CC^* \times T$ weights (where $T\subset \GL(4)$ maximal torus) are
\[\wt_t=z+z_1,\  \wt_{ij}=(j-1)z+z_i-z_1\]
and hence the weights in the two contributing rows above (now corresponding to Contr 1 and Contr2 columns)  are:

\begin{table}[ht]
\centering
\scalebox{0.7}{
\begin{tabular}{|c||c|c|c|c|}
Variable & Contr. 1 & Contr 2 & Contr 3 & Contr 4\\
\hline
$t$ & $2z_1-z_2$ & $z_1-z_4+z_3$ & $3z_1-z_4$ & $2z_1-z_4$\\
\hline
$\b_{12}$ & $z_1+z_2-z_4$ & $2z_1-z_3$ & z & $z$\\
\hline
$\b_{13}$ & $2z$ & $2z$ &$2z$ & $2z$ \\
\hline
$\b_{14}$ & $3z$ & $3z$ & $3z$ & $3z$\\
\hline
$\b_{22}$ & $z_1+z_2-z_3$ & $2z_2-z_4$ & $2z_2+2z_1-z_4-z_3$ & $z_1+z_2-z_3$\\
\hline
$\b_{23}$ & $2z_1-z_3$ & $2z_1-z_3$ & $2z_1-z_3$ & $2z_1-z_3$ \\
\hline
$\b_{24}$ & $2z_1-z_4$ & $2z_1-z_4$ &$2z_1-z_4$ &  $2z_1-z_4$\\
\hline
$\b_{33}$ & $z_2+z_3-z_4-z_1$ & $z+z_3-2z_1$ & $z_3-2z_1$ & $2z_2-z_4$ \\
\hline
$\b_{34}$ & $z_1+z_2-z_4$ & $z_1+z_2-z_4$ & $z_1+z_2-z_4$ & $z_1+z_2-z_4$\\
\hline
$\b_{44}$ & $z+z_4-z_2-z_1$ & $z_1+z_4-z_2-z_3$ & $z_4-z_2-z_1$ &$z_3+z_4-2z_2-2z_1$\\
\hline
\end{tabular}}
\end{table}
The Euler class of the blown-up space at the four contributing fixed points are
{\footnotesize
\begin{align*}
Contr_1(\bz)& =6z^2(z_2+z_3-z_1-z_4)(z_1+z_2-z_4)^2(2z_1-z_3)(2z_1-z_4)(2z_1-z_2)(z_1+z_2-z_3)(z+z_4-z_2-z_1)\\
Contr_2(\bz)& =6z^2(z_2+z_3-z_1-z_4)(z_1+z_2-z_4)(2z_1-z_3)^2(2z_1-z_4)(z_4-z_1-z_3)(2z_2-z_4)(z+z_3-2z_1)\\
Contr_3(\bz)& =6z^3(z_3+z_4-2z_1-2z_2)(z_1+z_2-z_4)^2(2z_1-z_3)^2(2z_1-z_4)(3z_1-z_4)\\
Contr_4(\bz)& =6z^3(z_3+z_4-2z_1-2z_2)(z_1+z_2-z_4)(2z_1-z_3)(2z_1-z_4)^2(z_4-z_1-z_3)(2z_2-z_4)
\end{align*}}
The residue formula has the form 
{\footnotesize \begin{align*}
\Tp_4& =\res_{z_1<z_2<z_3<z_4<z}6z^3(z_1z_2z_3z_4)^{N-n}\prod_{1\le i<j\le 4}(z_i-z_j)\left(\sum_{i=1}^4\frac{1}{Contr_i(z)}\right) \prod_{i=1}^4 c_{TM-TN}(1/z_i)d\mathbf{z}
\end{align*}}
We start with the simple observation that the third and fourth term vanishes. Indeed,  
{\footnotesize \begin{multline}\nonumber
\res_{z_1<z_2<z_3<z_4<z}\frac{6z^3(z_1z_2z_3z_4)^{N-n}\prod_{i<j}(z_i-z_j)}{Contr_3(z)} \prod_{i=1}^4 c_{TM-TN}(1/z_i)=\\
\res_{z_1<z_2<z_3<z_4<z}\frac{(z_1z_2z_3z_4)^{N-n}\prod_{i<j}(z_i-z_j)}{(z_3+z_4-2z_1-2z_2)(z_1+z_2-z_4)(2z_1-z_3)(2z_1-z_4)^2(z_4-z_1-z_3)(2z_2-z_4)} \prod_{i=1}^4 c_{TM-TN}(1/z_i)=0
\end{multline}}
simply because the $z$ factors cancel out, so there is no $z$ variable in the expression, and hence the first residue with respect to $z$ is zero. Similar argument shows the vanishing of the fourth contribution. We are left with the first two terms, which can be written as
{\footnotesize \begin{align*}
\Tp_4&=\res_{z_1<z_2<z_3<z_4<z}6z^3(z_1z_2z_3z_4)^{N-n}\prod_{i<j}(z_i-z_j)\left(\frac{1}{Contr_1(z)}+\frac{1}{Contr_2(z)}\right) \prod_{i=1}^4 c_{TM-TN}(1/z_i)d\mathbf{z}=\\
& =\res_{z_1<z_2<z_3<z_4<z}\frac{(z_1z_2z_3z_4)^{N-n}\prod_{i<j}(z_i-z_j)d\bz}{(z_2+z_3-z_1-z_4)(z_1+z_2-z_4)(2z_1-z_3)(2z_1-z_4)} \prod_{i=1}^4 c_{TM-TN}(1/z_i) \cdot\\
& \cdot \left(\frac{z}{(z_1+z_2-z_4)(2z_1-z_2)(z_1+z_2-z_3)(z+z_4-z_2-z_1)}+\frac{z}{(2z_1-z_3)(z_4-z_1-z_3)(2z_2-z_4)(z+z_3-2z_1)}\right)
\end{align*}}
Note that the first residue with respect to $z$ is nonzero if and only if we take the term $\frac{z_1+z_2-z_4}{z}$ from the expansion 
\[\frac{z}{z+z_4-z_2-z_1}=1+\frac{z_1+z_2-z_4}{z}+\frac{(z_1+z_2-z_4)^2}{z^2}+\ldots \]
and the term $\frac{2z_1-z_3}{z}$ from the expansion 
\[\frac{z}{z+z_3-2z_1}=1+\frac{2z_1-z_3}{z}+\frac{(2z_1-z_3)^2}{z^2}+\ldots. \]
Hence 
{\footnotesize \begin{align*}
\Tp_4 &=\res_{z_1<z_2<z_3<z_4}\frac{(z_1z_2z_3z_4)^{N-n}\prod_{i<j}(z_i-z_j)d\mathbf{z}}{(z_2+z_3-z_1-z_4)(z_1+z_2-z_4)(2z_1-z_3)(2z_1-z_4)}\cdot \prod_{i=1}^4 c_{TM-TN}(1/z_i) \cdot \\
& \cdot \left(\frac{1}{(2z_1-z_2)(z_1+z_2-z_3)}+\frac{1}{(z_4-z_1-z_3)(2z_2-z_4)}\right) 
\end{align*}}
And here comes the miracle. The following identity holds:
\begin{multline}\nonumber \frac{1}{(z_2+z_3-z_1-z_4)(2z_1-z_2)(z_1+z_2-z_3)}+\frac{1}{(z_2+z_3-z_1-z_4)(z_4-z_1-z_3)(2z_2-z_4)}=\\
=\frac{2z_1+z_2-z_4}{(2z_1-z_2)(z_1+z_2-z_3)(z_1+z_3-z_4)(2z_2-z_4)}
\end{multline}
This identity transforms our new formula into the old formula of \cite{bsz}:
{\footnotesize \begin{align*}
\Tp_4 =\res_{z_1<z_2<z_3<z_4}\frac{((2z_1+z_2-z_4)(z_1z_2z_3z_4)^{N-n}\prod_{i<j}(z_i-z_j)\prod_{i=1}^4 c_{TM-TN}(1/z_i)d\mathbf{z}}{(2z_1-z_2)(z_1+z_2-z_3)(z_4-z_1-z_3)(2z_2-z_4)(z_1+z_2-z_4)(2z_1-z_3)(2z_1-z_4)}=\Tp_4^{BSZ}
\end{align*}}

\subsection{Thom polynomials of $A_5$ singularities}\label{subsec:k=5} Here the affine coordinates over the flag $\ff=(e_1,e_2,e_3,e_4) \in \flag_4(\CC^n)$ are 
\begin{align*}
w_1=& e_1 \\
w_2= &\b_{22}e_2+\b_{12}e_1\\
w_3= &\b_{33}e_3+\b_{23}e_2+\b_{13}e_1\\
w_4= &\b_{44}e_4+\b_{34}e_3+\b_{24}e_2+\b_{14}e_1\\
w_5= &\b_{55}e_5+\b_{45}e_4+\b_{35}e_3+\b_{25}e_2+\b_{15}e_1
\end{align*}
and the rational map $\phi: U_0=\mathrm{Spec}(\CC[t,\b_{ij}]) \to \PP(\sym^{\le 5}\CC^n)$ is
{\footnotesize
\begin{align*}
\phi(t,w_1,\ldots, w_5)=& w_1 \wedge w_2 \wedge w_3 \wedge w_4 \wedge w_5 + \\
t &[(w_1 \wedge w_1^2 \wedge w_3 \wedge w_4 \wedge w_5) + (w_1 \wedge w_2 \wedge w_1w_2 \wedge w_4 \wedge w_5) + (w_1 \wedge w_2 \wedge w_3 \wedge (w_2^2+w_1w_3) \wedge w_5)\\
&+ (w_1 \wedge w_2 \wedge w_3 \wedge w_4 \wedge (w_1w_4+w_2w_3))  ]+\\
t^2 & [(w_1\wedge w_2 \wedge w_3 \wedge w_4 \wedge (w_1^2w_3 +w_1w_2^2))+(w_1 \wedge w_2 \wedge w_3 \wedge (w_1w_3 +w_2^2)\wedge (w_1w_4 +w_2w_3)) + \\
& (w_1 \wedge w_2 \wedge w_3 \wedge w_1^2w_2 \wedge w_5) +w_1 \wedge w_2 \wedge w_1w_2 \wedge w_4 \wedge (w_1w_4 +w_2w_3) + \\
& (w_1 \wedge w_2 \wedge w_1w_2 \wedge (w_1w_3 +w_2^2)\wedge w_5) +w_1 \wedge w_2 \wedge w_1^3 \wedge w_4 \wedge w_5 + \\
& (w_1 \wedge w_1^2 \wedge w_3 \wedge w_4 \wedge (w_1w_4 +w_2w_3))+w_1 \wedge w_1^2 \wedge w_3 \wedge (w_1w_3 +w_2^2)\wedge w_5 +\\
& (w_1 \wedge w_1^2 \wedge w_1w_2 \wedge w_4 \wedge w_5)]+\\
\ldots  & \ldots +\\
t^{10} & [w_1 \wedge w_1^2 \wedge w_1^3 \wedge w_1^4 \wedge w_1^5]
\end{align*}}

\begin{figure}[h]
\centering
{\tiny \begin{tikzpicture}[node distance=2cm]
\node (I) [boxes] {$(t^1,\b_{55}^4,\b_{44}^3,\b_{45}^4,\b_{33}^2,\b_{34}^3,\b_{35}^4)$};
\node (II) [boxes, below of =I ,yshift=1cm] {$(t^1,\b_{22}^1,\b_{23}^2,\b_{24}^3,\b_{25}^4)$};
\node (nocontrII) [nocontr, left of=II,xshift=-1.1cm ] {No contr.};
\node (III) [boxes, below of=II,yshift=1cm] {$(\b_{22}^1,\b_{33}^1,\b_{34}^2,\b_{44}^2)$};
\node (nocontrIII) [nocontr, left of =III,xshift=-1.1cm] {No contr.};
\node (IV) [boxes, below of=III, yshift=1cm] {$(\b_{22}^0,\b_{44}^1,\b_{45}^3,\b_{55}^3)$}; 
\node (V) [boxes, below of=IV, yshift=1cm] {$(\b_{12}^1,\b_{13}^2,\b_{14}^3,\b_{22}^0,\b_{23}^1,\b_{24}^2,\b_{25}^3)$};
\node (Vb) [boxes, below of=V, yshift=1cm] {$(\b_{22}^0,\b_{24}^2,\b_{33}^1,\b_{35}^3)$};
\node (VI) [boxes, below of=Vb, yshift=1cm] {$(\b_{23}^1(sub),\b_{24}^2(sub),\b_{33}^1)$};
\node (nocontrVI) [nocontr, left of =VI,xshift=-1cm] {No contr.};
\node (VII) [boxes, below of=VI, yshift=1cm] {$(\b_{22}^0,\b_{24}^1,\b_{34}^1,\b_{44}^1)$};
\node (VIII) [boxes, below of=VII, yshift=1cm] {$(\b_{12}^1,\b_{24}^1,\b_{33}^1,\b_{34}^1,\b_{44}^1)$};
\node (nocontrVIII) [nocontr, left of =VIII,xshift=-1.2cm] {No contr.};
\node (IX) [boxes, below of=VIII, yshift=1cm] {$(\b_{25}^3(sub),\b_{35}^3,\b_{44}^1,\b_{45}^3,\b_{55}^3)$};
\node (IXb) [boxes, right of=IX,xshift = 2.5cm] {$(\b_{25}^3 (sub),\b_{33}^1,\b_{35}^3,\b_{45}^3,\b_{55}^3)$};
\node (Xb) [boxes, below of =IXb,xshift=2.6cm,yshift=1cm] {$(t^0,\b_{22}^0,\b_{24}^0,\b_{34}^0(sub),\b_{44}^0)$};
\node (XIb) [boxes, below of=Xb, yshift=1cm] {$\b_{22}^0,\b_{25}^2,\b_{45}^2,\b_{55}^2)$};
\node (XIc) [boxes, below of=XIb, yshift=1cm] {$\b_{44}^0,\b_{25}^2,\b_{45}^2,\b_{55}^2)$};
\node (XIIb) [boxes, below of=XIc, yshift=1cm] {$(\b_{22}^0)^2(\b_{44}^0)^3\b_{33}^1,\b_{25}^2,\b_{35}^2(sub),\b_{45}^2,\b_{55}^2)$};
\node (XIIIb) [boxes, below of=XIIb, yshift=1cm] {$\b_{25}^1,\b_{33}^1,\b_{35}^1,\b_{45}^1,\b_{55}^1)$};
\node (contrXIIIb) [contr, right of=XIIIb, xshift=0cm] {1};
\node (XIVb) [boxes, below of=XIIIb, yshift=1cm, xshift =-0.1cm] {$t^0,\b_{22}^0,\b_{25}^0,\b_{35}^0,\b_{55}^0)$};
\node (contrXIVb) [contr, below of=XIVb, yshift=1.3cm,xshift=-1.3cm] {8};
\node (XVb) [boxes, right of=XIVb, xshift = 0.5cm] {$t^0,\b_{22}^0,\b_{25}^0,\b_{35}^0,\b_{55}^0)$};
\node (contrXVb) [contr, right of=XVb, xshift=0cm] {9};
\node (XVIb) [boxes, below of=XVb, xshift = -1.7cm,yshift = 1cm] {$t^0,\b_{25}^0,\b_{35}^0,\b_{45}^0(sub)\b_{55}^0$};
\node (contrXVIb) [contr, left of=XVIb,yshift=-0.3cm, xshift=0.3cm] {13};
\node (XVIIIb) [boxes, below of=XVb, xshift = 1cm,yshift = 1cm] {$\b_{22},\b_{25}^0,\b_{35}^0,\b_{45}^0,\b_{55}^0$};
\node (contrXVIIIb) [contr, right of=XVIIIb, xshift=-0.5cm] {10};
\node (XVIIb) [boxes, below of=XVIb, xshift = -0.6cm,yshift = 1.1cm] {$\b_{23}^0,\b_{25}^0,\b_{35}^0,\b_{45}^0,\b_{55}^0,\b_{22}^0$};
\node (contrXVIIb) [contr, left of=XVIIb,yshift=-0.3cm, xshift=0.3cm] {14};
\node (XXb) [boxes, below of=XVIIb, xshift = 0cm,yshift = 1.1cm] {$\b_{23}^0,\b_{25}^0,\b_{35}^0,\b_{45}^0,\b_{55}^0,\b_{44}^0$};
\node (contrXXb) [contr, left of=XXb,yshift=-0.3cm, xshift=0.3cm] {15};
\node (XIXb) [boxes, below of=XVIIIb, xshift = 0cm,yshift = 1.1cm] {$\b_{23}^0,\b_{25}^0,\b_{35}^0,\b_{45}^0,\b_{55}^0,t^0$};
\node (contrXIXb) [contr, right of=XIXb, yshift=-0.3cm,xshift=-0.5cm] {11};
\node (XXIb) [boxes, below of=XIXb, xshift = 0cm,yshift = 1.1cm] {$\b_{23}^0,\b_{25}^0,\b_{35}^0,\b_{45}^0,\b_{55}^0,\b_{44}^0$};
\node (contrXXIb) [contr, right of=XXIb, yshift=-0.3cm,xshift=-0.5cm] {12};
\node (X) [boxes, below of=IX, yshift=1cm] {$(\b_{22}^0,\b_{25}^2,\b_{45}^2,\b_{55}^2)$};
\node (XI) [boxes, below of=X, yshift=1cm] {$\b_{25}^2,\b_{33}^0,\b_{35}^2 (sub),\b_{55}^2)$};
\node (XII) [boxes, below of=XI, yshift=1cm] {$(\b_{25}^2,(\b_{22}^0)^2\b_{33}^0\b_{44}^1,\b_{35}^2,\b_{45}^2,\b_{55}^2)$};
\node (XIV) [boxes, below of=XII, yshift=1cm] {$(\b_{25}^1,\b_{35}^1,\b_{44}^1,\b_{45}^1,\b_{55}^1)$};
\node (contrXIV) [contr, left of=XIV,xshift=-1.5cm,yshift=0.5cm ] {0};
\node (nocontrXIV) [nocontr, below of=contrXIV, yshift=1cm] {No contr.};
\node (XV) [boxes, below of=XIV, yshift=0.8cm, xshift=-1cm] {$(t^0 (sub),\b_{22}^0,\b_{25}^0,\b_{35}^0,\b_{55}^0)$};
\node (contrXV) [contr, left of=XV,xshift=-0.5cm] {2};
\node (XVII) [boxes, below of=XIV,xshift = 2.4cm,yshift = 0.8cm] {$(t^0,\b_{22}^0,\b_{25}^0,\b_{35}^0,\b_{55}^0))$};
\node (XVIII) [boxes, below of=XVII,yshift = 1.1cm,xshift=0.5cm] {$(\b_{22}^0,\b_{25}^0,\b_{35}^0,\b_{45}^0,\b_{55}^0))$};
\node (XIX) [boxes, below of=XVIII,yshift = 1.1cm] {$(t^0,\b_{25}^0,\b_{34}^0,\b_{35}^0,\b_{45}^0,\b_{55}^0))$};
\node (XX) [boxes, below of=XVII,yshift = 1.1cm,xshift=-2.7cm] {$(t^0,\b_{25}^0,\b_{35}^0,\b_{45}^0,\b_{55}^0))$};
\node (contrXX) [contr, left of=XX,xshift=-0.3cm] {6};
\node (XXI) [boxes, below of=XX,yshift = 1.1cm,xshift=-0.1cm] {$(\b_{22}^0,\b_{25}^0,\b_{34}^0,\b_{35}^0,\b_{45}^0,\b_{55}^0))$};
\node (contrXVII) [contr, right of=XVII, xshift=0cm] {3};
\node (contrXVIII) [contr, right of=XVIII, xshift=0cm] {4};
\node (contrXIX) [contr, below of=XIX, yshift=1.1cm] {5};
\node (contrXXI) [contr, below of=XXI, yshift=1.1cm] {7};

\draw [arrow] (I) -- node[anchor=east] {$t^1$} (II);
\draw [arrow] (II) -- node[anchor=east] {$\b_{22}^1$}(III);
\draw [arrow] (II) -- node[anchor=south] {$t^1$} (nocontrII);
\draw [arrow] (III) -- node[anchor=south] {$\b_{22}^1$} (nocontrIII);
\draw [arrow] (III) -- node[anchor=east] {$\b_{33}^1$} (IV);
\draw [arrow] (IV) -- node[anchor=east] {$\b_{22}^0$} (V);
\draw [arrow] (V) -- node[anchor=east] {$\b_{22}^0$} (Vb);
\draw [arrow] (Vb) -- node[anchor=east] {$\b_{22}^0$} (VI);
\draw [arrow] (VI) -- node[anchor=east] {$\b_{33}^1$} (VII);
\draw [arrow] (VI) -- node[anchor=south] {$\b_{23}^1$} (nocontrVI);
\draw [arrow] (VII) -- node[anchor=east] {$\b_{22}^0$} (VIII);
\draw [arrow] (VIII) -- node[anchor=south] {$\b_{12}^1,\b_{24}^1,\b_{34}^1$} (nocontrVIII);
\draw [arrow] (VIII) -- node[anchor=east] {$\b_{44}^1$} (IX);
\draw [arrow] (VIII) -- node[anchor=south] {$\b_{33}^1$} (IXb);
\draw [arrow] (IXb) -- node[anchor=east] {$\b_{33}^1$} (Xb);
\draw [arrow] (Xb) -- node[anchor=east] {$\b_{44}^0$} (XIb);
\draw [arrow] (XIb) -- node[anchor=east] {$\b_{22}^0$} (XIc);
\draw [arrow] (XIc) -- node[anchor=east] {$\b_{44}^0$} (XIIb);
\draw [arrow] (XIIb) -- node[anchor=east] {$\b_{22}^0(2),\b_{44}^0(3),\b_{33}^1$} (XIIIb);
\draw [arrow] (XIIIb) -- node[anchor=south] {$\b_{55}^1$} (contrXIIIb);
\draw [arrow] (XIIIb) -- node[anchor=east] {$\b_{45}^1$} (XIVb);
\draw [arrow] (XIVb) -- node[anchor=west] {$\b_{55}^0$} (contrXIVb);
\draw [arrow] (XVb) -- node[anchor=south] {$\b_{55}^0$} (contrXVb);
\draw [arrow] (XIIIb) -- node[anchor=east] {$\b_{33}^1$} (XVb);
\draw [arrow] (XVb) -- node[anchor=east] {$\b_{22}^0$} (XVIb);
\draw [arrow] (XVb) -- node[anchor=east] {$t^0$} (XVIIIb);
\draw [arrow] (XVIb) -- node[anchor=east] {$t^0$} (XVIIb);
\draw [arrow] (XVIIIb) -- node[anchor=east] {$\b_{22}^0$} (XIXb);
\draw [arrow] (XVIIb) -- node[anchor=east] {$\b_{22}^0$} (XXb);
\draw [arrow] (XIXb) -- node[anchor=east] {$t^0$} (XXIb);
\draw [arrow] (IX) -- node[anchor=east] {$\b_{44}^1$} (X);
\draw [arrow] (X) -- node[anchor=east] {$\b_{22}^0 $} (XI);
\draw [arrow] (XI) -- node[anchor=east] {$\b_{33}^0$} (XII);
\draw [arrow] (XII) -- node[anchor=east] {$\b_{22}^0 \times 2,\b_{33}^0,\b_{44}^1$} (XIV);
\draw [arrow] (XIV) -- node[anchor=east] {$\b_{45}^1$} (XV);
\draw [arrow] (XIV) -- node[anchor=south] {$\b_{55}^1$} (contrXIV);
\draw [arrow] (XIV) -- node[anchor=south] {$\b_{25}^1,\b_{35}^1$} (nocontrXIV);
\draw [arrow] (XIV) -- node[anchor=east] {$\b_{44}^1$} (XVII);
\draw [arrow] (XV) -- node[anchor=south] {$\b_{55}^0$} (contrXV);
\draw [arrow] (XVII) -- node[anchor=east] {$t^0$} (XVIII);
\draw [arrow] (XVIII) -- node[anchor=east] {$\b_{22}^0$} (XIX);
\draw [arrow] (XVII) -- node[anchor=west] {$\b_{22}^0$} (XX);
\draw [arrow] (XX) -- node[anchor=east] {$t^0$} (XXI);
\draw [arrow] (XVII) -- node[anchor=south] {$\b_{55}^0$} (contrXVII);
\draw [arrow] (XVIII) -- node[anchor=south] {$\b_{55}^0$} (contrXVIII);
\draw [arrow] (XIX) -- node[anchor=east] {$\b_{55}^0$} (contrXIX);
\draw [arrow] (XX) -- node[anchor=south] {$\b_{55}^0$} (contrXX);
\draw [arrow] (XXI) -- node[anchor=east] {$\b_{55}^0$} (contrXXI);

\end{tikzpicture}}
\caption{The blow-up tree for $k=5$}\label{tree5}
\end{figure}
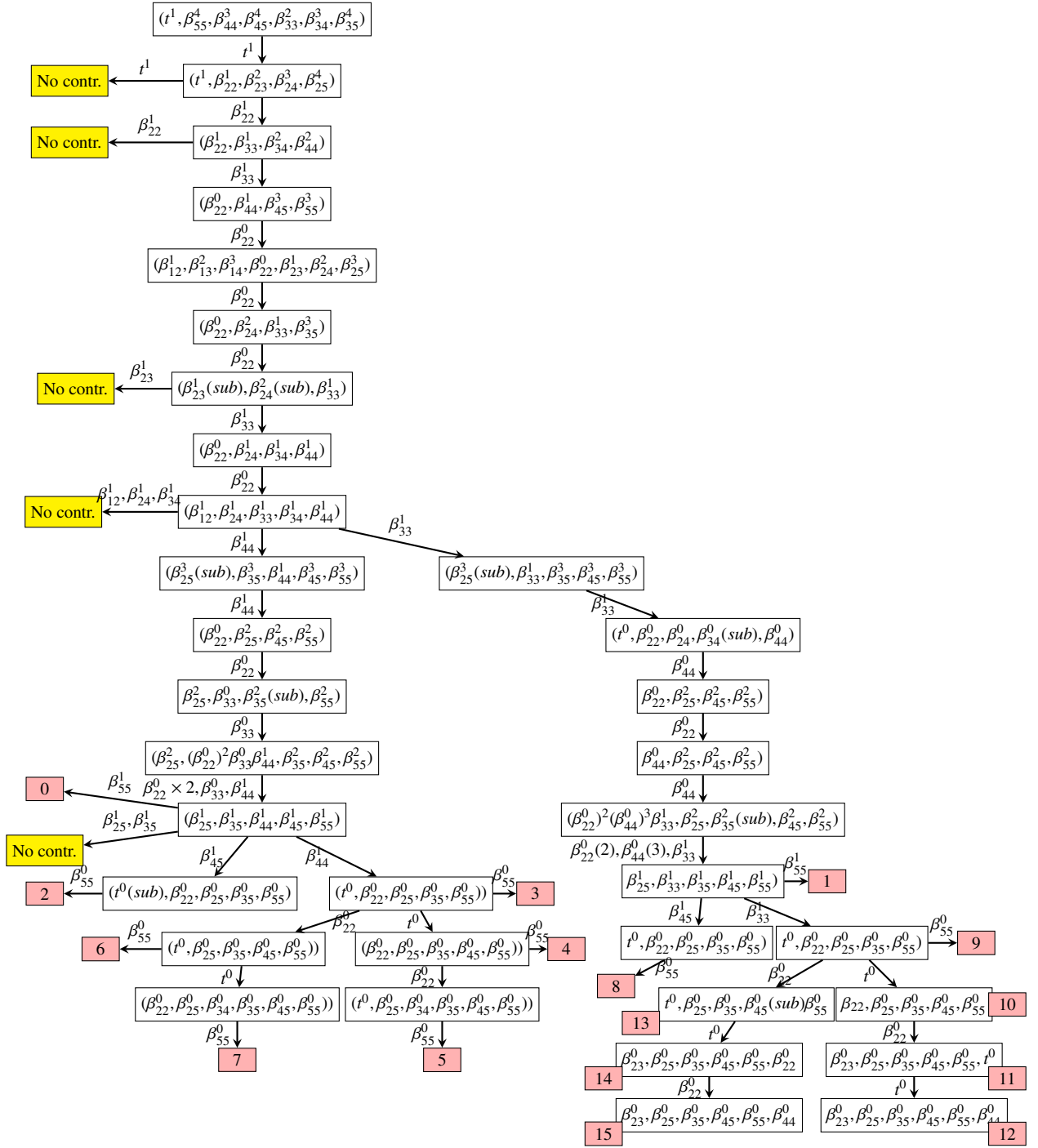

The blowing up process is described by the blow-up three in Figure \ref{tree5}. We end up with having $15$ leaves corresponding to $15$ fixed points on the blown-up space mapped to $\ff=e_1\wedge \ldots \wedge e_5$. At each step we (i) might perform some change of coordinates, we refer this as the smoothening step in the next section, and we indicate transformed variables by (sub) in the tree. This substitution does not make any impact on the weight calculations, but it is crucial in order to work with smooth centers in blow ups (ii) pick a $\diff_5$-invariant subspace corresponding to a cluster and (iii) choose the affine chart, corresponding to a minimal weight variable in the cluster.

The Euler classes $\mathrm{Contr(i)}$ at the $15$ fixed points are collected below. Each one of these is the products of $14$ linear factors (weights of the tangent space). Note that in Contr0 and Contr1 there are more than $5$ linear terms containing $z$, hence their $z$-residue is $0$, they do not contribute to the formula. The remaining 13 Euler classes all have degree $4$ in $z$.
The residue formula reduces to the following form: 
\begin{align*}
\Tp_5& =\res_{z_1<\ldots <z_5<z}24z^4(z_1z_2z_3z_4z_5)^{N-n}\prod_{i<j}(z_i-z_j)\left(\sum_{i=2}^{15}\frac{1}{Contr_i(z)}\right) \prod_{i=1}^5 c_{TM-TN}(1/z_i) d\mathbf{z}
\end{align*}
Comparing this with the formula in \cite{bsz}
\[\Tp_5^{BSZ}=\res_{\bz}\frac{(z_1\ldots z_5)^{N-n})\prod_{i<j}(z_i-z_j)Q_5(\bz)}{\prod_{i+j\le l\le 5}(z_i+z_j-z_l)}\prod_{i=1}^5 c_{TM-TN}(1/z_i) d\mathbf{z}\]
our new formula provides a partial fraction decomposition of the old:
\[\frac{(2z_1+z_2-z_5)(2z_1^2+3z_1z_2-2z_1z_5+2z_2z_3-z_2z_4-z_2z_5-z_3z_4+z_4z_5)}{\prod_{i+j\le l\le 5}(z_i+z_j-z_l)}=\res_{z}\sum_{i=2}^{15} \frac{24z^4 dz}{Contr_i(z,\bz)}\]
 
\begingroup
\allowdisplaybreaks
{\scriptsize 
\begin{align*} 
\mathrm{Contr2}&=(2\,{z}_{1}-{z}_{2})({z}_{1}+{z}_{2}-{z}_{4})(2\,{z}-{z}_{1}-{z}_{2}+{z}_{3})(3\,{z}-{z}_{1}-{z}_{2}+{z}_{3})(4\,{z})({z}_{1}+{z}_{2}-{z}_{3})(2\,{z}_{1}-{z}_{3})(2\,{z}_{1}-{z}_{4})\\&(3\,{z}_{1}-{z}_{2}-{z}_{3})(-{z}_{1}+{z}_{2}+{z}_{3}-{z}_{4})({z}_{1}+{z}_{2}-{z}_{4})(2\,{z}_{1}-{z}_{3})(-{z}_{1}-{z}_{2}+{z}_{4})({z}-{z}_{1}-{z}_{2}+{z}_{3})({z}_{1}-{z}_{2}-{z}_{3}+{z}_{5}) \\
\mathrm{Contr3}&=(2\,{z}_{1}-{z}_{2})({z}_{1}+{z}_{2}-{z}_{4})(2\,{z}-{z}_{1}-{z}_{2}+{z}_{3})(3\,{z}-{z}_{1}-{z}_{2}+{z}_{3})(4\,{z})({z}_{1}+{z}_{2}-{z}_{3})(2\,{z}_{1}-{z}_{3})(2\,{z}_{1}-{z}_{4})(4\,{z}_{1}-{z}_{3}-{z}_{4})\\&(-{z}_{1}+{z}_{2}+{z}_{3}-{z}_{4})({z}_{1}+{z}_{2}-{z}_{4})(3\,{z}_{1}+{z}_{2}-{z}_{3}-{z}_{4})({z}-2\,{z}_{1}-2\,{z}_{2}+{z}_{3}+{z}_{4})({z}_{1}+{z}_{2}-{z}_{4})(2\,{z}_{1}-{z}_{3}-{z}_{4}+{z}_{5})\\
\mathrm{Contr4}&= (2\,{z}_{1}-{z}_{2})({z}_{1}+{z}_{2}-{z}_{4})(2\,{z}-{z}_{1}-{z}_{2}+{z}_{3})(3\,{z}-{z}_{1}-{z}_{2}+{z}_{3})(4\,{z})(-{z}_{1}+2\,{z}_{2}-{z}_{3})(2\,{z}_{1}-{z}_{3})(2\,{z}_{1}-{z}_{4})(2\,{z}_{1}+{z}_{2}-{z}_{3}-{z}_{4})\\& (-{z}_{1}+{z}_{2}+{z}_{3}-{z}_{4})({z}_{1}+{z}_{2}-{z}_{4})({z}_{1}+2\,{z}_{2}-{z}_{3}-{z}_{4})({z}-2\,{z}_{1}-2\,{z}_{2}+{z}_{3}+{z}_{4})({z}_{1}+{z}_{2}-{z}_{4})({z}_{2}-{z}_{3}-{z}_{4}+{z}_{5})\\
\mathrm{Contr5}&=(2\,{z}_{1}-{z}_{2})({z}_{1}+{z}_{2}-{z}_{4})(2\,{z}-{z}_{1}-{z}_{2}+{z}_{3})(3\,{z}-{z}_{1}-{z}_{2}+{z}_{3})(4\,{z})(-{z}_{1}+2\,{z}_{2}-{z}_{3})(2\,{z}_{1}-{z}_{3})(2\,{z}_{1}-{z}_{4})(3\,{z}_{1}-{z}_{2}-{z}_{4})\\&(-{z}_{1}+{z}_{2}+{z}_{3}-{z}_{4})({z}_{1}+{z}_{2}-{z}_{4})(2\,{z}_{1}-{z}_{4})({z}-2\,{z}_{1}-2\,{z}_{2}+{z}_{3}+{z}_{4})(2\,{z}_{1}-{z}_{2}+{z}_{3}-{z}_{4})({z}_{1}-{z}_{2}-{z}_{4}+{z}_{5})\\
\mathrm{Contr6}&=({z}_{1}-2\,{z}_{2}+{z}_{3})({z}_{1}+{z}_{2}-{z}_{4})(2\,{z}-{z}_{1}-{z}_{2}+{z}_{3})(3\,{z}-{z}_{1}-{z}_{2}+{z}_{3})(4\,{z})({z}_{1}+{z}_{2}-{z}_{3})(2\,{z}_{1}-{z}_{3})(2\,{z}_{1}-{z}_{4})\\&(3\,{z}_{1}-{z}_{2}-{z}_{4})(-{z}_{1}+{z}_{2}+{z}_{3}-{z}_{4})({z}_{1}+{z}_{2}-{z}_{4})(2\,{z}_{1}-{z}_{4})({z}-2\,{z}_{1}-2\,{z}_{2}+{z}_{3}+{z}_{4})({z}_{1}+{z}_{2}-{z}_{4})({z}_{1}-{z}_{2}-{z}_{4}+{z}_{5})\\ 
\mathrm{Contr7}&=({z}_{1}-2\,{z}_{2}+{z}_{3})({z}_{1}+{z}_{2}-{z}_{4})(2\,{z}-{z}_{1}-{z}_{2}+{z}_{3})(3\,{z}-{z}_{1}-{z}_{2}+{z}_{3})(4\,{z})({z}_{1}+{z}_{2}-{z}_{3})(2\,{z}_{1}-{z}_{3})(2\,{z}_{1}-{z}_{4})\\&(2\,{z}_{1}+{z}_{2}-{z}_{3}-{z}_{4})(-{z}_{1}+{z}_{2}+{z}_{3}-{z}_{4})({z}_{1}+{z}_{2}-{z}_{4})({z}_{1}+2\,{z}_{2}-{z}_{3}-{z}_{4})({z}-2\,{z}_{1}-2\,{z}_{2}+{z}_{3}+{z}_{4})(3\,{z}_{2}-{z}_{3}-{z}_{4})({z}_{2}-{z}_{3}-{z}_{4}+{z}_{5})\\
\mathrm{Contr8}&=({z}_{1}+{z}_{3}-{z}_{4})(2\,{z}_{1}-{z}_{3})(2\,{z}-{z}_{1}-{z}_{2}+{z}_{3})(3\,{z}-{z}_{1}-{z}_{2}+{z}_{3})(4\,{z})(2\,{z}_{2}-{z}_{4})(2\,{z}_{1}-{z}_{3})\\&(2\,{z}_{1}-{z}_{4})(2\,{z}_{1}-{z}_{4})(-2\,{z}_{1}+{z}_{3})({z}_{1}+{z}_{2}-{z}_{4})({z}_{1}+{z}_{2}-{z}_{4})({z}_{1}-{z}_{2}-{z}_{3}+{z}_{4})({z}-{z}_{1}-{z}_{2}+{z}_{3})(-{z}_{4}+{z}_{5})\\
\mathrm{Contr9}&=({z}_{1}+{z}_{3}-{z}_{4})(2\,{z}_{1}-{z}_{3})(2\,{z}-{z}_{1}-{z}_{2}+{z}_{3})(3\,{z}-{z}_{1}-{z}_{2}+{z}_{3})(4\,{z})(2\,{z}_{2}-{z}_{4})(2\,{z}_{1}-{z}_{3}))(2\,{z}_{1}-{z}_{4})\\&(4\,{z}_{1}-{z}_{3}-{z}_{4})({z}-3\,{z}_{1}-{z}_{2}+2\,{z}_{3}({z}_{1}+{z}_{2}-{z}_{4})(3\,{z}_{1}+{z}_{2}-{z}_{3}-{z}_{4})({z}_{1}-{z}_{2}-{z}_{3}+{z}_{4})(2\,{z}_{1}-{z}_{3})(2\,{z}_{1}-{z}_{3}-{z}_{4}+{z}_{5})\\
\mathrm{Contr10}&=({z}_{1}+{z}_{3}-{z}_{4})(2\,{z}_{1}-{z}_{3})(2\,{z}-{z}_{1}-{z}_{2}+{z}_{3})(3\,{z}-{z}_{1}-{z}_{2}+{z}_{3})(4\,{z})(-{z}_{1}+2\,{z}_{2}-{z}_{3})(2\,{z}_{1}-{z}_{3})\\&(2\,{z}_{1}-{z}_{4})(3\,{z}_{1}-2\,{z}_{3})({z}-3\,{z}_{1}-{z}_{2}+2\,{z}_{3})({z}_{1}+{z}_{2}-{z}_{4})(2\,{z}_{1}+{z}_{2}-2\,{z}_{3})({z}_{1}-{z}_{2}-{z}_{3}+{z}_{4})(2\,{z}_{1}-{z}_{3})({z}_{1}-2\,{z}_{3}+{z}_{5})\\
\mathrm{Contr11}&=({z}_{1}+{z}_{3}-{z}_{4})(2\,{z}_{1}-{z}_{3})(2\,{z}-{z}_{1}-{z}_{2}+{z}_{3})(3\,{z}-{z}_{1}-{z}_{2}+{z}_{3})(4\,{z})(-{z}_{1}+2\,{z}_{2}-{z}_{3})(2\,{z}_{1}-{z}_{3})(2\,{z}_{1}-{z}_{4})\\&(4\,{z}_{1}-2\,{z}_{2}-{z}_{3})({z}-3\,{z}_{1}-{z}_{2}+2\,{z}_{3})({z}_{1}+{z}_{2}-{z}_{4})(3\,{z}_{1}-{z}_{2}-{z}_{3})({z}_{1}-{z}_{2}-{z}_{3}+{z}_{4})(3\,{z}_{1}-2\,{z}_{2})(2\,{z}_{1}-2\,{z}_{2}-{z}_{3}+{z}_{5})\\
\mathrm{Contr12}&=({z}_{1}+{z}_{3}-{z}_{4})(2\,{z}_{1}-{z}_{3})(2\,{z}-{z}_{1}-{z}_{2}+{z}_{3})(3\,{z}-{z}_{1}-{z}_{2}+{z}_{3})(4\,{z})(-{z}_{1}+2\,{z}_{2}-{z}_{3})({z}_{1}-2\,{z}_{3}+{z}_{4})(2\,{z}_{1}-{z}_{4})\\&(3\,{z}_{1}-2\,{z}_{2}-2\,{z}_{3}+{z}_{4})({z}-3\,{z}_{1}-{z}_{2}+2\,{z}_{3})({z}_{1}+{z}_{2}-{z}_{4})(2\,{z}_{1}-{z}_{2}-2\,{z}_{3}+{z}_{4})({z}_{1}-{z}_{2}-{z}_{3}+{z}_{4})(2\,{z}_{1}-2\,{z}_{2}-{z}_{3}+{z}_{4})\\&({z}_{1}-2\,{z}_{2}-2\,{z}_{3}+{z}_{4}+{z}_{5})\\
\mathrm{Contr13}&=({z}_{1}-2\,{z}_{2}+{z}_{3})(2\,{z}_{1}-{z}_{3})(2\,{z}-{z}_{1}-{z}_{2}+{z}_{3})(3\,{z}-{z}_{1}-{z}_{2}+{z}_{3})(4\,{z})(2\,{z}_{2}-{z}_{4})(2\,{z}_{1}-{z}_{3})(2\,{z}_{1}-{z}_{4})\\&(4\,{z}_{1}-2\,{z}_{2}-{z}_{3})({z}-3\,{z}_{1}-{z}_{2}+2\,{z}_{3})({z}_{1}+{z}_{2}-{z}_{4})(3\,{z}_{1}-{z}_{2}-{z}_{3})({z}_{1}-{z}_{2}-{z}_{3}+{z}_{4})(2\,{z}_{1}-{z}_{3})(2\,{z}_{1}-2\,{z}_{2}-{z}_{3}+{z}_{5})\\
\mathrm{Contr14}&=({z}_{1}-2\,{z}_{2}+{z}_{3})(2\,{z}_{1}-{z}_{3})(2\,{z}-{z}_{1}-{z}_{2}+{z}_{3})(3\,{z}-{z}_{1}-{z}_{2}+{z}_{3})(4\,{z})(2\,{z}_{2}-{z}_{4})(2\,{z}_{1}-{z}_{3})(2\,{z}_{1}-{z}_{4})\\&(3\,{z}_{1}-2\,{z}_{3})({z}-3\,{z}_{1}-{z}_{2}+2\,{z}_{3})({z}_{1}+{z}_{2}-{z}_{4})(2\,{z}_{1}+{z}_{2}-2\,{z}_{3})({z}_{1}-{z}_{2}-{z}_{3}+{z}_{4})({z}_{1}+2\,{z}_{2}-2\,{z}_{3})({z}_{1}-2\,{z}_{3}+{z}_{5})\\
\mathrm{Contr15}&=({z}_{1}-2\,{z}_{2}+{z}_{3})(2\,{z}_{1}-{z}_{3})(2\,{z}-{z}_{1}-{z}_{2}+{z}_{3})(3\,{z}-{z}_{1}-{z}_{2}+{z}_{3})(4\,{z})(2\,{z}_{2}-{z}_{4})(2\,{z}_{1}-2\,{z}_{2}-{z}_{3}+{z}_{4})(2\,{z}_{1}-{z}_{4})\\&(3\,{z}_{1}-2\,{z}_{2}-2\,{z}_{3}+{z}_{4})({z}-3\,{z}_{1}-{z}_{2}+2\,{z}_{3})({z}_{1}+{z}_{2}-{z}_{4})(2\,{z}_{1}-{z}_{2}-2\,{z}_{3}+{z}_{4})({z}_{1}-{z}_{2}-{z}_{3}+{z}_{4})({z}_{1}-2\,{z}_{3}+{z}_{4})\\&({z}_{1}-2\,{z}_{2}-2\,{z}_{3}+{z}_{4}+{z}_{5}) 
\end{align*}
}
\endgroup

\section{The blowing up process : polynomial vs monomial maps}
We recall that we aim to construct a $\diff_k$-equivariant blow-up of the rational map 
\[\phi:\tPP^{ss} \subset U_0 \dasharrow \grass_k(\symdot)\]
\[\phi(t:w_1,\ldots w_k) =e_1 \wedge (w_2+te_1^2) \wedge \ldots \wedge (\sum_{i_1+2i_2 +\ldots +ki_k=k}t^{i_1+\ldots +i_k-1}e_1^{i_1}\ldots e_{k}^{i_k}).\]
from the affine chart $U_0 \subset \tPP$ to the projective space $\PP[\wedge^k \symdot]$. Each basis element $e_\pi=e_1 \wedge e_{\pi_2}\wedge \ldots \wedge e_{\pi_k}$ is labeled by a sequence $\pi=(\pi_1,\ldots, \pi_k)$ of partitions, where $\pi_r=(i_1,\ldots, i_k)\in \ZZ_{\ge 0}^k$ is a positive partition of $r=i_1+\ldots +ri_r$  for $1\le r \le k$. We can and will associate to such a partition $\pi_r=(i_1,\ldots, i_k)\in \ZZ_{\ge 0}^k$ the box centered at $(i_1,\ldots, i_k)$ on an $k$-dimensional lattice, and hence $\pi$ is given by $k$ boxes. We call such a collection a {\em semi-partition} in dimension $k$ and we denote the set of semi-partitions by $\Pi_{k}$. In short, semi-partitions parametrise the coordinates of a subspace of $\PP(\wedge^k \symdot)$ in which the image of $\phi$ sits. Some of these semi-partitions are proper $k$-dimensional partitions, that is
\[(i_1,\ldots, i_k)\in \pi \Rightarrow (i_1-1,i_2,\ldots, i_k), \ldots, (i_1,\ldots, i_k-1) \in \pi \]
Torus fixed points under the $T^n \subset \GL(n)$ action on $\Hilb^k(\CC^n)$ correspond to proper partitions $\l$: indeed, these correspond to monomial ideals $I_\l$ of length $k$ fixed by the torus, and those partitions sitting in $\Span(e_1,\ldots, e_n) \subset \CC^n$ correspond to basis of the subspace where $im(\phi)$ sits.
So $\phi$ has the form 
\[\phi(t,\b_{ij})=[p_{\pi}(t,\b_{i,j}):\pi \in \Pi_{k,n}] \]
where $p_\pi$ are polynomial functions (Pl\"ucker coordinates), see \eqref{mapk=3} for the $k=3$ case. The indeterminacy locus of $\phi$ is the zero set of the ideal
\[I_0=(p_{\pi}: \pi \in \Pi_{k,n}),\]
generated by the coordinate functions. This indeterminacy locus is highly singular.
\begin{remark} The zero set is highly singular, and the equivariant dual of its tangent cone at the distinguished point $e_1\wedge \ldots \wedge e_k$ encodes the $Q_k$ polynomial of \cite{bsz}, see the Introduction. As we explained, the main difficulty we face immediately working with $Q_k$, is that the description of the relations among these generators (i.e the first syzygies) is out of reach, and it involved deep Borel geometry. We have not seen any progress in their calculation in the last 12 years. 
\end{remark}
However, in our approach,  we do not need to work with the syzygies. We need a $\diff_k$-equivariant blow-up, such that $\phi$ extends to the NRGIT semistable locus. And we can perform most of these blow-ups at smooth centers, working with monomial ideals.  
\subsection{From polynomial to monomial maps}
For a polynomial $p \in \CC[x_1,\ldots, x_m]$ let $\calm(p)$ denote the set of monomials of $p$. For a polynomial rational map $f: \CC[x_1,\ldots, x_m] \dasharrow \PP^{r}$ given by $f(x_1,\ldots, x_m)=[p_0(\bx),\ldots, p_r(\bx)]$ we let 
\[I(f)=(p_1,\ldots, p_r) \subset \CC[x_1,\ldots, x_m]\]
denote the ideal of coordinate functions. We introduce the notation $\calm(f)=\cup_{i=1}^r \calm(p_i)$ for the set of all monomials (with multiplicity) of $f$ and 
\[M(f)=(\calm(f))=(\calm(p_i):\ i=1,\ldots ,r) \subset \CC[x_1,\ldots, x_m]\]
for the monomial ideal generated by all monomials of the polynomial coordinate functions of $f$. For a fixed order of elements of $\calm(f)$) let $f^\mon: \mathbb{A}^m \to \PP^{N-1}$, 
\[f^{\mon}(x_1,\ldots, x_m)=[m: m\in \calm(f)] \in \PP^{N-1}\]
denote the corresponding monomial map, where $N=|\calm(f)|$. For an ideal $I \subset \CC[x_1,\ldots, x_m]$ we use the shorthand notation $Z(I)=\mathrm{Spec}(\CC[\bx]/I)$ for the scheme determined by $I$ and $Z^\red(I)$ for the underlying reduced variety. A basic but important observation is that 
\begin{equation}\label{observation}
Z^\red(M(f)) \subset Z^\red(I(f))  
\end{equation}
holds for any $f$, which simply says that any polynomial vanishes at those points where all its monomials vanish. 

The idea is to blow-up along (irreducible components of) the locus $Z^\red(M(f))$ instead of the poorly understood $Z^\red(I(f)$). To do so, we take the primary decomposition
\[M(f)=Q_1 \cap Q_2 \cap \ldots \cap  Q_{l}\]
of $M(f)$, with primary ideals $Q_1,\ldots, Q_l$ whose radical ideals are the prime ideals $\mathrm{rad}(Q_i)=P_i$. This corresponds to the decomposition 
\[Z(M(f))= m_1Z(P_1) \cup m_2Z(P_2) \cup \ldots \cup m_{l}Z(P_{l})\]
of the scheme into irreducible (possibly embedded) reduced components $Z(P_i)$ with multiplicity $m_i$. Since $M(f)$ is monomial, this decomposition has a particularly nice form: 
\[P_i=(\calc_i) \text{ for a cluster of variables } \calc_i \subset \{x_1,\ldots, x_m\}\]
and $Q_i=(\tilde{\calc}_i)$ is generated by powers of elements in $\calc_i$. Hence 
\[Z(P_i)=Z(\calc_i)=\{(x_1,\ldots, x_m) \in \AAA^m: x_j=0 \text{ for } j\in \calc_i\}\]
is a coordinate subspace of $\AAA^m$. Pick $\calc:=\calc_i$ for some $i$, and let
\[\pi_\calc: \blow_{\calc}\AAA^m \to \AAA^m\]
denote the blow-up of $\AAA^m$ at $Z(\calc)$. Any $c\in \calc$ defines an affine chart $\blow_{\calc,c}\AAA^m \subset \blow_{\calc}\AAA^m$, and the restriction $\pi_{\calc,c}=\pi_{\calc}|_{\blow_{\calc,c}\AAA^m}$ is defined through  
\begin{equation}\label{transform}
\pi_{\calc,c}^*(x_j)=\begin{cases} x_j & j \notin \calc \setminus \{c\} \\ x_ix_j & j \in \calc \setminus \{c\} \end{cases}
\end{equation}
We will abuse notation and will use $x_j$ for the affine coordinate $\pi_{\calc,c}^*(x_j)$ on $\blow_{\calc,c}\AAA^m$. This way at each stage and on each affine chart we have the same notation for the coordinates, but we need to follow the simple transformation rule \eqref{transform} at each blow-up. The blow-up of the rational map $f: \AAA^m \rightarrow \PP^r$ is obtained on $\blow_{\calc,c}\AAA^m$ by the substitution \eqref{transform}:
\begin{equation}\label{blownupmorphism}
\xymatrix{
 \blow_{\calc}\AAA^m  \ar[d]^{\pi_\calc} &  \blow_{\calc,c}\AAA^m \ar@{_{(}->}[l] \ar@{.>}[rd]^-{\tilde{f}_{\calc,c}=f \circ \pi_{\calc,c}^*} &  \\
 \AAA^m \ar@{.>}[rr]^-{f} & & \PP^{r} }
\end{equation}
Note that all coordinate functions of the polynomial map $f \circ \pi_{\calc,c}^*$ is divisible by the variable $c$ (or even with some power of $c$), and hence we divide by this to get the polynomial map $\tilde{f}_{\calc,c}$.

We iterate this blow-up process. Take the primary decomposition of $M(\tilde{f}_{\calc,c})$ and pick a component of the monomial indeterminacy locus $Z(M(\tilde{f}_{\calc,c}))$ corresponding to a cluster (linear subspace) $\calc'$ of coordinates. This choice, of course, must be compatible throughout the affine charts, so that we blow-up along the same subvariety in the quasi-projective $\blow_{\calc}\AAA^m$. In order to see finiteness of this iterated blow-up process, we need to compare these primary decompositions, and apply
\begin{lemma} Let $Z(M(f))= m_1Z(P_1) \cup m_2Z(P_2) \cup \ldots \cup m_{l}Z(P_{l})$ be the primary decomposition and assume $Z(P_1)=Z(\calc_1)$ is an irreducible component (not embedded component). Let $\tilde{f}: \blow_{\calc_1,c}\AAA^m \to \PP^N$ be the blow-up map from the affine chart corresponding to $c\in \calc_1$. Then each component in the primary decomposition of $Z(M(\tilde{f}))$ is either (i) the proper transform of a component of $Z(M(f))$, or (ii) a subspace strictly contained in the proper transform of a component of $Z(M(f))$.
\end{lemma}
This implies that in each step, at least one component of the primary decomposition vanishes, or it splits into smaller subspaces, hence the blow-up process is finite.

\subsection{Smoothening trick} How long can we continue blowing-up along monomial ideals described above? Iterating the blow-up process, the degree of the coordinate functions of $f=[f_1,\ldots, f_{r+1}]$ eventually become lower, and essentially we arrive to a stage where one of these coordinate functions has a linear term. By a linear change of coordinates we can assume that this linear term is one of $x_1,\ldots, x_m$, and without loss of generality assume that $f_1=x_1+q(\bx)$ for some polynomial $q$. Then $x_1 \in \calm(f)$, and therefore all clusters in the primary decomposition contain $x_1$. For a cluster $\calc$ the blown-up map $\tilde{f}_{\calc,x_1}: \blow_{\calc,x_1}\AAA^m \dasharrow \PP^r$ has first coordinate function 
\[\tilde{f}_1=1+\tilde{q}(\bx)\] 
which is nonzero if $\tilde{q}=0$, and in this case $\tilde{f}_{\calc,x_1} :\blow_{\calc,x_1}\AAA^m \to \PP^r$ is a morphism. But if $\tilde{q} \neq 0$ then $\tilde{f}$ is not necessarily defined on the whole affine chart $\AAA^m$, and since $1 \in \calm(\tilde{f})$, the blow-up process terminates. In this situation we need a different blow-up, which we call the {\em smoothening trick}. 
\begin{definition} A rational map $f=[f_1,\ldots, f_{r+1}]$ is called \textbf{semi-linear}, if $f_i=x_j+p(\bx)$ for some $1\le i\le r+1$, $1\le j \le m$ with some polynomial $q$ of degree $\deg(q)\ge 2$. We will call such an $f_i$ \textbf{linear-headed}. 
\end{definition}

Let $f=[f_0,\ldots, f_r] :\AAA^m \dasharrow \PP^{r}$ be a semi-linear rational map such that 
\[f_i(\bx)= \begin{cases} x_i+p_i(\bx) & \text{ for } 1\le i \le s \\
q_i(\bx) & \text{ for } s+1 \le i \le r+1
\end{cases}\]
with degree of terminality $1\le s \le r+1$, and polynomials $p_i, q_i$ such that $\deg(q_i)\ge 2$ for $i=s+1,\ldots, r+1$ (note that $p_i=0$ or $\deg(p_i)\ge 2$). Let $\calc$ be a cluster corresponding to an irreducible component of the monomial ideal $M(f)$. Due to the linear terms we necessarily have $x_1,\ldots, x_s \in \calc$, so   
\[\calc=\{x_1,\ldots, x_s, x_{c_1},\ldots, x_{c_l}\}.\]
for some $c_1>s,\ldots, c_l >s$. We define the {\em semi-monomial ideal} 
\[M(f,\calc)=(x_1+p_1(\bx),\ldots, x_s+p_s(\bx),x_{c_1},\ldots, x_{c_l}),\]
by substituting $x_i$ with the entire coordinate function in $\calc$ if $x_i$ appears as a linear factor. At this stage, instead of  $\blow_{\calc}\AAA^m$ we will take $\blow_{M(f,\calc)}\AAA^m$, and study the blow-up rational map $\tilde{f}$ on the various affine charts. Note that 
\begin{multline}\nonumber
\blow_{M(f,\calc)}\AAA^m = \{((x_1,\ldots, x_m),[y_1:\ldots, y_s:z_1:\ldots :z_l]) \in \AAA^m \times \PP^{s+l-1}:\\
(x_i+p_i)y_j=(x_j+p_j)y_i, x_{c_i}z_j=x_{c_j}z_i, x_{c_j}y_i=(x_i+p_i)z_j\}
\end{multline}
 is not necessarily smooth, but nevertheless, the blown-up rational map $\tilde{f}: \blow_{M(f,\calc)}\AAA^m \dasharrow \PP^r$ is well-defined on the affine charts $U(y_i)=\{y_i \neq 0\}$, and more precisely, we will use the following 
 \begin{lemma} The $i$th projective coordinate of any point in $\tilde{f}|_{U(y_i)}: U(y_i) \to \PP^r$ is nonzero. 
 \end{lemma}
\begin{remark} As an immediate corollary we will see that the torus fixed-points in $f(U(y_i)) \subset \PP(\wedge^k \symdot)$ in our setup will have nonzero $i$th coordinate, and hence the image of most of these affine charts will not contain the distinguished torus fixed-point of Residue Vanishing Theorem, and hence the contribution of these affine charts to the Thom polynomial formula is zero. 
\end{remark}
Now let's focus on the other affine charts $U(z_i) \subset \blow_{M(f,\calc)}\AAA^m$ and the blown-up maps $\tilde{f}_{z_i}:=\tilde{f}|_{U(z_i)}:U(z_i) \to \PP^{r}$. Fix $i=1$, then 
\[U(z_1)=\{((x_1,\ldots, x_m),(y_1,\ldots, y_s,z_2, \ldots, z_l))\in \AAA^m \times \AAA^{s+l-1} :x_{c_i}=x_{c_1}z_i, (x_i+p_i)=y_ix_{c_1}\}\]
The presentation of $\tilde{f}_{z_1}$ is not unique in these affine coordinates:
\begin{equation}
\tf_{z_1}=[x_1+p_1(\bx),\ldots x_s+p_s(\bx),q_{s+1}(\bx),\ldots, q_{r+1}(\bx)]=[y_1,\ldots, y_s, \hat{q}_{s+1}(\bx),\ldots, \hat{q}_{r+1}(\bx)]
\end{equation} 
where $\hat{q}_i(\bx)=\frac{1}{x_{c_1}}q_i(\hat{x}_1,\ldots, \hat{x}_m)$ with 
\[\hat{x}_i=\begin{cases} x_i & \text{ if } x_i \notin \{c_2,\ldots, c_l\} \\
x_{c_1}x_i & \text{ if } i \in \{c_2,\ldots, c_l\}
\end{cases}\]
Since $\calc$ covers all monomials of $f$, at least one term of each monomial in $q_i(x_1,\ldots, x_s, \hat{x}_{s+1},\ldots, \hat{x}_m)$ is divisible by $x_{c_1}$, and hence $\hat{q}_i(\bx)$ is a polynomial in $x_1,\ldots, x_m$. Now we would like to write $\hat{q}_i(\x)$ as a polynomial in $x_1+p_1,\ldots, x_s+p_s, x_{s+1},\ldots, x_m$, but this is, of course, not necessarily possible, as the non-linear change of coordinates 
\begin{equation}\label{tau} 
\xi_i= \begin{cases} x_i+p_i(\bx) & 1\le i \le s\\
x_i & s+1 \le i \le m 
\end{cases}
\end{equation}
is not globally invertible on $\AAA^m$. However, it is invertible in any finite jet space, that is, up to high order polynomials. Assume $p_1 \neq 0$ and hence $\deg(p_1) \ge 2$, then for any positive integer $d$ one can iteratively write 
\begin{align}\label{inverse}
x_1=\xi_1-p_1(x_1,\ldots, x_s,x_{s+1},\ldots, x_m)=\xi_1-p_1(\xi_1-p_1(\bx),\ldots, \xi_s-p_s(\bx),\xi_{s+1},\ldots, \xi_m)=\ldots =\\ \nonumber
=\xi_1+p_1^{\le d}(\xi_1,\ldots, \xi_m)+q_1^{\ge d}(\xi_1,\ldots, \xi_m, x_1,\ldots, x_m))
\end{align}
for some polynomials $p_1^{\le d},q_1^{\ge d}$ with $\deg(q_1^{\ge d})\ge d$. The {\em smoothening trick} is the following: we replace $\tilde{f}_{z_i}$ with its $d$-jet for sufficiently large $d=d(k)$ depending only on $k$. To do so, we define 
\begin{equation}
x_i^{\le d}= \begin{cases} \xi_i+p_1^{\le d}(\xi_1,\ldots, \xi_m) & 1\le i \le s\\
\xi_i & s+1 \le i \le m 
\end{cases}
\end{equation}
and think of $x_i^{\le d}$ as {\em order $d$ approximation} of $x_i$. We aim to replace the original $f$ with an approximation $f^{\le d}$, such that the latter is not a semi-linear map, and hence the monomial blow-up process proceeds. Take the diagram
\begin{equation}\label{strategy2b}
\xymatrix{& \Spec(\CC[\xi_1,\ldots, \xi_m]) \ar@{.>}[rd]_-{f^{\le d}}  \ar@<0.5ex>[ld]^-{\tau^{-1}_{\le d}}& \\
 \Spec(\CC[x_1,\ldots, x_m]) \ar@<0.5ex>[ru]^-{\tau} \ar@{.>}[rr]_{f} & & \PP^r }
\end{equation}
where 
\begin{itemize}

\item $\tau(x_1,\ldots, x_m)=(x_1+p_1(\bx),\ldots, x_s+p_s(\bx),x_{s+1},\ldots, x_m)$
\item $\tau^{-1}_{\le d}(\xi_1,\ldots, \xi_m)=(\xi_1+p_1^{\le d}(\mathbf{\xi}),\ldots ,\xi_s+p_s^{\le d}(\mathbf{\xi}),\xi_{s+1},\ldots, \xi_m)$
\item $f=[f_1:\ldots:f_{r+1}]$
\item $f^{\le d}(\xi)=f \circ \tau_{\le d}^{-1}=[\xi_1,\ldots, \xi_s, q_{s+1}(\bx^{\le d}),\ldots, q_{r+1}(\bx^{\le d})]$. 
\end{itemize}
Note that $f^{\le d}$ is not semi-linear map: any linear monomial is a coordinate function itself, without quadratic or higher order terms.
On the other hand, the difference 
\[f-f^{\le d} \circ \tau=[0,0,\ldots, 0,q_{s+1}(\bx)-q_{s+1}(\tau_{\le d}^{-1}\circ \tau(\bx)),\ldots, q_{r+1}(\bx)-q_{r+1}(\tau_{\le d}^{-1}\circ \tau(\bx))]\] 
is a polynomial rational map of degree $\ge d$. More precisely, $\bx=(\tau_{\le d}^{-1}\circ \tau)(\bx)$ up to degree $d$, but even a stronger property holds. Let $m=x_1^{a_1} \ldots x_m^{a_m}$ be a monomial of $p_1$ with $a_1+\ldots +a_m \ge 2$. Then 
\begin{align*}
\xi_1-m(x_1,\ldots, x_s,x_{s+1},\ldots, x_m)=\xi_1-m(\xi_1-p_1(\bx),\ldots, \xi_s-p_s(\bx),\xi_{s+1},\ldots, \xi_m)=\\
=\xi_1-\xi_1^{a_1} \ldots \xi_m^{a_m}+\sum_{b_1,\ldots ,b_s} p_1(\bx)^{b_1}\ldots p_s(\bx)^{b_s}\prod_{i=1}^s\xi_i^{a_i-b_i}\prod_{i=s+1}^m\xi_i^{a_i}
\end{align*} 
What happens with the monomial $m$ during these substitutions? We see that all appearing monomials in $\xi_1,\ldots, \xi_m$ are obtained by iterating the following basic step:
\begin{equation}
m=\xi_1^{a_1}\ldots \xi_m^{a_m} \leadsto m(i,m')=\xi_1^{a_1}\ldots \xi_{i-1}^{a_{i-1}} \xi_i^{a_i-1}m'  \xi_{i+1}^{a_{i+1}} \ldots  \xi_m^{a_m} \text{ for some } m' \in M(p_i) \tag{Basic Step}
\end{equation}
where $M(p)$ stands for the set of monomials of the polynomial $p$. In short, for any $1\le i\le i$  we can substitute a factor $\xi_i$ with any monomial in $p_i$. For a fixed initial monomial $m$ and a positive integer $\ell \in \ZZ_+$ let $\mathrm{Mon}^\ell(m)$ denote the set of all monomials obtained from $m$ by $\ell$ iteration of the Basic Step. It is clear that the generated monomial ideal $M(\mathrm{Mon}^\ell(m))$ stabilises, that is, there is an $L=L(m)$ depending on $m$ such that 
\[M(\mathrm{Mon}^\ell(m))=M(\mathrm{Mon}^L(m)) \text{ for } \ell>L\]   
Taking $L=\max_{m \in M(p_1,\ldots, p_s)} L(m)$ to be the largest for all monomials appearing in the $p_i's$, we arrive to the following 
\begin{lemma} Let $\tau^*: \CC[\xi_1,\ldots, \xi_m] \to \CC[x_1,\ldots x_m]$ be the polynomial change defined in \eqref{tau}:
\begin{align*} \tau^*(\xi_i)= \begin{cases} x_i+p_i(\bx) & 1\le i \le s\\
x_i & s+1 \le i \le m 
\end{cases}\end{align*}
with $\deg(p_i)\ge 2$, and let 
\[\tau_{\le d}^*(x_i)=\xi_i+p_i^{\le d}(\xi_1,\ldots, \xi_m) \ \ \ 1\le i\le s\]
be the degree-d jet of the inverse transformation. Then there is $L>0$ such that 
\[M(\calm(p_1^{\le d},\ldots, p_s^{\le d})) = M(\calm(p_1^{\le L},\ldots, p_s^{\le L})) \text{ for any } d \ge L.\]
\end{lemma}
By applying the morphism $\tau^*$ we arrive at  
\begin{corollary} There is an $L>0$ such that for $d\ge L$ the monomials of $f^{\le d} \circ \tau$ generate the monomial ideal of $f$, that is: 
\[M(\calm(f))=M(\calm(f^{\le d} \circ \tau))\] 
\end{corollary}
For a polynomial map $f$ let $f_{\le d}$ denote its $d$-jet, that is, the degree $\le d$ part. The crucial properties of the two rational maps $f$ and $f^{\le d} \circ \tau$ are the following: for $d>\max(\deg(f),L)$ they satisfy
\begin{enumerate}
\item $f=(f^{\le d} \circ \tau)_{\le \deg(f)}$
\item $M(\calm(f))=M(\calm(f^{\le d} \circ \tau))$, which means that all monomials of $f^{\le d} \circ \tau$ of degree bigger than $\deg(f)$ are divisible by some monomial of $f$.
\end{enumerate}  
We will say that $f$ and $f^{\le d} \circ \tau$ are {\em $d$-equivalent} and we use the same terminology for any two polynomial rational maps satisfying (1) and (2) above. 

Recall that we focus on affine charts $U(z_i) \subset \blow_{M(f,\calc)}\AAA^m$ and the blown-up maps $\tilde{f}_{z_i}:=\tilde{f}|_{U(z_i)}:U(z_i) \to \PP^{r}$. Fix $i=1$, then 
\[U(z_1)=\{((x_1,\ldots, x_m),(y_1,\ldots, y_s,z_2, \ldots, z_l))\in \AAA^m \times \AAA^{s+l-1} :x_{c_i}=x_{c_1}z_i, (x_i+p_i)=y_ix_{c_1}\}\]
and 
\[\tf_{z_1}=[x_1+p_1(\bx),\ldots x_s+p_s(\bx),q_{s+1}(\bx),\ldots, q_{r+1}(\bx)]=[y_1,\ldots, y_s, \hat{q}_{s+1}(\bx),\ldots, \hat{q}_{r+1}(\bx)]\]
where $\hat{q}_i(\bx)=\frac{1}{x_{c_1}}q_i(\hat{x}_1,\ldots, \hat{x}_m)$. Recall the problem was that $U(z_1)$ is {\em not necessarily smooth, hence we want to replace it with a smooth blow-up at a coordinate subspace.} Blowing up commutes with flat base change, so the map $\tau$ induces a morphism 
\[\tilde{\tau}: \blow_{(x_1+p_1(\bx),\ldots, x_s+p_s(\bx),x_{c_1},\ldots, x_{c_l})}\AAA^m \to \blow_{((\xi_1,\ldots, \xi_s,\xi_{c_1},\ldots, \xi_{c_l}))}\AAA^m\]
because the preimage of the subspace whose ideal is $(\xi_1,\ldots, \xi_s,\xi_{c_1},\ldots, \xi_{c_l})$ has ideal $M(f,\calc)$:
\[\tau^{-1}(\Spec(\xi_1,\ldots, \xi_s,\xi_{c_1},\ldots, \xi_{c_l}))=\Spec((x_1+p_1(\bx),\ldots, x_s+p_s(\bx),x_{c_1},\ldots, x_{c_l}))=\Spec(M(f,\calc))\] 
This restricts to $\tilde{\tau}|_{U(z_1)}:U(z_1) \to U(u_1)$ where 
\[U(u_1)=\{((\xi_1,\ldots, \xi_m),(t_1,\ldots, t_s,u_2, \ldots, u_l))\in \AAA^m \times \AAA^{s+l-1} :\xi_{c_i}=\xi_{c_1}u_i, \xi_i=t_i\xi_{c_1}\}\simeq \Spec(\CC[\xi_{c_1},\bt,\bu])\]
and we obtain the diagram

\begin{equation}\label{strategy3}
\xymatrix@C+1pc{
U(z_1)  \ar[d]^{\pi} \ar[r]^-{\tilde{\tau}} \ar@{.>}[rrd]^-{\tilde{f}_{z_1}}&  U(u_1)  \ar[d]^{\pi} \ar@{.>}[rd]^-{\tilde{f}_{u_1}^{\le d}}& \\
 \Spec(\CC[x_1,\ldots, x_m]) \ar[r]^-{\tau} \ar@/_1pc/[rr]_-{f} & \Spec(\CC[\xi_1,\ldots, \xi_m]) \ar@{.>}[r]_-{f^{\le d}} & \PP^r }
\end{equation}
Here 
\begin{itemize}
\item $\tilde{\tau}(x_1,\ldots, x_m, y_1,\ldots, y_s, z_1,\ldots, z_l)=(x_1+p_1(\bx), \ldots, x_s+p_s(\bx),x_{s+1},\ldots, x_m, t_1,\ldots, t_s, u_1,\ldots, u_l)$.
\item $\tf_{u_1}^{\le d}(\xi,\bt,\bu)=[\xi_1,\ldots, \xi_s, q_{s+1}(\bx^{\le d}),\ldots, q_{r+1}(\bx^{\le d})]=[t_1,\ldots, t_s, \hat{q}_{s+1}(\bx^{\le d}),\ldots, \hat{q}_{r+1}(\bx^{\le d})]$
\item $\tf_{u_1}^{\le d} \circ \tilde{\tau}$ and $\tf_{z_1}^{\le d}$ are also $d$-equivalent, indeed:
\begin{align*}
\tf_{z_1}^{\le d}-\tf_{u_1}^{\le d} \circ \tilde{\tau}=[0,\ldots, 0, \hat{q}(\bx)-\hat{q}_{s+1}(\tau_{\le d}^{-1}\circ \tau(\bx)),\ldots, \hat{q}(\bx)-\hat{q}_{r+1}(\tau_{\le d}^{-1}\circ \tau(\bx))]
\end{align*}
\end{itemize}
The key point is that $\tf_{u_1}^{\le d}$ is not semi-linear: all coordinate functions are either linear or of degree at least $2$ in the affine coordinates $\xi_{c_1},\bt,\bu$ on $U(1)$. Hence we can proceed with a smooth monomial blow-up $\blow_{\calc}U(u_1)$ at a cluster $\calc \subset \{\xi_{c_1},t_1,\ldots, t_s, u_2, \ldots, u_l\}$ which 'covers' the monomial ideal of $\tf_{u_1}^{\le d}$ and apply base change to obtain the blow-up $\blow_{\tilde{\tau}^*\calc}U(z_1)$ where $\tilde{\tau}^*(\calc) \subset \{x_{c_1},y_1,\ldots, y_s,z_2,\ldots, z_l\}$. Due to $d$-equivalence, $\tilde{\tau}^*(\calc)$ corresponds to a component in the primary decomposition of $M(\tf_{z_1})=M(\tf_{u_1}^{\le d} \circ \tilde{\tau})$. 
Then $\tf_{(2)}^{\le d} \circ \tilde{\tau}$ and $\tf_{(2)}$ are $d$-equivalent again.
\begin{equation}\label{strategy2}
\xymatrix@C+1pc{\blow_{\tilde{\tau}^*\calc}U(z_1) \supset U  \ar[d] \ar[r]^-{\tilde{\tau}} \ar@{.>}[rrdd]_-{\tilde{f}_{(2)}}& U' \subset \blow_{\calc}U(u_1) \ar[d] \ar@{.>}[rdd]^-{\tilde{f}^{\le d}_{(2)}}& \\
U(z_1)  \ar[d] \ar[r]^-{\tilde{\tau}} &  U(u_1)  \ar[d] \ar@{.>}[rd]^-{\tilde{f}_{u_1}^{\le d}}& \\
 \Spec(\CC[x_1,\ldots, x_m]) \ar[r]^-{\tau} \ar@/_1pc/[rr]_-{f} & \Spec(\CC[\xi_1,\ldots, \xi_m]) \ar@{.>}[r]_-{f^{\le d}} & \PP^r }
\end{equation}

After repeated use of monomial blow-ups we arrive to a semi-linear rational map, where we use the smoothening trick again, and then we continue with monomial blow-ups again. At each stage we obtain two $d$-equivalent rational maps $\tf_{(i)}^{\le d}\circ \tilde{\tau}$ and $\tf_{(i)}$, hence the number of components in their monomial ideals are equal. The process will eventually end after $R$ steps, where $\tf_{(R)}^{\le d}$ is well-defined, that is, one of its coordinate functions is constant $1$. Due to $d$-equivalence, the same coordinate function is constant $1$ of the composition $\tf_{(R)}^{\le d}\circ \tilde{\tau}$. The sketch of the blowing-up process is the following.    
\begin{equation}\label{strategy4}
\xymatrix@R+1pc@C-1pc{U_0^{r_1+\ldots +r_l}   \ar[d]|{\hole \vdots \hole} \ar@/^5pc/[rrrddd]^-{\tf_{(r_1+\ldots +r_l)}} \ar[r]^-{\tilde{\tau}_1} &  U_1^{r_2+\ldots +r_l} \ar[d]|{\hole \vdots \hole} \ar[r]^-{\tilde{\tau}_2} &  U_2^{r_3+\ldots +r_l} \ar[d]|{\hole \vdots \hole} \ar[r]|{\ldots} & U_l^0 \ar[ddd]^-{\tf_{r_1+\ldots +r_l}^{\le d}}&  \\
U_0^{r_1+r_2} \ar[d]|{\hole \vdots \hole} \ar[r]^-{\tilde{\tau}_1} \ar@{.>}[rrrdd]^-{\tf_{(r_1+r_2)}}&  U_1^{r_2} \ar[d]|{\hole \vdots \hole} \ar[r]^-{\tilde{\tau}_2} &  U_2^0 \ar@{.>}[rdd]^-{\tf_{r_1+r_2}^{\le d}}&  \\
U_0^{r_1}  \ar[d]|{\hole \vdots \hole}  \ar@{.>}[rrrd]^-{\tf_{(r_1)}} \ar[r]^-{\tilde{\tau}_1} &  U_1^0 \ar@{.>}[rrd]^-{\tf_{(r_1)}^{\le d}}& \\
 U_0^0 \ar@{.>}[rrr]_-{f} & & & \PP^r=\PP(\symdot) }
\end{equation}
All maps in this diagram are $T \times \CC^*$ equivariant, and therefore the following holds.
\begin{lemma}\label{fixedpoints} Let $\bgg=[e_1\wedge \ldots, \wedge e_k]\in \PP(\wedge^k \Sym^{\le k})$ denote the $T\times \CC^*$-fixed point corresponding to the coordinate $\pi_\dist=((1),(2),\ldots, (k))$. Let $R=r_1+\ldots +r_l$ denote the total number of blow-ups in the process described above, and introduce the shorthand notation $\tilde{\tau}=\tilde{\tau}_l \circ \ldots \circ \tilde{\tau}_1$. Then 
\[\tf_{(R)}^{-1}(\ff_\dist)=(\tf_{(R)}^{\le d} \circ \tilde{\tau})^{-1}(\ff_\dist).\]
\end{lemma}
\begin{proof}
Assume $x \in \tf_{(R)}^{-1}([e_1\wedge \ldots, \wedge e_k])$. Then $(\tf_{(R)})_{\pi_\dist}(x)=1$, hence $(\tf_{(R)}^{\le d} \circ \tilde{\tau})_{\pi_\dist}(x)=1$. But $x$ is $T \times \CC^*$-fixed, and  $\tf_{(R)}^{\le d} \circ \tilde{\tau}$ is equivariant, hence $(\tf_{(R)}^{\le d} \circ \tilde{\tau})_{\pi_\dist}(x)$ is a fixed point whose $\pi_{\dist}$-coordinate is nonzero. It must then be equal to $[e_1\wedge \ldots, \wedge e_k]$.
\end{proof}

\subsection{The blowing up algorithm} After this general introduction, we apply the monomial blow-up process described in the previous section for the rational map $\phi: U(0) \dasharrow \PP^M, \phi(t,\b_{ij})=[p_{\pi}(t,\b_{i,j}):\pi \in \Pi_{k,n}]$. We now summarize the blowing-up procedure. 

\noindent \textbf{\underline{The algorithm}}  
\begin{enumerate} 
\item \textbf{First blow-up} Take the rational map $\phi: U_0^0=\Spec(\CC[t,\b_{ij}]) \dasharrow \PP^N$ defined as $\phi(t,\b_{ij})=[p_{\pi}(t,\b_{i,j}):\pi \in \Pi_{k,n}]$, and the corresponding monomial map $\phi^\mon$ with monomial ideal $M(\phi)$. Take primary decomposition $Z(M(\phi))= m_1Z(\calc_1) \cup m_2Z(\calc_2) \cup \ldots \cup m_{l}Z(\calc_{l})$
into irreducible coordinate subspaces determined by the clusters $\calc_i \subset \{t,\b_{ij}\}$. Pick a cluster $\calc \in \{\calc_1,\ldots, \calc_l\}$ and a variable $c \in \calc$ such that
\begin{enumerate}
\item $\calc$ is $\diff_k=\CC^* \rtimes U$-invariant.
\item $c$ has minimal $\CC^*$-weight. 
\end{enumerate} 
Blow up $U_0^0$ along $\calc$, and let $\tilde{\phi}:\blow_{\calc}U_0^0 \dasharrow \PP^N$ denote the blow-up of $\phi$. The minimal weight space on $\blow_{\calc}U_0^0$ is covered by those affine charts which correspond to minimal-weight elements of $\calc$. Hence we need to study only these affine charts, and in particular, the restriction of $\tilde{\phi}$ to the affine chart $\AAA^m \simeq \blow_{\calc,c}U_0^0$. We will denote this affine chart, $U_0^1$, and we keep the notation $\mathbf{B}=\{t,\b_{ij}\}$ for the coordinates on $U_0^1$ but record the change of their weight as explained in Step (2)

\item \textbf{Iteration of blow-ups} Assume we iterated this process $r_1$ times, so that we get a sequence of blow ups
\[\xymatrix@C+1pc{U_0^{r_1} \ar[r]^-{\calc_{r_1},c_{r_1}} & U_0^{r_1-1} \ar[r]^-{\calc_{r_1-1},c_{r_1-1}} & \ldots \ar[r]^-{\calc_2,c_2} & U_0^1 \ar[r]^-{\calc_1,c_1} & U_0^0}\]
which is also encoded in the short form  $\xymatrix{\calc_1 \ar[r]^{c_1} & \calc_2 \ar[r]^{c_2} & \ldots \ar[r]^{c_{r_1-1}} & \calc_{r_1}}$, 
where $\calc_i$ is a $\diff_k$-invariant cluster of affine coordinates on $U_0^i$ and $c_i \in \calc_i$ has minimal $\CC^*$-weight. This $c_i$ determines the affine chart we pick in the next step. The blow-up rational map is $\phi^{0}=\phi$ and $\phi^{{i}}: U_0^i \dasharrow \PP^N$. We keep the same notation $\mathbf{B}=\{t,\b_{ij}\}$ for the basis of $U_0^i$ for all $i$, so they transformed at step $i$ as follows: 
\[\pi^*_{\calc_i,c_i}\b=\begin{cases} \b & \b \notin \calc_i \setminus \{c_i\} \\ \b c_i & \b \in \calc_i \setminus \{c_i\} \end{cases}\]
and $\phi^{{i+1}}=\phi^{{i}} \circ \pi^*_{\calc_i,c_i}$. 

Let $T\subset B_k \subset \GL(k)$ be the maximal torus, the $T$-weights on $\CC^k$ are $z_1,\ldots, z_k$. Each variable of $\mathbf{B}$ at each level is endowed with a $T^n \times \CC^*$ weight in this path; the weight of $\b \in \mathbf{B}$ in $U_0^i$ is denoted by $\wt_{\b}^i$. The initial weights on $U_0^0$ are $\wt^0_t=z+z_1$ and $\wt_{ij}^0=z_i+(j-1)z-z_1$. The simple rule of forming these weights is the following:
\begin{equation}\label{weightrule}
\wt^{i+1}_{\b}=\begin{cases} \wt^{i}_\b & \text{ if } \b=c_i \text{ or } \b \notin \calc_{i}\\
                                        \wt^{i}_\b-\wt^i_{c_i} & \text{ if } \b \in \calc_{i+1}\setminus \{c_i\} \\
                                               \end{cases}
                                               \end{equation}
\item \textbf{Smoothening step} Assume that after $r_1$ monomial blow-ups we arrive to the terminal rational map $\phi^{{r_1}}: U_0^{r_1} \dasharrow \PP^N$. In this case we apply the smoothening trick explained in the previous section: we define a polynomial substitution $\tilde{\tau}: \AAA^m \simeq U_0^{r_1} \to U_1^0 \simeq \AAA^m$ and replace the semi-monomial blow-up of $U_0^{r_1}$ with a linear blow-up $U_1^1$ of $U_1^0$ at the cluster $\calc_{r_1} \subset \mathbf{B}$. We define the degree $d$ approximation $\tilde{\phi}_1^{\le d}: U_1^1 \dasharrow \PP^N$. Then we continue with regular monomial blow-ups of $U_1^1$ until we stuck with a terminal rational map. where we perform the smoothening trick again.

\item \textbf{Termination}  Our process terminates after $l$ smoothening steps and a total of $r_1+\ldots +r_l$ monomial blow-ups. We arrive to a morphisms $\tilde{\phi}: U_l^{r_1+\ldots +r_l} \to \PP(\wedge^k \symdot)$ and $\tilde{\phi}^{\le d}: U_l^{0} \to \PP(\wedge^k \symdot)$. 
\end{enumerate}
\noindent \underline{\textbf{The blow-up tree}}
At each stage of the blowing-up process we need to pick an affine chart, and proceed on that chart. Hence the algorithm is described by a rooted tree $\calt_k$, whose edges are directed from the root towards the leaves. Any vertex of such a tree has a unique ancestor and several descendants.
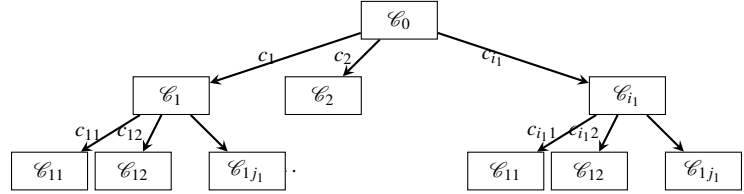
\begin{center}
{\tiny \begin{tikzpicture}[node distance=2cm]
\node (I) [boxes] {$\calc_0$};
\node (IIa) [boxes,below of=I,yshift=1cm,xshift=-3cm] {$\calc_{1}$};
\node (IIb) [boxes, below of=I,yshift=1cm,xshift=-1cm] {$\calc_{2}$};
\node (IIc) [boxes, below of=I,yshift=1cm,xshift=3cm] {$\calc_{i_1}$};

\node (IIIa) [boxes,below of=IIa,yshift=1cm,xshift=-1.6cm] {$\calc_{11}$};
\node (IIIb) [boxes, below of=IIa,yshift=1cm,xshift=-0.5cm] {$\calc_{12}$};
\node (dots) [right of =IIIb]{$\ldots$};
\node (IIIc) [boxes, below of=IIa,yshift=1cm,xshift=1cm] {$\calc_{1j_1}$};

\node (IIId) [boxes,below of=IIc,yshift=1cm,xshift=-1.6cm] {$\calc_{11}$};
\node (IIIe) [boxes, below of=IIc,yshift=1cm,xshift=-0.5cm] {$\calc_{12}$};
\node (IIIf) [boxes, below of=IIc,yshift=1cm,xshift=1cm] {$\calc_{1j_1}$};

\draw [arrow] (I) -- node[anchor=east] {$c_1$} (IIa);
\draw [arrow] (I) -- node[anchor=east] {$c_2$} (IIb);
\draw [arrow] (I) -- node[anchor=east] {$c_{i_1}$} (IIc);

\draw [arrow] (IIa) -- node[anchor=east] {$c_{11}$} (IIIa);
\draw [arrow] (IIa) -- node[anchor=east] {$c_{12}$} (IIIb);
\draw [arrow] (IIa) -- node[anchor=east] {$$} (IIIc);

\draw [arrow] (IIc) -- node[anchor=east] {$c_{i_11}$} (IIId);
\draw [arrow] (IIc) -- node[anchor=east] {$c_{i_12}$} (IIIe);
\draw [arrow] (IIc) -- node[anchor=east] {$$} (IIIf);

\end{tikzpicture}}
\end{center}

\begin{enumerate}
\item The root of $\calt_k$ is labeled by a cluster $\calc_0 \subset \{t,\b_{ij}\}$ of variables on $U_0=\Spec[t,\b_{ij}] \subset \tilde{\PP}=\blow_{[1:0,\ldots, 0]}\PP(\CC \oplus \Hom^{\ff}(\CC^k,\CC^n))$. This cluster corresponds to a maximal irreducible component in the primary decomposition of $M(\phi)$. 
\item Each edge from the root is labeled by an element of $\calc_0$ of minimal $\CC^*$-weight. This indicates the affine chart $\blow_{\calc_0,c_i}U(0)$. The restriction of $\tilde{\phi}$ to this affine chart is $\tilde{\phi}_{\calc,c_i}$.
\item The endpoint of the edge $c_i$ is labeled by a cluster $\calc_i \subset \{t,\b_{ij}\}$ corresponding to a maximal irreducible component of the monomial ideal $M(\tilde{\phi}_{\calc,c_i}$. We endow elements of $\calc_i$ with an integer upper index, which indicates the $\CC^*$-weight of that coordinate. 
\item In general, all vertices are labeled by clusters of $\{t,\b_{ij}\}$, and for an edge $\xymatrix{\calc \ar[r]^-{c} & \calc'}$ the edge is labeled by an element $c \in \calc$ of minimal $\CC^*$ weight, indicating the affine chart $\blow_{\calc,c}\AAA^n$ of the blow-up $\blow_{\calc}\AAA^n$ of the previous affine chart at $\calc$. In case of a smoothening step. we indicate the $x_i \to \xi_i$ coordinate change in the box. 
\item There are two types of leaves of $\calt_k$: 
\begin{enumerate}
\item $\xymatrix{\calc \ar[r]^-{c} & \boxed{\text{Contr.}}}$. These are the leaves where $\blow_{\calc,c}\AAA^m$ contains a torus fixed point mapped to $\bgg$, and hence these corresdpond to a term in the residue formula of the main theorem. The set of contributing leaves of $\calt_k$ is denoted by $\mathrm{Contr}(\calt_k)$.
\item $\xymatrix{\calc \ar[r]^-{c} & \boxed{\text{No contr.}}}$ This means that $\blow_{\calc,c}\AAA^m$ does not contain a torus fixed point mapped to $\bgg$, and hence the contribution of this affine chart is $0$ in the residue formula. Non-contributing affine charts are easy to identify: this happens if and only if there is a Plucker coordinate $\pi \in \Pi_{k,n}$ such that $(\tilde{\phi}_{\calc,c})_{\pi_\dist}=m \cdot m'$ for some monomial $m'$ of $(\tilde{\phi}_{\calc,c})_{\pi}$ and a non-constant polynomial $m$, where $\pi_\dist=((1),\ldots, (k))$ as before.
\end{enumerate}
\end{enumerate}
The leaves $U_0^{r_1+\ldots +r_l}$ form an affine cover of the master blown-up space $\mathrm{Jet}_k=\blow_{\calt_k}\PP[\CC \oplus J_k(1,n)]$. The maps glue together to a morphism $\tilde{\phi}: \mathrm{Jet}_k \to \PP[\wedge^k \symdot]$. On the other hand the smooth affine charts $U_l^0$ glue together to give a nonsingular blow-up $\mathrm{Jet}_k^{\le d}$ with a morphism $\tilde{\phi}^{\le d}: \mathrm{Jet}_k^{\le d} \to \PP[\wedge^k \symdot]$, and a morphism $\tilde{\tau}: \mathrm{Jet}_k \to \mathrm{Jet}_k^{\le d}$.

\begin{equation}\label{globalblowups}
\xymatrix{\mathrm{Jet}_k \ar[rd]^-{\tilde{\phi}} \ar[rr]^-{\tilde{\tau}}& & \mathrm{Jet}_k^{\le d}  \ar[ld]^-{\tilde{\phi}^{\le d}}    \\
 & \PP[\wedge^k \symdot]  &}
\end{equation}

\subsection{Integration on the blow-up}

The main goal of the blowing-up algorithm and the intermediate smoothening steps is to reduce integration over the possibly singular $\mathrm{Jet}_k$ to the nonsingular $\mathrm{Jet}_k^{\le d}$. This is possible by applying equivariant localisation and Lemma \ref{fixedpoints}.
\begin{proposition} Let $\a=\mathrm{Euler}(\cale \otimes \CC^m)=\Pi_{j=1}^m\Pi_{i=1}^k (\theta_j+z_i)$ be the Euler class in Theorem \ref{thm:thomtau}. Then 
\[\int_{\mathrm{Jet}_k/\!/\diff_k}\tilde{\phi}^*\a=\int_{\mathrm{Jet}_k/\!/\diff_k}(\tilde{\phi}^{\le d} \circ \tilde{\tau})^*\a=\int_{\mathrm{Jet}^{\le d}_k}(\tilde{\phi}^{\le d})^*\a\] 
\end{proposition} 
\begin{proof}
The second equality simply tells that $\int_{X}f^*\a=\int_{Y}\a$ for any rational map $f: X\to Y$. For the first equality we apply the Residue Vanishing Theorem and Lemma \ref{fixedpoints}:
\begin{multline}
\int_{\mathrm{Jet}_k/\!/\diff_k}\tilde{\phi}^*\a=\sum_{L \in \tilde{\phi}^{-1}(\ff_\dist)} \sires \frac{
(k-1)!z^{k-1}\prod_{m<l}(z_m-z_l) \alpha(z_1,\ldots ,z_k)}{
\Euler^L(\mathrm{Jet}_k)  \prod_{l=1}^k\prod_{i=1}^n(\lambda_i-z_l)} \,\dbz=\\ 
\sum_{L \in (\tilde{\phi}^{\le d}\circ \tilde{\tau})^{-1}(\ff_\dist)} \sires \frac{
(k-1)!z^{k-1}\prod_{m<l}(z_m-z_l) \alpha(z_1,\ldots ,z_k)}{
\Euler^L(\mathrm{Jet}_k)  \prod_{l=1}^k\prod_{i=1}^n(\lambda_i-z_l)} \,\dbz=\int_{\mathrm{Jet}_k/\!/\diff_k}(\tilde{\phi}^{\le d} \circ \tilde{\tau})^*\a
\end{multline}
Note that $\mathrm{Jet}_k$ is not necessarily smooth at $L$, but $\mathrm{Jet}_k \subset \PP^N$ is a projective variety, and hence $\Euler^L(\mathrm{Jet}_k)$ stands for the Euler-Rossmann class 
\[\Euler^L(\mathrm{Jet}_k)=\frac{\mathrm{Euler}^{T\times \CC^*}(\TT_L\PP^N)}{\emu_L[\mathrm{Jet}_k,\PP^N]}\]
\end{proof}
Since $\mathrm{Jet}^{\le d}_k$ is smooth, the equivariant Euler class of the tangent space at a fixed point corresponding to a leave $L$ is the product of the weights:
\[\Euler^L(\mathrm{Jet}_k^{\le d})=\prod_{\b \in \{t,\b_{ij}\}} \wt^L(\b)\]
where the weights $\wt^L(\b)$ are formed by the simple inductive rule \eqref{weightrule} followed along the path from the root to $L$ in $\calt_k$.  
This finishes the proof of the main theorem of this paper, which can be formulated as follows.
\begin{maintheorem}\label{maintheorem} For arbitrary integers $k\ll n \le m$, the Thom polynomial for the $A_k$-singularity with
$n$-dimensional source space and $m$-dimensional target space is given by the following iterated residue formula.
\[\Tp_k^{n,m}=\res_{z_1=\infty}\ldots \res_{z_k=\infty}\cdot \sum_{L \in \mathrm{Contr}(\calt_k)} \frac{(k-1)!z^{k-1} (z_1\ldots z_k)^{m-n}\prod_{i<j}(z_i-z_j)}{\prod_{\b \in \mathbf{B}} \wt^{L}(\b)} \prod_{i=1}^k c_{TM-TN}(1/z_i)d\bz\]
where we sum over the contributing leaves in the blow-up tree $\calt_{k}$.
\end{maintheorem}

\bibliographystyle{abbrv}
\bibliography{thom.bib}

\end{document}